\documentclass[10pt,a4paper]{article}
\usepackage{epsfig}
\usepackage{amssymb}
\usepackage{amsmath}
\usepackage{amsthm}
\usepackage{mathtools}
\usepackage{enumitem}
\usepackage{caption}
\usepackage{cite}
\usepackage{array,booktabs}
\usepackage[hyphens]{url}
\usepackage[hidelinks]{hyperref}
\hypersetup{breaklinks = true}
\usepackage[english]{babel}
\usepackage{amsfonts, verbatim}
\usepackage{amsmath}
\usepackage{color,tikz} 
\usepackage[section]{placeins} 
\usepackage{psfrag} 
\usepackage{multirow}
\usepackage{subfig}

\newcommand{\R}{\mathbb{R}}
\newcommand{\N}{{\mathbb{N}}}
\renewcommand{\d}{{\rm{d}}}
\newcommand{\Cc}[1]{\mathbf{C_{\rm c}^{#1}}}
\newcommand{\C}[1]{\mathbf{C^{#1}}}
\newtheorem{theorem}{Theorem}[section]
\newtheorem{remark}{Remark}[section]

\newtheorem{proposition}{Proposition}[section]

\newtheorem{definition}{Definition}[section]
\theoremstyle{definition}
\renewcommand{\L}[1]{\mathbf{L^{\pmb{#1}}}}
\newcommand{\Lloc}[1]{\mathbf{L^{\pmb{#1}}_{loc}}}
\newcommand{\sign}{\mathrm{sign}}
\newcommand{\BV}{\mathbf{BV}}
\newcommand{\tv}{\mathrm{TV}}
\newcommand{\supp}{\mathrm{supp}}
\newcommand{\caratt}[1]{{\displaystyle\chi_{\strut{\textstyle #1}}}}

\usepackage{calc}
\usepackage{xcolor,colortbl}

\title{An existence result for a constrained two-phase transition model with metastable phase for vehicular traffic}

\author{Mohamed Benyahia, Carlotta Donadello, Nikodem Dymski, Massimiliano D.\ Rosini}

\newcommand{\Addresses}{{
  \bigskip
  \footnotesize

  Mohamed Benyahia, \textsc{Gran Sasso Science Institute, Viale F. Crispi 7, 67100 L'Aquila, Italy}\par\nopagebreak
  \textit{E - mail address}: \texttt{benyahia.ramiz@gmail.com}

  \medskip

  Carlotta Donadello, \textsc{Laboratoire de math\'ematiques, CNRS UMR 6623, Universit\'e de Franche - Comt\'e, 16 route de Gray, 25030 Besan\c{c}on, France}\par\nopagebreak
  \textit{E - mail address}: \texttt{carlotta.donadello@univ-fcomte.fr}

  \medskip

  Nikodem Dymski, \textsc{Uniwersytet Marii Curie-Sk\l odowskiej, Plac Marii Curie-Sk\l odowskiej 1, 20 - 031 Lublin, Poland}, \textsc{Inria Sophia Antipolis-M\'editerran\'ee, Universit\'e C\^ote d'Azur, Inria, CNRS, LJAD, 06902 Sophia-Antipolis, France}\par\nopagebreak
  \textit{E - mail address}: \texttt{nikodem.dymski@inria.fr}

  \medskip

  Massimiliano D.\ Rosini, (Corresponding author), \textsc{Dipartimento di Matematica e Informatica, Universit\`a di Ferrara, Via Machiavelli 35, 44121 Ferrara, Italy}\par\nopagebreak
  \textit{E - mail address}: \texttt{rsnmsm@unife.it}
}}

\begin{document}
\allowdisplaybreaks

\maketitle

\begin{abstract}
In this paper we study a phase transition model for vehicular traffic flows.
Two phases are taken into account, according to whether the traffic is light or heavy. 
We assume that the two phases have a non-empty intersection, the so called \emph{metastable} phase.
The model is given by the Lighthill-Whitham-Richards model in the \emph{free-flow} phase and by the Aw-Rascle-Zhang model in the \emph{congested} phase.
In particular, we study the existence of solutions to Cauchy problems satisfying a local point constraint on the density flux.
We prove that if the constraint $F$ is higher than the minimal flux $f_{\rm c}^-$ of the metastable phase, then constrained Cauchy problems with initial data of bounded total variation admit globally defined solutions.
We also provide sufficient conditions on the initial data that guarantee the global existence of solutions also in the case $F < f_{\rm c}^-$.
These results are obtained by applying the wave-front tracking technique.

\end{abstract}

\maketitle

\section{Introduction}

In this paper we study one of the constrained phase transition models of hyperbolic conservation laws introduced in~\cite{edda-nikodem-mohamed}.
The application of such model is, for instance, the modelling of vehicular traffic along a road with pointlike inhomogeneities characterized by limited capacity, such as speed bumps, construction sites, tollbooths, etc.

The model considers two different phases corresponding to the \emph{congested phase} $\Omega_{\rm c}$ and the \emph{free-flow phase} $\Omega_{\rm f}$. 
The model is given by a $2\times 2$ system of conservation laws in the congested phase, coupled with a scalar conservation law in the free-flow phase. 
The coupling is achieved via \emph{phase transitions}, namely discontinuities between two states belonging to different phases and satisfying the Rankine-Hugoniot conditions.

The first two-phase model has been proposed by Colombo in~\cite{Colombo}.
The motivation stems from experimental data, according to which the density flux represented in the fundamental diagram is one-dimensional for high velocities, while it covers a two-dimensional domain for low velocities, see~\cite[Figure 1.1]{Colombo}. 
For this reason, it is reasonable to describe the dynamics in the congested regime with a $2\times 2$ system of conservation laws and those in the free regime with a scalar conservation law. 

Later, Goatin proposed in~\cite{goatin2006aw} a two-phase model obtained by coupling the ARZ model by Aw, Rascle and Zhang~\cite{AwRascle, Zhang} for the congested phase $\Omega_{\rm c}$, with the LWR model by Lighthill, Whitham and Richards~\cite{LighthillWhitham, Richards} for the free-flow phase $\Omega_{\rm f}$. 
We recall that this model has been recently generalized in~\cite{BenyahiaRosini01}.

Both the models introduced in~\cite{Colombo} and~\cite{goatin2006aw} assume that $\Omega_{\rm c} \cap \Omega_{\rm f} = \emptyset$.
The first two-phase model that considers a metastable phase $\Omega_{\rm c} \cap \Omega_{\rm f}\neq\emptyset$ has been introduced in~\cite{Blandin}.
We also recall that, differently from~\cite{BenyahiaRosini01, goatin2006aw}, for the models in~\cite{Colombo, Blandin} the density flux function vanishes at a maximal density, whose inverse corresponds to the average length of the vehicles.
Here we consider the case $\Omega_{\rm c} \cap \Omega_{\rm f}\neq\emptyset$.
For this reason, in order to ensure the well-posedness of the Cauchy problems, see \cite[Remark 2]{Colombo}, we also assume that $\Omega_{\rm f}$ is characterized by a unique value of the velocity, $V$.
At last, we consider an heterogeneous traffic with vehicles having different lengths and allow the density flux function to vanish at different densities.

These two-phase models have been recently generalized in~\cite{BenyahiaRosini02, edda-nikodem-mohamed} by considering Riemann problems, namely Cauchy problems for piecewise constant initial data with a single jump, coupled with a constraint on the density flux, so that at the interface $x = 0$ the density flux of the solution must be lower than a given constant quantity $F$.
This condition is referred to as \emph{unilateral point constraint} and can be thought of as a pointwise bottleneck at $x = 0$ that hinders the density flow, see~\cite{Rosinibook} and the references therein. 
In vehicular traffic, a point constraint accounts for inhomogeneities of the road and models, for instance, the presence of a toll gate across which the flow of the vehicles cannot exceed its capacity $F$.

In the case in which no constraint conditions are enforced, existence results for the Cauchy problems for the above mentioned two-phase transition models have already been established, see~\cite{Blandin, BenyahiaRosini01, ColomboGoatinPriuli, goatin2006aw}.
In the present paper, we focus on the constrained version proposed in~\cite{edda-nikodem-mohamed} for the model introduced in~\cite{BenyahiaRosini01} and prove an existence result for constrained Cauchy problems.
More precisely, we use the Riemann solvers established in~\cite{BenyahiaRosini01} and~\cite{edda-nikodem-mohamed} in a wave-front tracking scheme and prove that the obtained approximate solution $u_n$ converges (up to a subsequence) to a globally defined solution of the constrained Cauchy problem with general $\BV$-initial data, at least in the case $F \geq f_{\rm c}^-$, where the threshold value $f_{\rm c}^-$ is the minimal density flux of the metastable phase, see \figurename~\ref{f:notations}.
At last, in the case $F < f_{\rm c}^-$ we give sufficient conditions on the initial data that ensure the convergence of $u_n$ to a globally defined solution of the constrained Cauchy problem.

The paper is organized as follows.
In the next section we introduce the notations used throughout the paper, the model, the definitions of solutions to the unconstrained and constrained Cauchy problems, the main result in Theorem~\ref{t:mainF} and at last the Riemann solvers for the unconstrained and constrained Riemann problems.
In Section~\ref{s:exPT} we apply the model to reproduce the traffic across a toll gate. 
Finally, in the last section we defer the technical proofs.

\section{Notations, definitions and main result}

In this section we state the main assumptions on the parameters, collect useful notations, see \figurename~\ref{f:notations}, give the definition of solutions, state the main result in Theorem~\ref{t:mainF} and at last introduce the Riemann solvers.

\begin{figure}[!ht]
\begin{center}
\resizebox{\textwidth}{!}{
\def\ratio{1}
\def\pic{50mm}
\begin{tikzpicture}[every node/.style={anchor=south west,inner sep=0pt},x=1mm/\ratio, y=1mm/\ratio]
\node at (4,4) {\includegraphics[height=\pic]{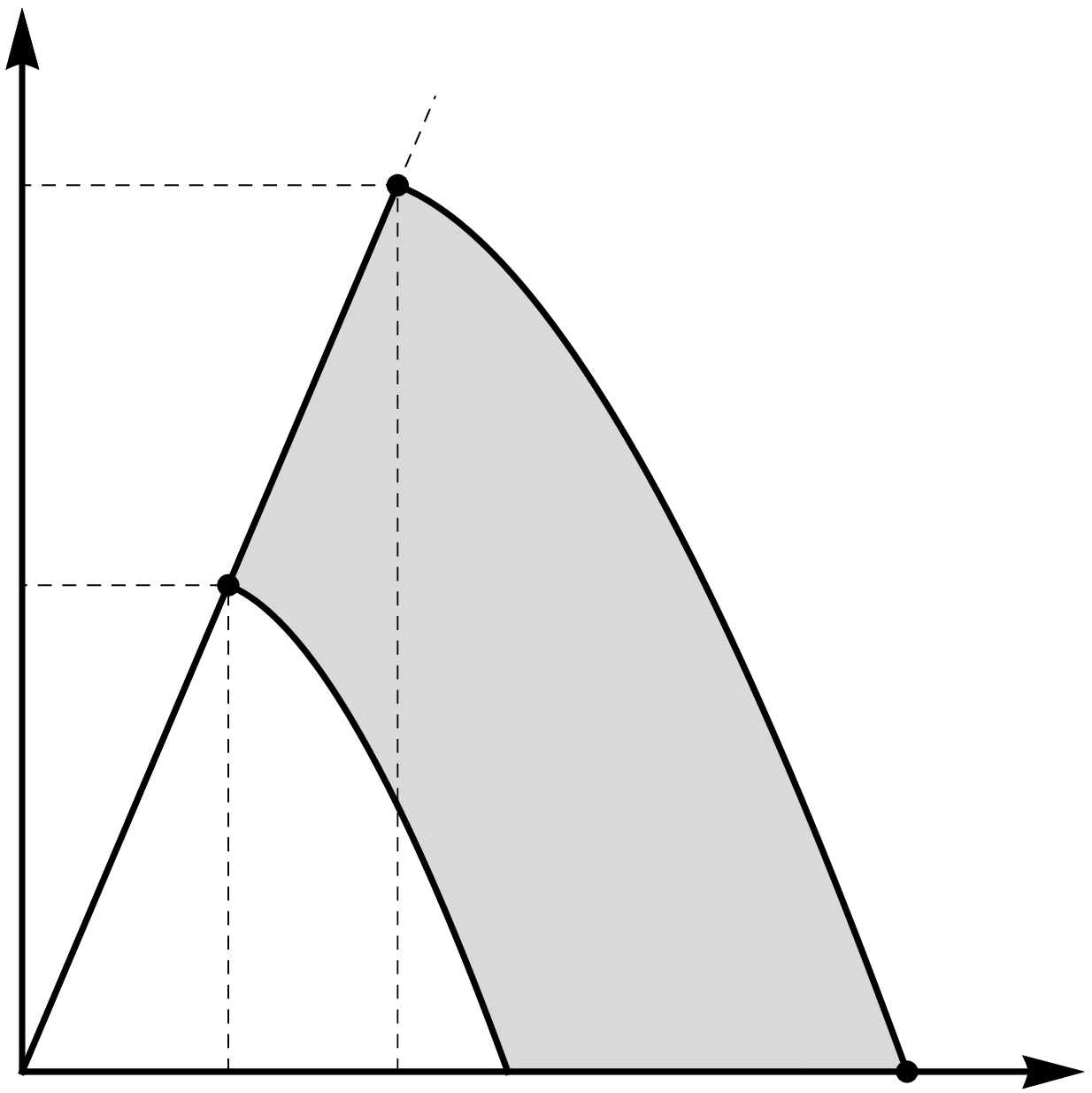}};
\node at (52,0) {$\rho$};
\node at (44,0) {$R$};
\node at (11.5,0) {$\rho^-$};
\node at (21,0) {$\rho^+$};
\node at (0,25) {$f_{\rm c}^-$};
\node at (0,43) {$f_{\rm c}^+$};
\node at (0,50) {$f$};
\node at (6,18) {$\Omega_{\rm f}^-$};
\node at (12,33) {$\Omega_{\rm f}^+$};
\node at (29,18) {$\Omega_{\rm c}$};
\node at (25,50) {$V$};
\end{tikzpicture}
\quad
\begin{tikzpicture}[every node/.style={anchor=south west,inner sep=0pt},x=1mm/\ratio, y=1mm/\ratio]
\node at (4,4) {\includegraphics[height=\pic]{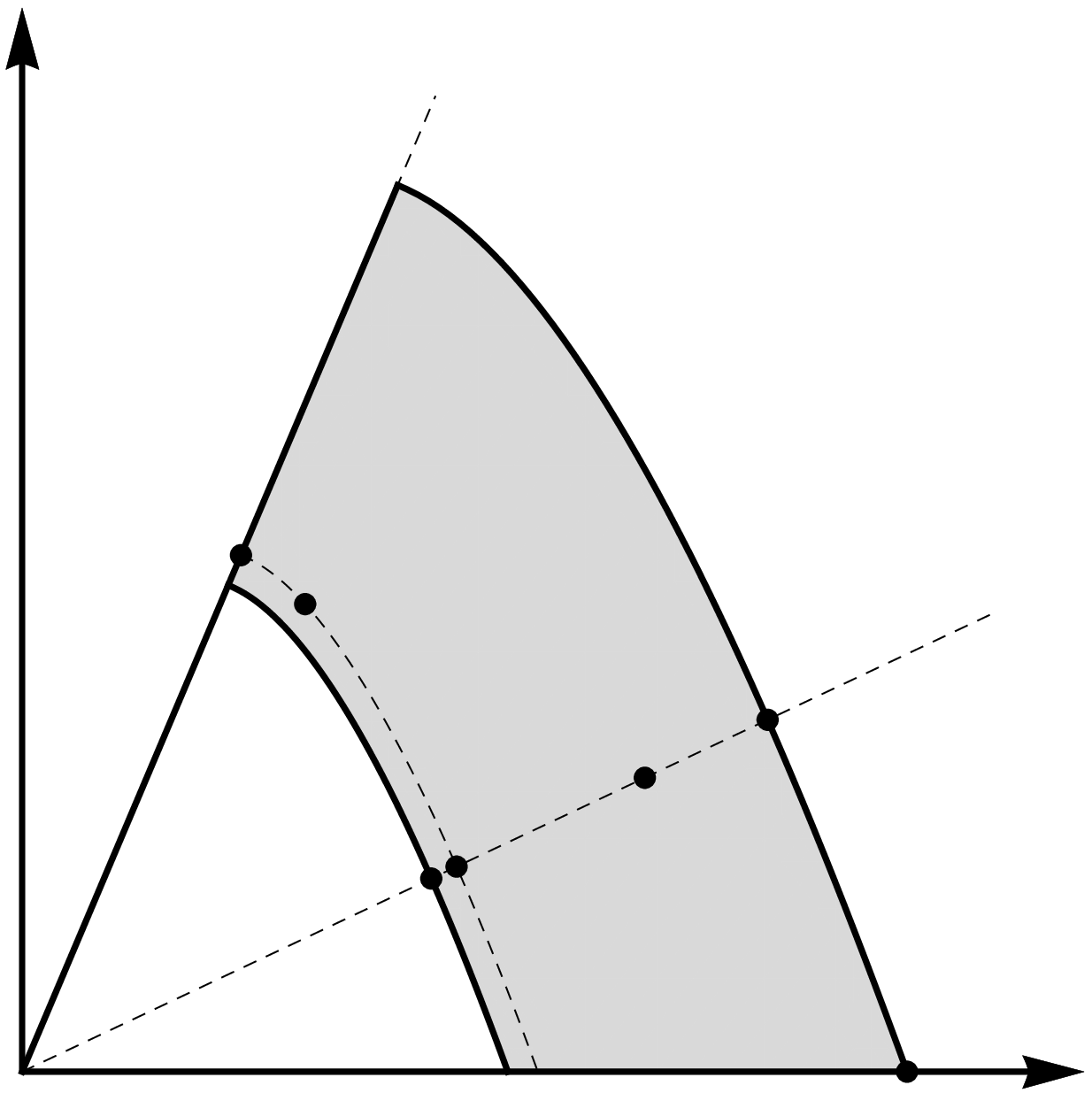}};
\node at (52,0) {$\rho$};
\node at (44,0) {$R$};
\node at (0,50) {$f$};
\node at (25,50) {$V$};
\node at (26.5,11.5) {$\mathtt{u}_*(u_\ell,u_r)$};
\node at (29,18.5) {$u_r$};
\node at (11,13) {$\mathtt{v}^-(u_r)$};
\node at (41,18) {$\mathtt{v}^+(u_r)$};
\node at (18,27) {$u_\ell$};
\node at (6,29) {$\omega(u_\ell)$};
\end{tikzpicture}
\quad
\begin{tikzpicture}[every node/.style={anchor=south west,inner sep=0pt},x=1mm/\ratio, y=1mm/\ratio]
\node at (4,4) {\includegraphics[height=\pic]{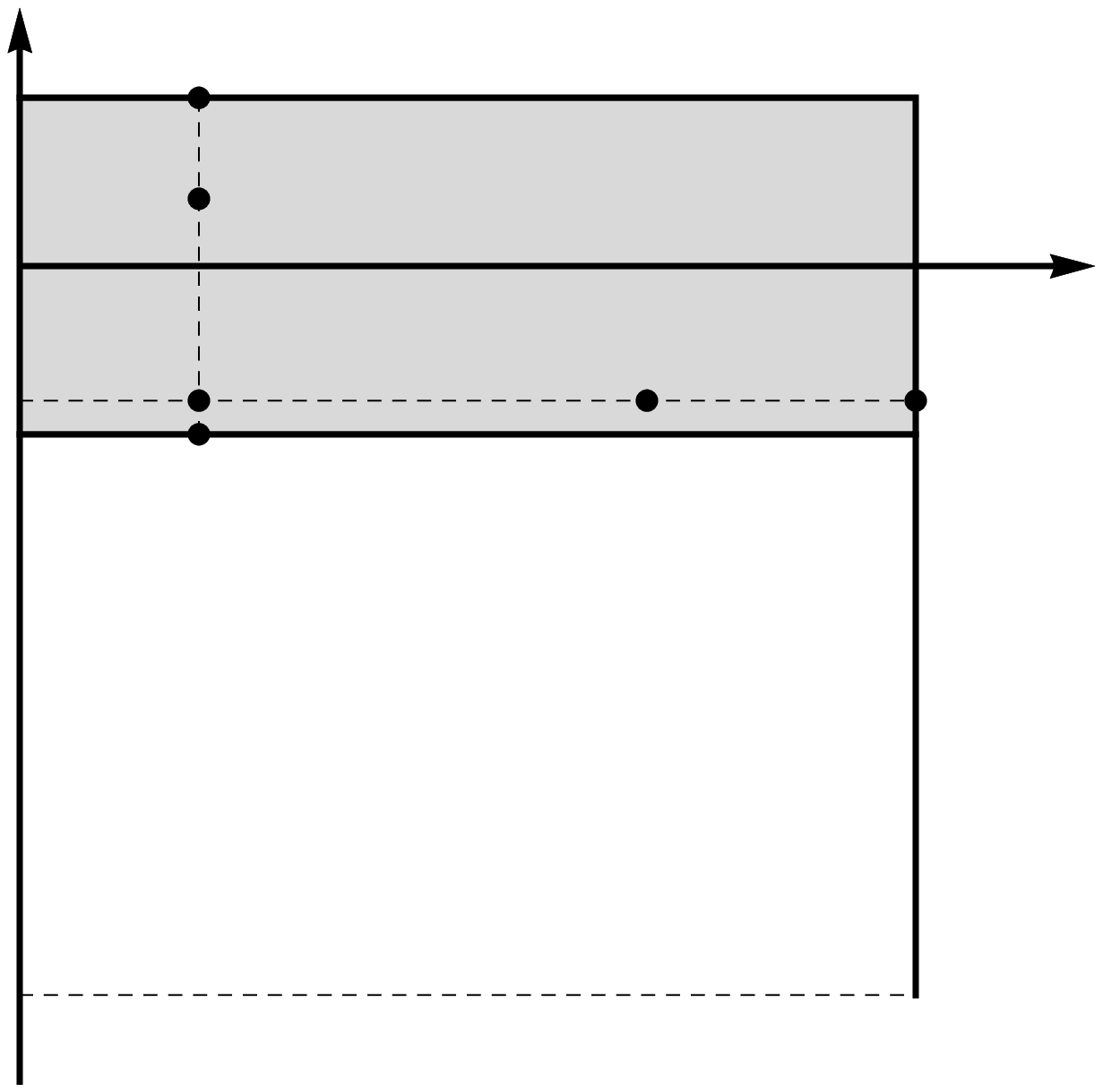}};
\node at (52,38) {$v$};
\node at (41,38) {$V$};
\node at (-1,33) {$w^-$};
\node at (-7,8) {$w^--1$};
\node at (-1,48.5) {$w^+$};
\node at (0,52) {$w$};
\node at (46,18) {$\Omega_{\rm f}^-$};
\node at (46,43) {$\Omega_{\rm f}^+$};
\node at (25,43) {$\Omega_{\rm c}$};
\node at (9,29.5) {$\mathtt{v}^-(u_r)$};
\node at (9,49.5) {$\mathtt{v}^+(u_r)$};
\node at (5.5,35.7) {$\mathtt{u}_*(u_\ell,u_r)$};
\node at (46,33.5) {$\omega(u_\ell)$};
\node at (9,45) {$u_r$};
\node at (34,35.7) {$u_\ell$};
\node at (52,0) {\color{white}{$\rho$}};
\end{tikzpicture}
}
\end{center}
\caption{Notations.}
\label{f:notations}
\end{figure}

\subsection{Notations}

Denote by $\rho \geq 0$ and $v \geq 0$ the density and the velocity of the vehicles, respectively.
Let $u \doteq (\rho,v)$ and $f(u) \doteq v\,\rho$ be the density flux.
If $V>0$ is the unique velocity in the free-flow phase $\Omega_{\rm f}$ and $\rho^+$ is the maximal density in $\Omega_{\rm f}$, then
\[
\Omega_{\rm f} \doteq \bigl\{ u \in \mathbb{R}_+^2 : \rho \leq \rho^+,\ v = V \bigr\},\]
where $\mathbb{R}_+ \doteq [0,\infty)$.
If the velocity $V$ is reached in the congested phase $\Omega_{\rm c}$ for densities ranging in $[\rho^-,\rho^+] \subset (0,\infty)$, then
\[
\Omega_{\rm c} \doteq \bigl\{ u \in \mathbb{R}_+^2 : v\le V,\ w^- \leq v + p(\rho) \leq w^+ \bigr\},
\]
where $w^\pm \doteq p(\rho^\pm) + V$.
Above $p \in {\mathbf{C^{2}}}((0,\infty);\mathbb{R})$ is an anticipation factor, which takes into account drivers' reactions to the state of traffic in front of them.
We assume that
\begin{align}\label{eq:p}
&p(0)=0,&
&p'(\rho)>0,&
&2\,p'(\rho) + p''(\rho) \, \rho>0&
\text{for every }\rho > 0.
\end{align}
Typical choices for $p$ are $p(\rho) \doteq \rho^\gamma$ with $\gamma>0$, see \cite{AwRascle}, and $p(\rho) \doteq V_{\rm ref} \ln(\rho/\rho_{\max})$ with $V_{\rm ref}>0$ and $\rho_{\max}>0$, see \cite{goatin2006aw}.

Let $f_{\rm c}^\pm \doteq V \, \rho^\pm$ and $R \doteq p^{-1}(w^+)>0$ be the maximal density (in the congested phase).
Let
\begin{align*}
&\Omega \doteq \Omega_{\rm f} \cup \Omega_{\rm c},&
&\Omega_{\rm f}^-\doteq \bigl\{ u \in \Omega_{\rm f} : \rho \in [0,\rho^-) \bigr\},&
&\Omega_{\rm f}^+\doteq \bigl\{ u \in \Omega_{\rm f} : \rho \in [\rho^-,\rho^+] \bigr\},&
&\Omega_{\rm c}^-\doteq\Omega_{\rm c}\setminus\Omega_{\rm f}^+.
\end{align*}
Notice that $\Omega_{\rm f}\cap\Omega_{\rm c} = \Omega_{\rm f}^+$.
We assume that
\begin{equation}\label{eq:TheMarsVolta}
v < p'(\rho) \, \rho \text{ for every } (\rho,v) \in \Omega_{\rm c}.
\end{equation}
The (extended) Lagrangian marker $\mathtt{w} \colon \Omega \to [w^--1,w^+]$ is defined by
\[
\mathtt{w}(u) \doteq
\begin{cases}
v + p(\rho)&\text{if }u \in \Omega_{\rm c},
\\
w^- -  1 + \dfrac{\rho}{\rho^-} &\text{if }u \in \Omega_{\rm f}^- .
\end{cases}
\]
Let $\mathtt{W} \colon \Omega \to [w^-,w^+]$ be defined by
\[
\mathtt{W}(u) \doteq \max\{w^-,\mathtt{w}(u)\}.
\]

The $2\times 2$ system of conservation laws describing the traffic in the congested phase has two characteristic families of Lax curves.
In the $(\rho,f)$-plane the Lax curves in $\Omega_{\rm c}$ of the first and second characteristic families passing through $\bar{u} = (\bar{\rho},\bar{v})\in \Omega_{\rm c}$ are respectively described by the graphs of the maps
\begin{align*}
\bigl[p^{-1}\bigl(\mathtt{w}(\bar{u})-V\bigr),p^{-1}\bigl(\mathtt{w}(\bar{u})\bigr)\bigr] \ni \rho &\mapsto \mathfrak{L}_{\mathtt{w}(\bar{u})}(\rho) \doteq f\bigl(\rho, \mathtt{w}(\bar{u})-p(\rho)\bigr),
\\
\bigl[p^{-1}(w^--\bar{v}),p^{-1}(w^+-\bar{v})\bigr]\ni \rho &\mapsto \bar{v} \, \rho.
\end{align*}

\begin{remark}
Conditions \eqref{eq:p} and \eqref{eq:TheMarsVolta} ensure that for any $w \in [w^-,w^+]$ the map $\rho \mapsto \mathfrak{L}_{w}(\rho) = (w-p(\rho))\,\rho$ is strictly concave and strictly decreasing in $[p^{-1}(w-V),p^{-1}(w)]$.
Indeed, for any $w \in [w^-,w^+]$ and $\rho \in [p^{-1}(w-V),p^{-1}(w)]$, we have that $(\rho,w-p(\rho)) \in \Omega_{\rm c}$ and therefore
\begin{align*}
\mathfrak{L}_{w}'(\rho) &= w-p(\rho)-p'(\rho)\,\rho < 0,&
\mathfrak{L}_{w}''(\rho) &= -2\,p(\rho)-p''(\rho)\,\rho < 0.
\end{align*}
If for instance $p(\rho) \doteq V_{\rm ref} \ln(\rho/\rho_{\max})$ with $V_{\rm ref}>0$ and $\rho_{\max}>0$, then $p'(\rho) \, \rho = V_{\rm ref}$ and \eqref{eq:TheMarsVolta} is equivalent to require $V < V_{\rm ref}$, while \eqref{eq:p} is trivial.
If for instance $p(\rho) \doteq \rho^\gamma$ with $\gamma>0$, then $\Omega_{\rm c} = \{ u \in \mathbb{R}_+^2 : v\le V,\ w^--v \leq \rho^\gamma \leq w^+-v \}$ and therefore
\[
\min_{u \in \Omega_{\rm c}} \bigl(p'(\rho) \, \rho - v\bigr) =
\min_{u \in \Omega_{\rm c}} \bigl(\gamma\,\rho^\gamma - v\bigr) =
\gamma\,w^--(\gamma+1)\,V,
\]
hence \eqref{eq:TheMarsVolta} is equivalent to require $(\gamma+1)\,V < \gamma\,w^-$, while \eqref{eq:p} is trivial.
\end{remark}

We introduce the following functions, see \figurename~\ref{f:notations}:
\begin{align*}
\omega &\colon \Omega_{\rm c} \to \Omega_{\rm f}^+,&
&u=\omega(\bar{u})\Longleftrightarrow
\begin{cases}
\mathtt{w}(u)=\mathtt{w}(\bar{u}),\\
v=V,
\end{cases}
\\
\mathtt{v}^\pm &\colon \Omega \to \Omega_{\rm c},&
&u^\pm=\mathtt{v}^\pm(\bar{u})\Longleftrightarrow
\begin{cases}
\mathtt{w}(u^\pm)=w^\pm,\\
v^\pm=\bar{v},
\end{cases}
\\
{\mathtt u}_* &\colon \Omega^2 \to \Omega_{\rm c},&
&u_*={\mathtt u}_*(u_\ell,u_r)\Longleftrightarrow
\begin{cases}
\mathtt{w}(u_*) = \mathtt{W}(u_\ell),\\
v_* = v_r,
\end{cases}
\\
\Lambda &\colon \bigl\{(u_\ell,u_r) \in \Omega^2 : \rho_\ell \neq \rho_r\bigr\} \to \mathbb{R},&
&\Lambda(u_\ell, u_r) \doteq \frac{f(u_r)-f(u_\ell)}{\rho_r-\rho_\ell}.
\end{align*}
Notice that:
\begin{itemize}[leftmargin=*]\setlength{\itemsep}{0cm}%
\item
the point $\omega(\bar{u})$ is the intersection of the Lax curve of the first characteristic family passing through $\bar{u}$ and $\Omega_{\rm f}^+$, namely the Lax curve of the second characteristic family passing through $(0,V)$;
\item
for any $w \in [w^-,w^+]$ the point $(p^{-1}(w), 0)$ is the intersection of the Lax curve of the first characteristic family corresponding to $w$ and the segment $\{(\rho,v) \in \Omega_{\rm c} : v = 0\}$, namely the Lax curve of the second characteristic family passing through $(p^{-1}(w^\pm),0)$;
\item
the point $\mathtt{v}^\pm(\bar{u})$ is the intersection of the Lax curve of the second characteristic family passing through $\bar{u}$ and $\{u \in \Omega_{\rm c} : \mathtt{w}(u) = w^\pm\}$, namely the Lax curve of the first characteristic family passing through $(p^{-1}(w^\pm),0)$;
\item
for any $u_\ell$, $u_r \in \Omega_{\rm c}$ the point $\mathtt{u}_*(u_\ell,u_r)$ is the intersection between the Lax curve of the first characteristic family passing through $u_\ell$ and the Lax curve of the second characteristic family passing through $u_r$;
\item
$\Lambda(u_\ell, u_r)$ is the speed of a discontinuity $(u_\ell,u_r)$, that in the $(\rho,f)$-coordinates coincides with the slope of the segment connecting $u_\ell$ and $u_r$.
\end{itemize}
\noindent
Observe that by definition $\mathtt{v}^\pm(\bar{u}) = \mathtt{u}_*((p^{-1}(w^\pm),0),\bar{u})$ and $\omega(\bar{u}) = \mathtt{u}_*(\bar{u},(0,V))$.

We denote by $\mathcal{R}$ and $\mathcal{R}_F$ the Riemann solver and the constrained Riemann solver introduced in~\cite{BenyahiaRosini01} and~\cite{edda-nikodem-mohamed}, respectively, see Section~\ref{s:RiemannSs} for more details.

\subsection{The constrained Cauchy problem} 

We study the constrained Cauchy problem for the phase transition model
\begin{align}\label{eq:system}
&\begin{array}{l}
\text{\textbf{Free-flow}}\\[2pt]
\begin{cases}
u \in \Omega_{\rm f},\\
\rho_t+(\rho\,V)_x=0,\\
v=V,
\end{cases}
\end{array}
&
\begin{array}{l}
\text{\textbf{Congested flow}}\\[2pt]
\begin{cases}
u \in \Omega_{\rm c},\\
\rho_t+(\rho\,v)_x=0,\\
\bigl(\rho\,\mathtt{w}(u)\bigr)_t + \bigl(\rho\,\mathtt{w}(u)\,v\bigr)_x=0,
\end{cases}
\end{array}
\end{align}
with initial datum
\begin{equation}\label{eq:initdat}
u(0,x)=u^o(x)
\end{equation}
and local point constraint on the density flux at $x=0$
\begin{equation}\label{eq:const}
f\bigl(u(t,0_\pm)\bigr) \le F,
\end{equation}
where $F \in [0,f_{\rm c}^+]$ is a given constant quantity.
To this aim we apply the wave-front tracking algorithm, which is based on the definition of the Riemann solvers defined in the next sections.

Introduce, see \figurename~\ref{f:notationsF}, $v_F^\pm \in [0,V]$ and $w_F \in [w^--1,w^+]$ defined by the following conditions:
\begin{align*}
&\text{if }F = f_{\rm c}^+:&
&v_F^+ \doteq V,&
&v_F^- \doteq V,&
&w_F \doteq w^+,
\\
&\text{if }F \in [f_{\rm c}^-,f_{\rm c}^+):&
&v_F^+ \doteq V,&
&v_F^-+p(F/v_F^-) = w^+,&
&w_F\doteq p\left(F/V\right)+V,
\\
&\text{if }F \in (0,f_{\rm c}^-):&
&v_F^++p(F/v_F^+) = w^-,&
&v_F^-+p(F/v_F^-) = w^+,&
&w_F \doteq w^- - 1 + \frac{F}{f_{\rm c}^-},
\\
&\text{if }F = 0:&
&v_F^+ \doteq 0,&
&v_F^- \doteq 0,&
&w_F \doteq w^- - 1.
\end{align*}

\begin{figure}[!ht]
\begin{center}
\begin{tikzpicture}[every node/.style={anchor=south west,inner sep=0pt},x=1mm, y=1mm]
\node at (4,4) {\includegraphics[height=50mm]{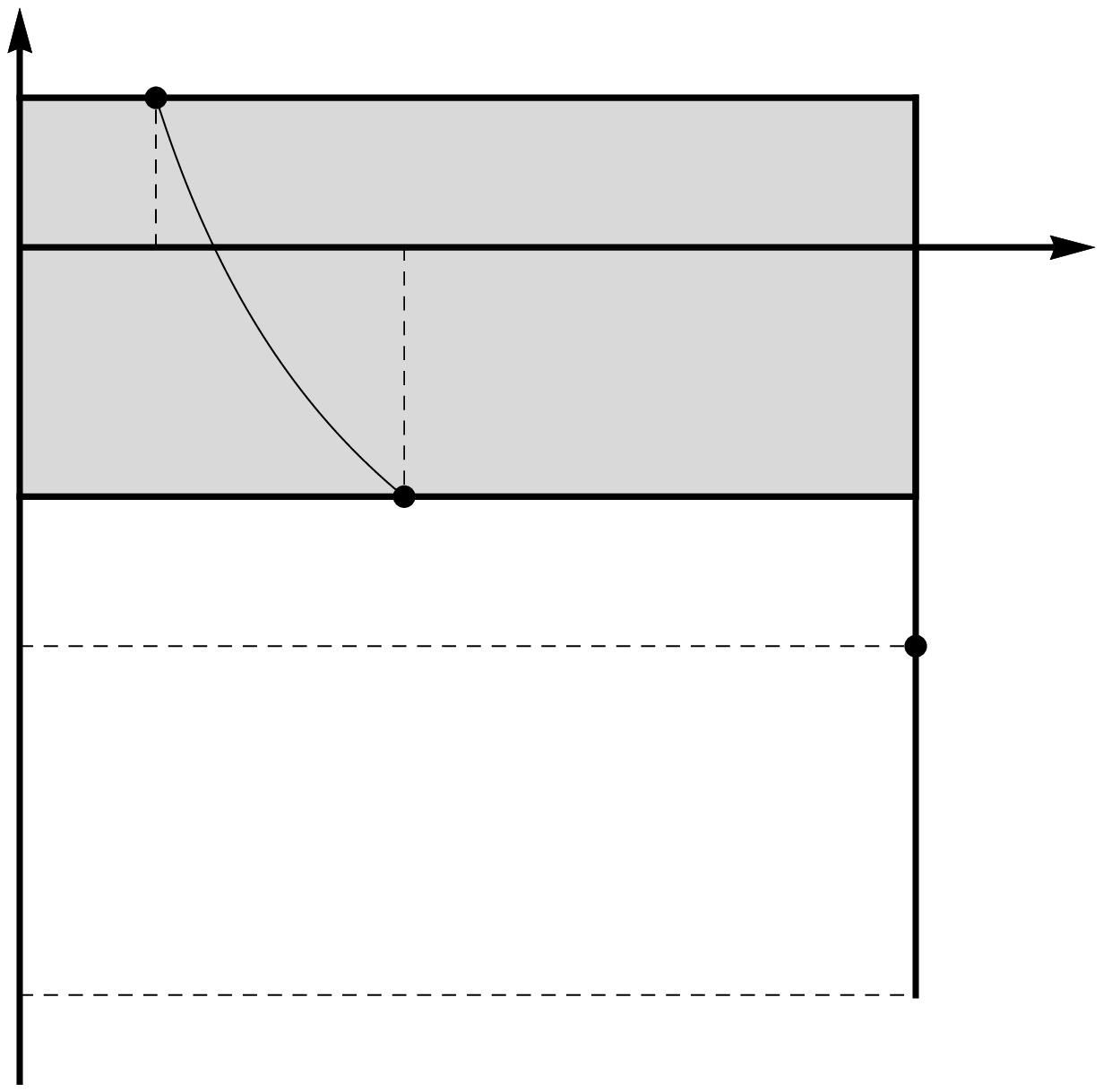}};
\node at (52,38) {$v$};
\node at (41,38) {$V$};
\node at (-1,31) {$w^-$};
\node at (-1,23) {$w_F$};
\node at (-7,8) {$w^--1$};
\node at (-1,48.5) {$w^+$};
\node at (0,52) {$w$};
\node at (9,38) {$v_F^-$};
\node at (20,43) {$v_F^+$};
\node at (52,0) {\color{white}{$\rho$}};
\end{tikzpicture}
\hspace{2cm}
\begin{tikzpicture}[every node/.style={anchor=south west,inner sep=0pt},x=1mm, y=1mm]
\node at (4,4) {\includegraphics[height=50mm]{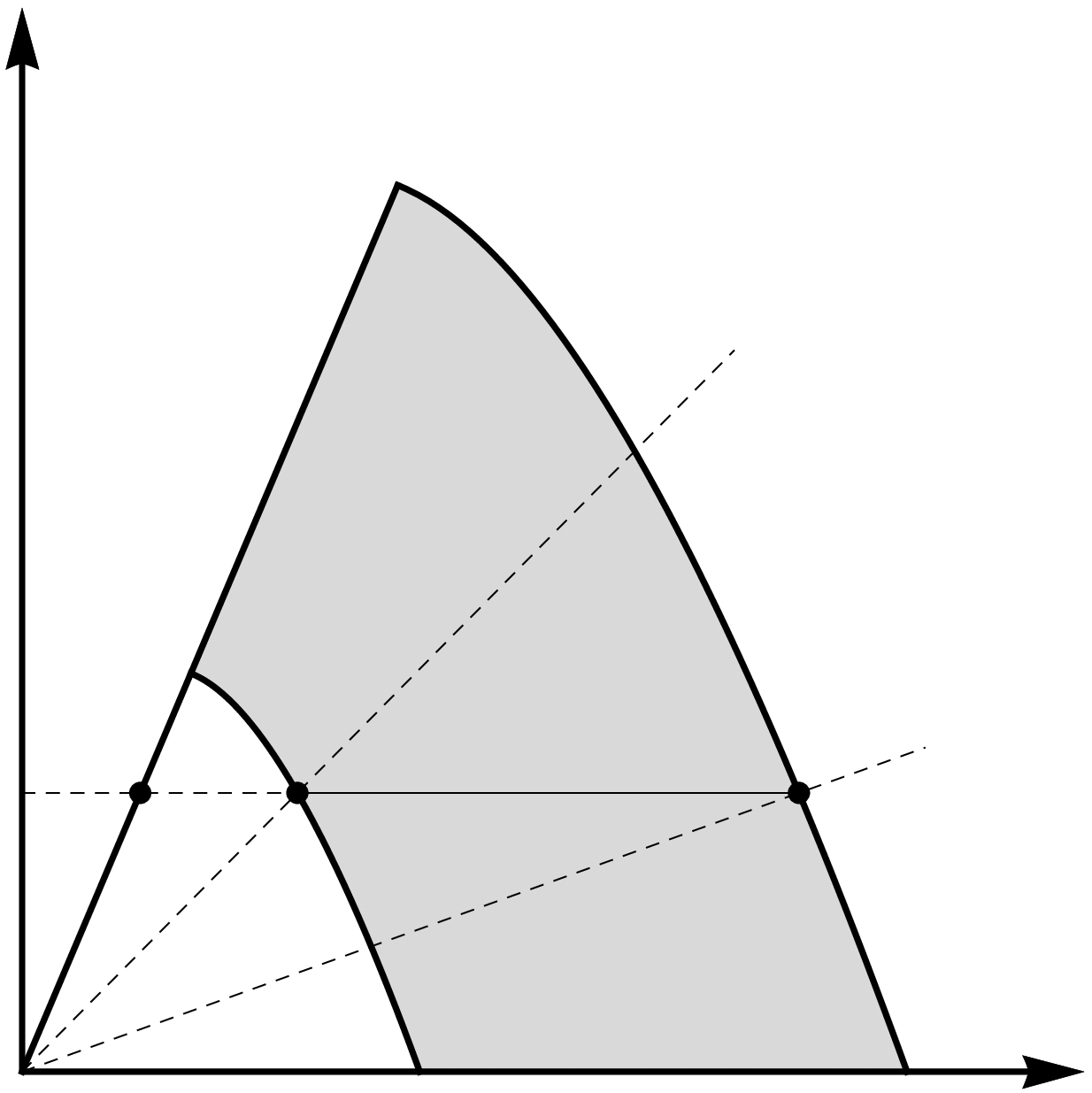}};
\node at (52,0) {$\rho$};
\node at (44,0) {$R$};
\node at (45,20) {$v_F^-$};
\node at (37,39) {$v_F^+$};
\node at (0,50) {$f$};
\node at (0,16) {$F$};
\end{tikzpicture}
\end{center}
\caption{Geometrical meaning of $w_F$, $v_F^\pm$ and $\Xi_F$ in the case $F \in (0,f_{\rm c}^-)$.
The curve in the figure on the left is the graph of $\Xi_F$, which corresponds to the horizontal solid segment in the figure on the right.}
\label{f:notationsF}
\end{figure}

For any $F \in (0,f_{\rm c}^+)$, let $\Xi_F:[v_F^-,v_F^+]\to[w^-,w^+]$ be given by $\Xi_F(v) \doteq v+p(F/v)$, see \figurename~\ref{f:notationsF}.
Notice that $\Xi_F$ is strictly decreasing because by \eqref{eq:TheMarsVolta}
\begin{align*}
\left(\frac{F}{v},v\right) &\in \Omega_{\rm c}&
&\Rightarrow&
\Xi_F'(v) &= 1-p'\left(\frac{F}{v}\right) \frac{F}{v^2} < 0,
\end{align*}
moreover it is strictly convex because by \eqref{eq:p}
\[
\Xi_F''(v) = \left[ 2\,p'\left(\frac{F}{v}\right)+p''\left(\frac{F}{v}\right)\frac{F}{v} \right] \frac{F}{v^3} >0.
\]

The notion of solution to Cauchy problem \eqref{eq:system}, \eqref{eq:initdat} necessarily involves both the notions of solution to the Cauchy problems for LWR and ARZ models, that have to be combined by defining which phase transitions are admissible, see~\cite{BenyahiaRosini01}.
Since the characteristic field corresponding to the free phase is linearly degenerate, a discontinuity between states in $\Omega_{\rm f}$ is entropic if and only if it satisfies the corresponding Rankine-Hugoniot condition, namely its speed of propagation is $V$.
For this reason we consider only the entropy-entropy flux pair
\begin{align*}
\mathtt{E}^k(u) &\doteq \begin{cases}
0 & \hbox{if } v \geq k,\\
\dfrac{\rho }{p^{-1}\bigl(\mathtt{W}(u)-k\bigr)} - 1 & \hbox{if } v < k,
\end{cases}
&
\mathtt{Q}^k(u) &\doteq \begin{cases}
0 & \hbox{if } v \geq k,\\
\dfrac{f(u)}{p^{-1}\bigl(\mathtt{W}(u)-k\bigr)} - k & \hbox{if } v < k,
\end{cases}
\end{align*}
for $u \in \Omega$ and $k \in [0,V]$, which is obtained by adapting the entropy-entropy flux pair introduced in~\cite{AndreianovDonadelloRosiniM3ASS2016} for the ARZ model. 

\begin{definition}\label{def:solunconstmod}
Let $u^o \in \BV(\R;\Omega)$.
We say that $u \in \L\infty((0,\infty);\BV(\R;\Omega)) \cap \C0(\R_+;\Lloc1(\R;\Omega))$ is a solution to Cauchy problem \eqref{eq:system}, \eqref{eq:initdat} if the following holds:
\begin{enumerate}[label={\textbf{(S.\arabic*)}},leftmargin=*]\setlength{\itemsep}{0cm}%
\item Condition \eqref{eq:initdat} holds for a.e.~$x \in \R$, namely
\[u(0,x) = u^o(x)\qquad \text{for a.e.~}x \in \R.\]
\item For any $\phi \in \Cc\infty((0,\infty)\times \R;\R)$ we have
\[
\int_0^\infty \int_{\R} 
\bigl(\rho \, \phi_t + f(u) \, \phi_x\bigr)
\begin{pmatrix}
1\\\mathtt{W}(u)
\end{pmatrix} \d x\,\d t = 
\begin{pmatrix}
0\\0
\end{pmatrix}.
\]
\item For any $k \in [0,V]$ and $\phi \in \Cc\infty((0,\infty)\times \R;\R)$ such that $\phi \geq0$ we have
\[
\int_0^\infty \int_{\R} \bigl(\mathtt{E}^k(u) \, \phi_t + \mathtt{Q}^k(u) \, \phi_x\bigr) \, \d x\,\d t \geq 0 .
\]
\end{enumerate}
\end{definition}

We recall the existence result proved in~\cite[Theorem~2.8]{BenyahiaRosini01}.

\begin{theorem}
Cauchy problem \eqref{eq:system}, \eqref{eq:initdat} with initial datum $u^o \in \L1 \cap \BV(\R; \Omega)$ admits a solution $u$ in the sense of Definition~\ref{def:solunconstmod}; moreover there exist two constants $C^o$ and $L^o$ such that for any $t,s \geq 0$
\begin{align*}
& \tv\bigl(u(t)\bigr) \leq \tv(u^o) ,&
&\left\|u(t)\right\|_{\L\infty(\R; \Omega)} \leq C^o ,&
&\left\|u(t) - u(s)\right\|_{\L1(\R; \Omega)} \leq L^o \ |t - s| .
\end{align*}
\end{theorem}

In the following definition we introduce the notion of solution to constrained Cauchy problem \eqref{eq:system}, \eqref{eq:initdat}, \eqref{eq:const}, which is obtained by adapting that introduced in Definition~\ref{def:solunconstmod} for Cauchy problem \eqref{eq:system}, \eqref{eq:initdat}.

\begin{definition}\label{def:solconstmod}
Let $u^o \in \BV(\R;\Omega)$.
We say that $u \in \L\infty\left((0,\infty);\BV(\R;\Omega)\right) \cap \C0\left(\R_+;\Lloc1(\R;\Omega)\right)$ is a solution to constrained Cauchy problem \eqref{eq:system}, \eqref{eq:initdat}, \eqref{eq:const} if the following holds:
\begin{enumerate}[label={\textbf{(CS.\arabic*)}},leftmargin=*]\setlength{\itemsep}{0cm}%
\item\label{CS1} Condition \eqref{eq:initdat} holds for a.e.~$x \in \R$, namely
\[u(0,x) = u^o(x)\qquad \text{for a.e.~}x \in \R.\]
\item\label{CS2} For any $\phi \in \Cc\infty((0,\infty)\times\R; \R)$ we have
\begin{equation}\label{e:AdrianBelew0}
\int_0^\infty
\int_{\R} 
\bigl( \rho \, \phi_t
+ f(u)\, \phi_x \bigr) \,
\d x\, \d t
=0
\end{equation}
and if $\phi(\cdot,0)\equiv0$ then
\begin{equation}\label{e:AdrianBelew1}
\int_0^\infty
\int_{\R} 
\bigl( \rho \, \phi_t
+ f(u)\, \phi_x \bigr) \, \mathtt{W}(u) \,
\d x\, \d t
=0.
\end{equation}

\item\label{CS3} For any $k \in [0,V]$ and $\phi \in \Cc\infty((0,\infty)\times \R;\R)$ such that $\phi(\cdot,0)\equiv0$ and $\phi \geq0$ we have
\begin{equation}
\label{e:Deftones}
\int_0^\infty
\int_{\R} \bigl(\mathtt{E}^k(u) \, \phi_t + \mathtt{Q}^k(u) \, \phi_x\bigr) \, \d x\, \d t \geq 0.
\end{equation}

\item\label{CS4} Condition \eqref{eq:const} holds for a.e.~$t >0$, namely
\[f\bigl(u(t,0_\pm)\bigr) \le F \qquad \text{for a.e.~}t >0.\]
\end{enumerate}
\end{definition}

In the following proposition we state which discontinuities are admissible for the solutions to \eqref{eq:system}, \eqref{eq:initdat}, \eqref{eq:const}.

\begin{proposition}\label{p:obvious}
Let $u$ be a solution of constrained Cauchy problem \eqref{eq:system}, \eqref{eq:initdat}, \eqref{eq:const} in the sense of Definition~\ref{def:solconstmod}. 
Then $u$ has the following properties:
\begin{itemize}
\item Any discontinuity $\delta(t)$ of $x\mapsto u(t,x)$ satisfies the first Rankine-Hugoniot jump condition
\begin{equation}
\Bigl[\rho\bigl(t,\delta(t)_+\bigr)-\rho\bigl(t,\delta(t)_-\bigr)\Bigr] \, \dot{\delta}(t) = f\bigl(u(t,\delta(t)_+)\bigr)-f\bigl(u(t,\delta(t)_-)\bigr),
\label{e:RH1}
\end{equation}
and if $\delta(t)\neq0$, then it satisfies also the second Rankine-Hugoniot jump condition
\begin{align}\nonumber&\
\Bigl[\rho\bigl(t,\delta(t)_+\bigr)\,\mathtt{W}\Bigl(u\bigl(t,\delta(t)_+\bigr)\Bigr)-\rho\bigl(t,\delta(t)_-\bigr)\,\mathtt{W}\Bigl(u\bigl(t,\delta(t)_-\bigr)\Bigr)\Bigr] \, \dot{\delta}(t)
\\
=&\ f\Bigl(u\bigl(t,\delta(t)_+\bigr)\Bigr)\,\mathtt{W}\Bigl(u\bigl(t,\delta(t)_+\bigr)\Bigr)-f\Bigl(u\bigl(t,\delta(t)_-\bigr)\Bigr)\,\mathtt{W}\Bigl(u\bigl(t,\delta(t)_-\bigr)\Bigr).
\label{e:RH2}
\end{align}

\item Any discontinuity of $u$ away from the constraint is classical, i.e.~it satisfies the Lax entropy inequalities.
\item Non-classical discontinuities of $u$ may occur only at the constraint location $x=0$, and in this case the (density) flux at $x=0$ does not exceed the maximal flux $F$ allowed by the constraint.
\end{itemize}
\end{proposition}

\begin{proof}
These properties follow directly from~\ref{CS2}, \ref{CS3} and~\ref{CS4}.
Let us just underline that \eqref{e:RH1}, \eqref{e:RH2} are equivalent to
\begin{align*}
\Bigl[v(t,0_+)-\dot{\delta}(t)\Bigr] \, \rho(t,0_+)&=
\Bigl[v(t,0_-)-\dot{\delta}(t)\Bigr] \, \rho(t,0_-),
\\
\Bigl[\mathtt{W}\bigl(u(t,0_+)\bigr)-\mathtt{W}\bigl(u(t,0_-)\bigr)\Bigr]
\Bigl[v(t,0_-)-\dot{\delta}(t)\Bigr] \, \rho(t,0_-)&=0.
\end{align*}
In particular phase transitions and shocks are admissible because for them $\mathtt{W}(u(t,0_+)) = \mathtt{W}(u(t,0_-))$, while the contact discontinuities are admissible because for them $\dot{\delta}(t) = v(t,0_\pm)$.
\end{proof}

\begin{remark}
Differently from any solution to Cauchy problem \eqref{eq:system}, \eqref{eq:initdat}, a solution $u$ to constrained Cauchy problem \eqref{eq:system}, \eqref{eq:initdat}, \eqref{eq:const} does not satisfy in general the second Rankine-Hugoniot condition \eqref{e:RH2} along $x=0$
\[
\rho(t,0_-) \, \mathtt{W}\bigl(u(t,0_-)\bigr) \, v(t,0_-) = \rho(t,0_+) \, \mathtt{W}\bigl(u(t,0_+)\bigr) \, v(t,0_+)
\qquad
\text{for a.e.~}t>0.
\]
Indeed the (extended) linearized momentum $\rho \, \mathtt{W}(u)$ is conserved across (classical) shocks and phase transitions, but in general it is not conserved across non-classical shocks even if they are between states in $\Omega_{\rm c}$.
As a consequence, a solution to \eqref{eq:system}, \eqref{eq:initdat}, \eqref{eq:const} taking values in $\Omega_{\rm c}$ is not necessarily a weak solution to the $2\times2$ system of conservation laws in \eqref{eq:system} for the congested flow.
For this reason in \eqref{e:AdrianBelew1} (and then also in \eqref{e:Deftones}) we consider test functions $\phi$ such that $\phi(\cdot,0)\equiv0$.

This is in the same spirit of the solutions considered in \cite{BenyahiaRosini02, dellemonachegoatin, MaxNikPaola, edda-nikodem-mohamed, GaravelloGoatin2011} for traffic through locations with reduced capacity.
However, with this choice for the test functions in \eqref{e:AdrianBelew1} and \eqref{e:Deftones} we loose the possibility to better characterize the (density) flux at $x=0$ associated to non-classical shocks.
In fact, differently from what is proved in~\cite{ColomboGoatinConstraint} for the LWR model and in~\cite{BCJMU-order2} for the ARZ model, we cannot ensure that the flux of the non-classical shocks of any solution is equal to the maximal flux $F$ allowed by the constraint.
Nevertheless, in Section~\ref{s:opt} we can give sufficient conditions ensuring that the solutions constructed with our wave-front tracking algorithm have this property, see Proposition~\ref{p:TheWineryDogs}.
\end{remark}

Let $[w^--1,w^+]\ni w\mapsto\hat{\mathtt u}(w,F) = (\hat{\mathtt r}(w,F),\hat{\mathtt v}(w,F)) \in \Omega_{\rm c}$ and $[0,V]\ni v\mapsto\check{\mathtt u}(v,F) = (\check{\mathtt r}(v,F),\check{\mathtt v}(v,F))  \in \Omega$ be defined in the $(v,w)$-coordinates by, see \figurename s~\ref{f:uhatucheck1} and~\ref{f:uhatucheck2},
\begin{subequations}\label{e:hat-check}
\begin{align}
\hat{\mathtt v}(w,F) &\doteq 
\begin{cases}
\Xi_F^{-1}(w)&\text{if }w>\max\{w^-,w_F\},
\\
v_F^+&\text{if }w_F<w\leq w^-,
\\
V&\text{if }w\leq w_F,
\end{cases}&
\hat{\mathtt w}(w,F) &\doteq 
\begin{cases}
w&\text{if }w>\max\{w^-,w_F\},
\\
w^-&\text{if }w_F<w\leq w^-,
\\
w_F&\text{if }w\leq w_F,
\end{cases}
\\
\check{\mathtt v}(v,F) &\doteq 
\begin{cases}
V&\hbox{if }v> v_F^+,\\
v&\hbox{if }v \in [v_F^-,v_F^+],\\
v_F^-&\hbox{if }v< v_F^-,
\end{cases}&
\check{\mathtt w}(v,F) &\doteq 
\begin{cases}
w_F&\hbox{if }v> v_F^+,\\
\Xi_F(v)&\hbox{if }v \in [v_F^-,v_F^+],\\
w^+&\hbox{if }v< v_F^-,
\end{cases}
\end{align}
\end{subequations}
where $\hat{\mathtt w} \doteq \mathtt{w} \circ \hat{\mathtt u}$ and $\check{\mathtt w} \doteq \mathtt{w} \circ \check{\mathtt u}$.
\begin{remark}\label{r:hatcheck}
Notice that
\[
f\bigl(\hat{\mathtt u}(w,F)\bigr)=f\bigl(\check{\mathtt u}(v,F)\bigr)=F.
\]
Moreover, $w\mapsto\hat{\mathtt u}(w,F)$ and $v\mapsto\check{\mathtt u}(v,F)$ are continuous if and only if $F\geq f_{\rm c}^-$, and in this case they are Lipschitz continuous.
On the other hand, if $F<f_{\rm c}^-$, then $w\mapsto\hat{\mathtt u}(w,F)$ and $v\mapsto\check{\mathtt u}(v,F)$ are only left-continuous.
Moreover $\hat{\mathtt w}(w,F) \geq w$ and $\check{\mathtt v}(v,F) \geq v$.
At last, $w\mapsto\hat{\mathtt w}(w,F)$ and $v\mapsto \check{\mathtt v}(v,F)$ are non-decreasing, while $w\mapsto\hat{\mathtt v}(w,F)$ and $v\mapsto\check{\mathtt w}(v,F)$ are non-increasing.
\end{remark}

\begin{figure}[!ht]
\begin{flushright}
\begin{tikzpicture}[every node/.style={anchor=south west,inner sep=0pt},x=1mm, y=1mm]
\node at (4,4) {\includegraphics[height=50mm]{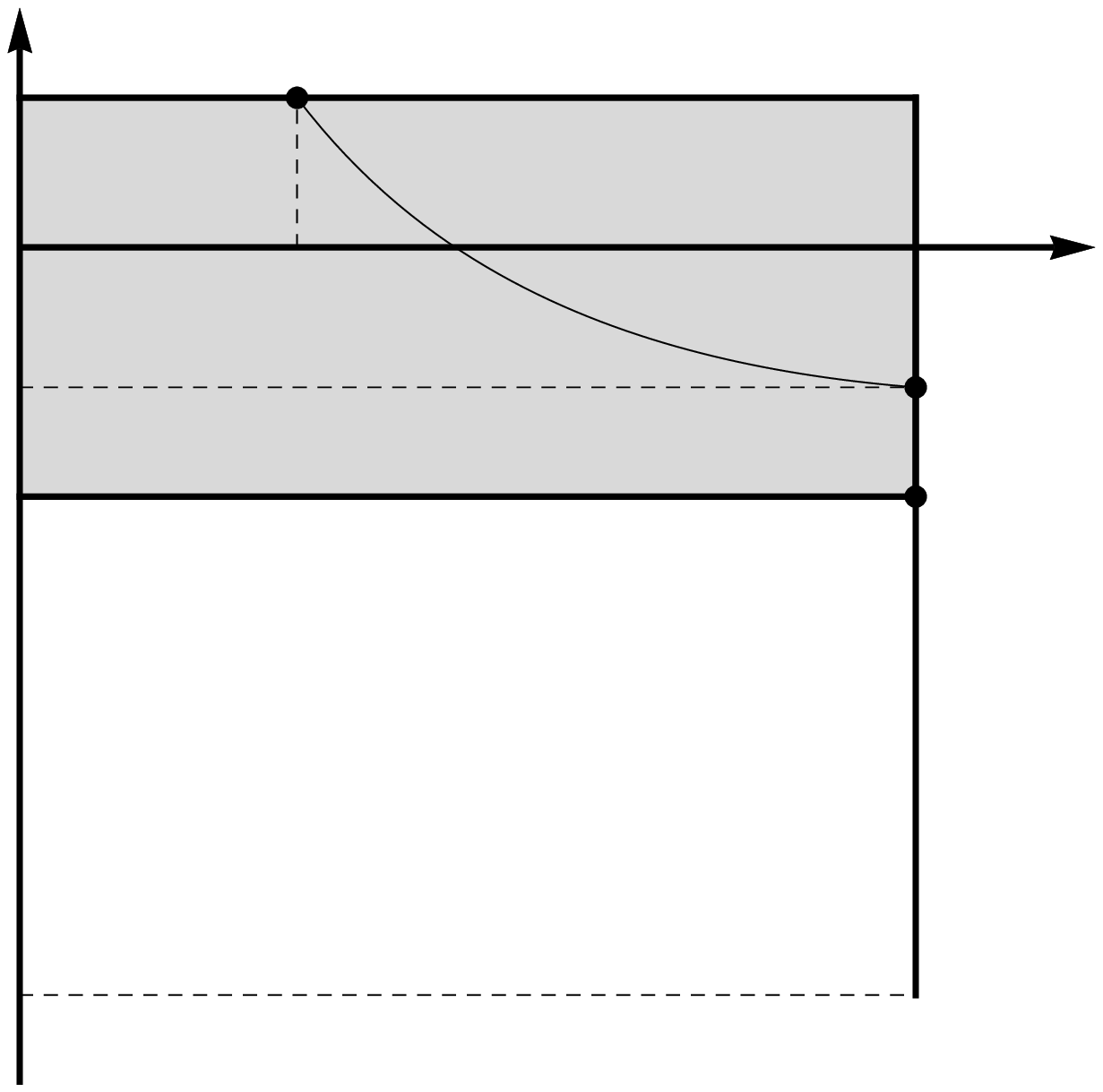}};
\node at (52,38) {$v$};
\node at (33,43) {$V = v_F^+$};
\node at (-1,30) {$w^-$};
\node at (-1,35) {$w_F$};
\node at (-7,8) {$w^--1$};
\node at (-1,48.5) {$w^+$};
\node at (0,52) {$w$};
\node at (15,38) {$v_F^-$};
\node at (52,-1) {\strut \color{white}{$\rho$}};
\end{tikzpicture}
\begin{tikzpicture}[every node/.style={anchor=south west,inner sep=0pt},x=1mm, y=1mm]
\node at (4,4) {\includegraphics[height=50mm]{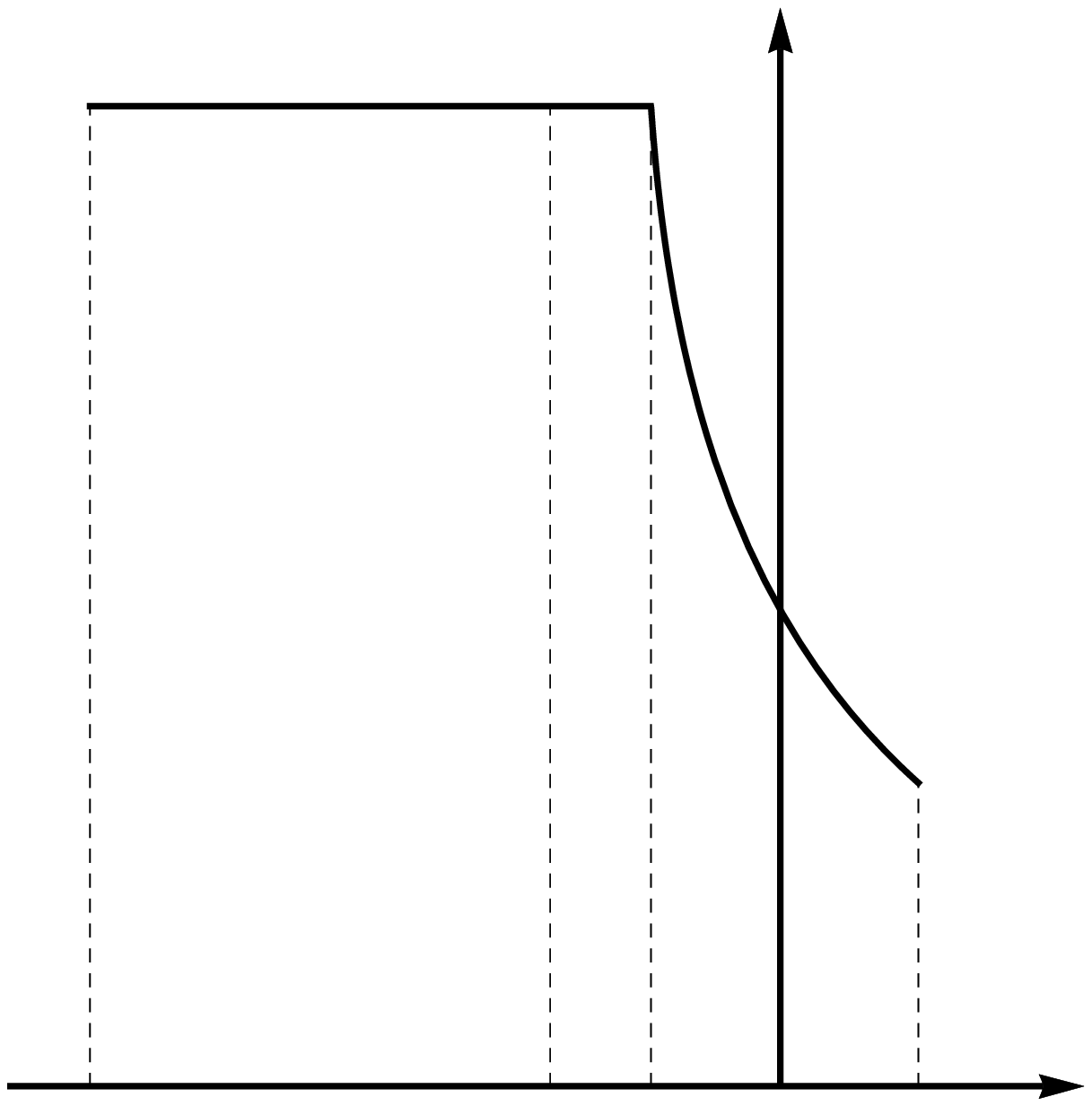}};
\node at (27,-1) {\strut $w^-$};
\node at (33,-1) {\strut $w_F$};
\node at (4,-1) {$\strut w^--1$};
\node at (44,-1) {\strut $w^+$};
\node at (35,52) {$\hat{\mathtt v}$};
\node at (48,16) {$v_F^-$};
\node at (20,50) {$V$};
\node at (52,-1) {\strut $w$};
\end{tikzpicture}
\begin{tikzpicture}[every node/.style={anchor=south west,inner sep=0pt},x=1mm, y=1mm]
\node at (4,4) {\includegraphics[height=50mm]{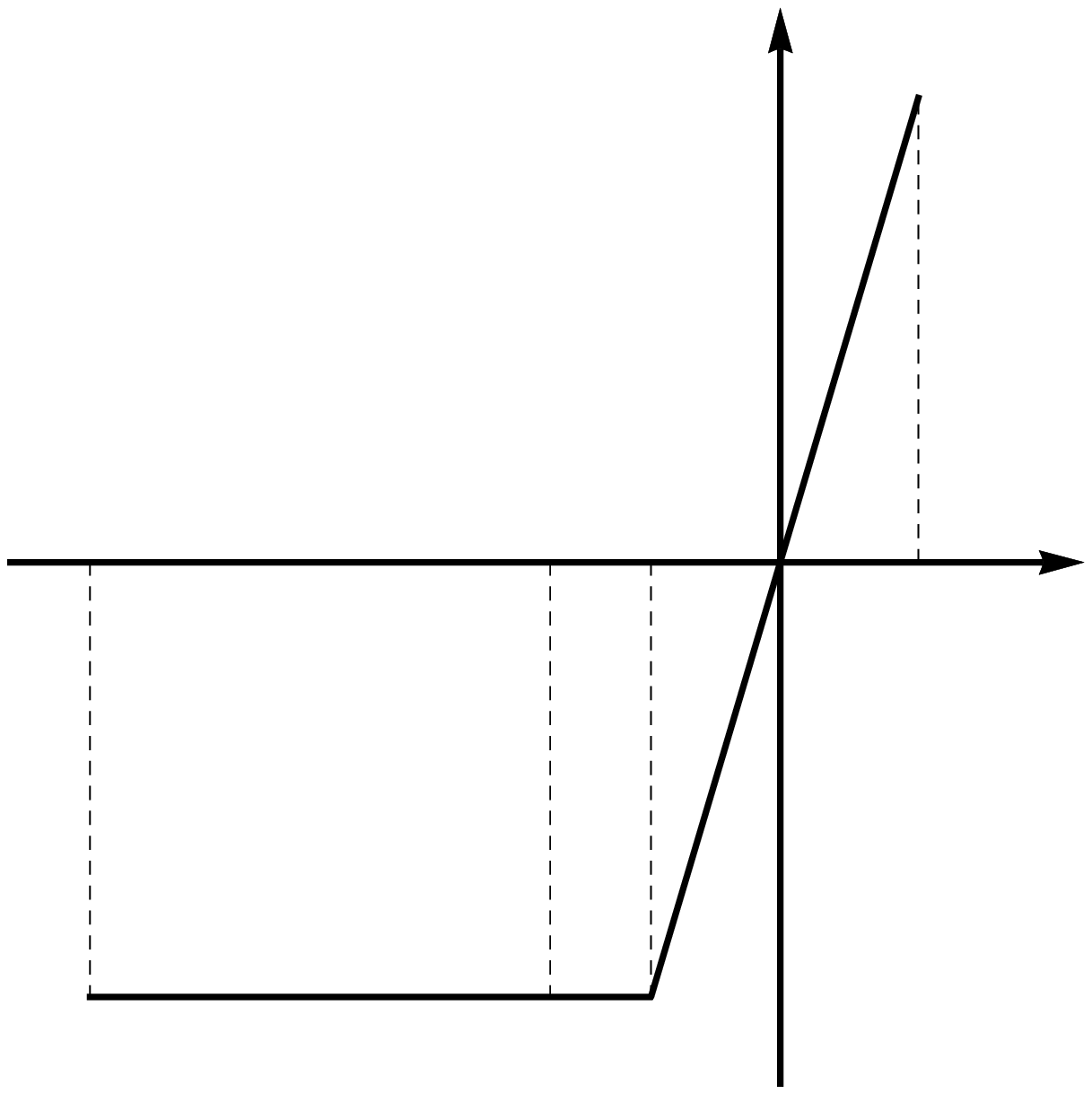}};
\node at (27,28.5) {\strut $w^-$};
\node at (33,28.5) {\strut $w_F$};
\node at (4,28.5) {$\strut w^--1$};
\node at (44,23) {\strut $w^+$};
\node at (35,52) {$\hat{\mathtt w}$};
\node at (52,23) {\strut $w$};
\node at (15,10) {$w_F$};
\node at (47,48) {$w^+$};
\node at (52,-1) {\strut \color{white}{$\rho$}};
\end{tikzpicture}
\\
\begin{tikzpicture}[every node/.style={anchor=south west,inner sep=0pt},x=1mm, y=1mm]
\node at (4,4) {\includegraphics[height=50mm]{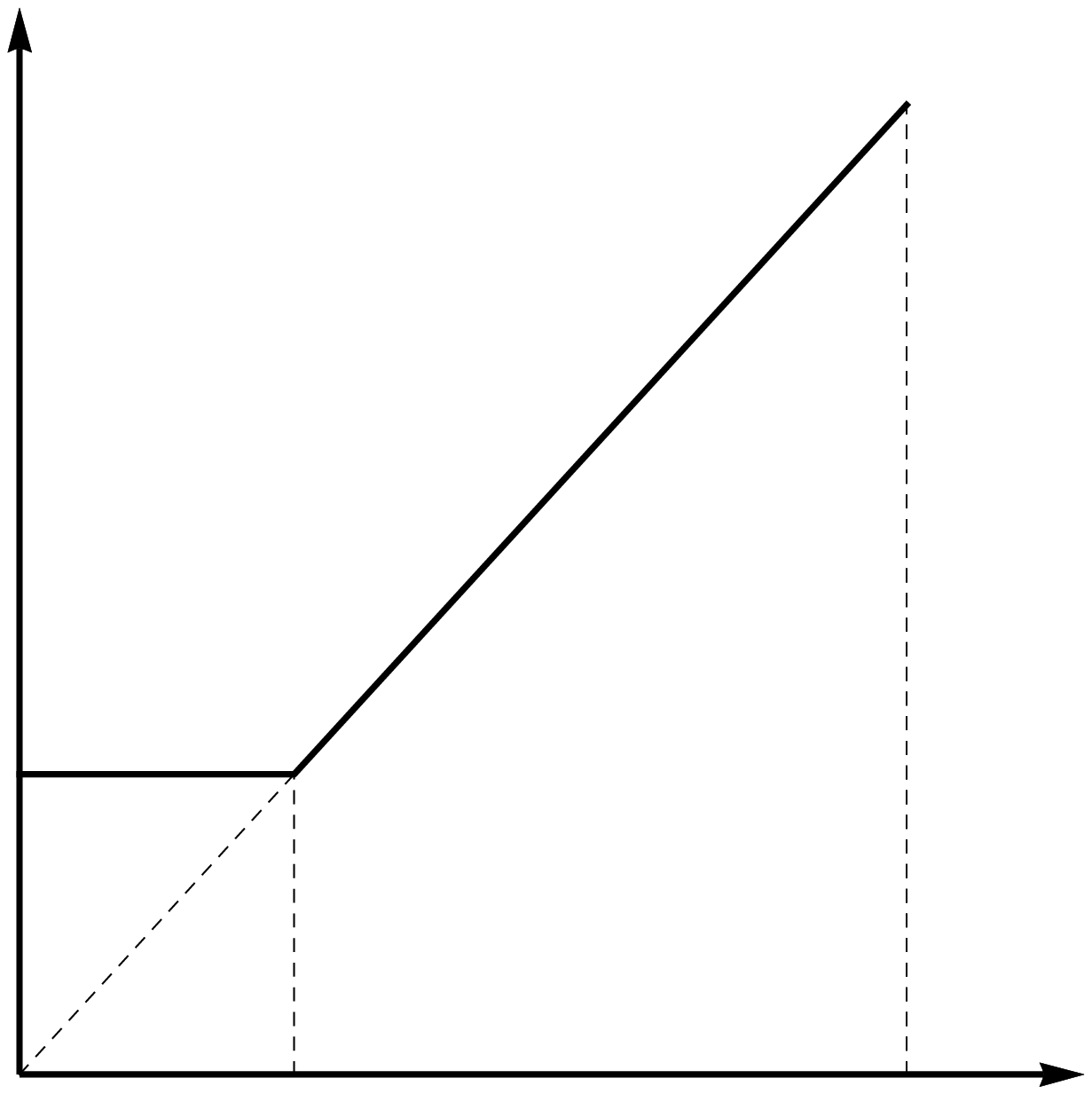}};
\node at (52,-1) {\strut $v$};
\node at (48,48) {$V$};
\node at (44,-1) {\strut $V$};
\node at (0,52) {$\check{\mathtt v}$};
\node at (8,19) {$v_F^-$};
\node at (15,-1) {\strut $v_F^-$};
\end{tikzpicture}
\begin{tikzpicture}[every node/.style={anchor=south west,inner sep=0pt},x=1mm, y=1mm]
\node at (4,4) {\includegraphics[height=50mm]{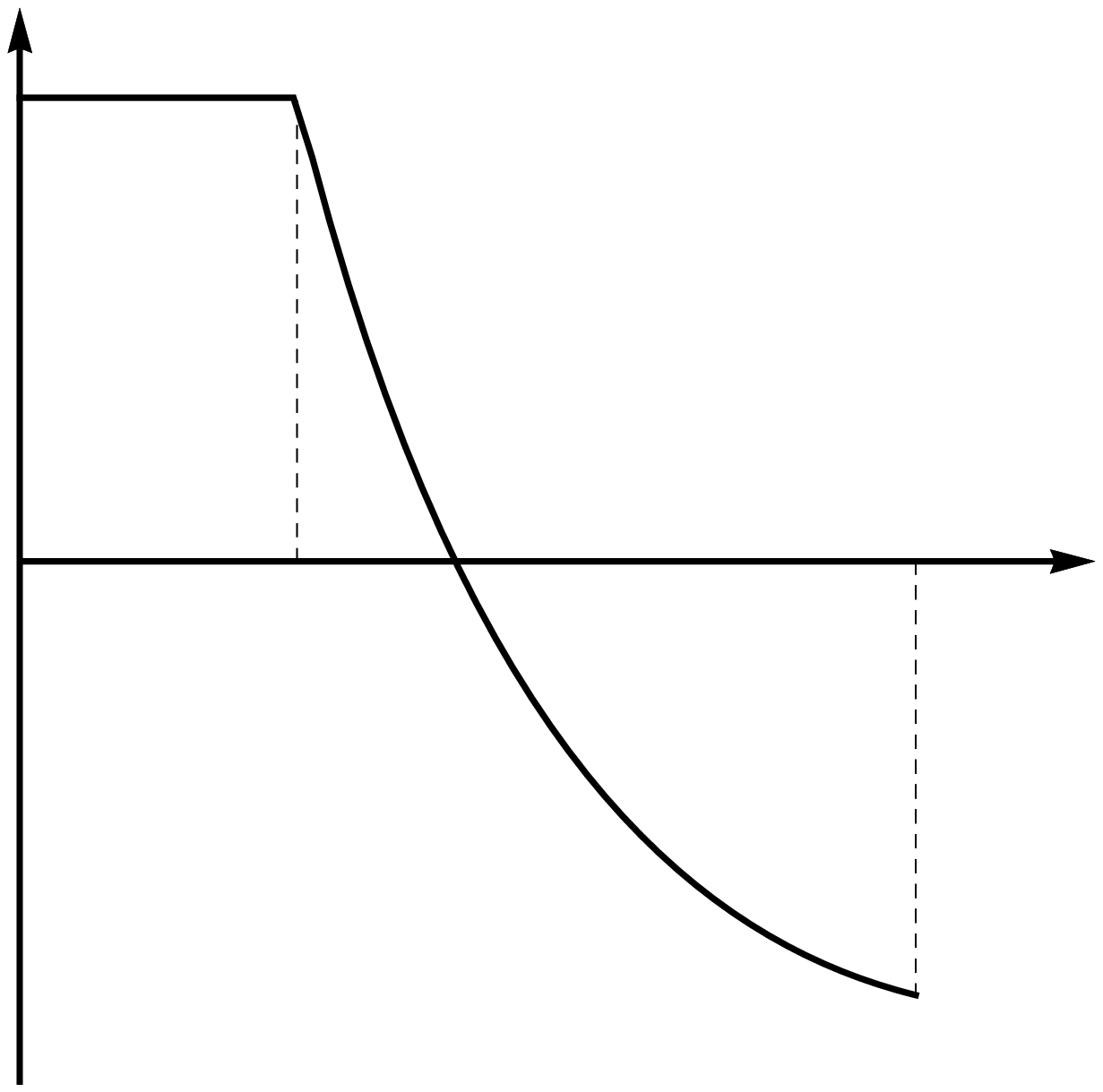}};
\node at (0,52) {$\check{\mathtt w}$};
\node at (46,7) {$w_F$};
\node at (10,50) {$w^+$};
\node at (52,-1) {\strut \color{white}{$\rho$}};
\node at (52,23) {\strut $v$};
\node at (15,23) {\strut $v_F^-$};
\end{tikzpicture}
\end{flushright}
\caption{Geometrical meaning of $\hat{\mathtt u}$ and $\check{\mathtt u}$ defined in \eqref{e:hat-check} in the case $F\in(f_{\rm c}^-,f_{\rm c}^+)$.
}
\label{f:uhatucheck1}
\end{figure}

\begin{figure}[!ht]
\begin{flushright}
\begin{tikzpicture}[every node/.style={anchor=south west,inner sep=0pt},x=1mm, y=1mm]
\node at (4,4) {\includegraphics[height=50mm]{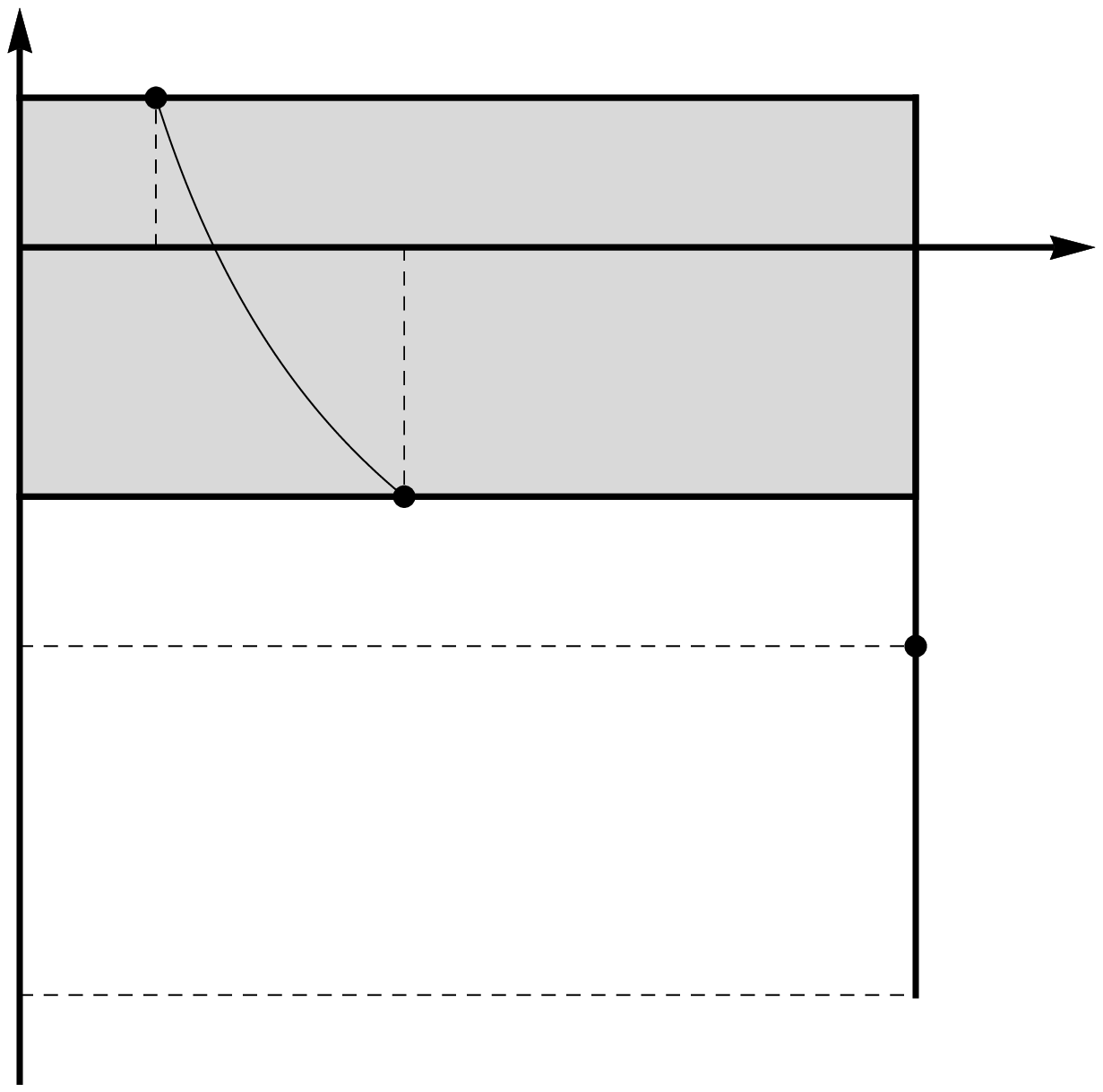}};
\node at (52,38) {$v$};
\node at (41,38) {$V$};
\node at (21,43) {$v_F^+$};
\node at (9,38) {$v_F^-$};
\node at (-1,30) {$w^-$};
\node at (-1,23) {$w_F$};
\node at (-7,8) {$w^--1$};
\node at (-1,48.5) {$w^+$};
\node at (0,52) {$w$};
\node at (52,-1) {\strut \color{white}{$\rho$}};
\end{tikzpicture}
\begin{tikzpicture}[every node/.style={anchor=south west,inner sep=0pt},x=1mm, y=1mm]
\node at (4,4) {\includegraphics[height=50mm]{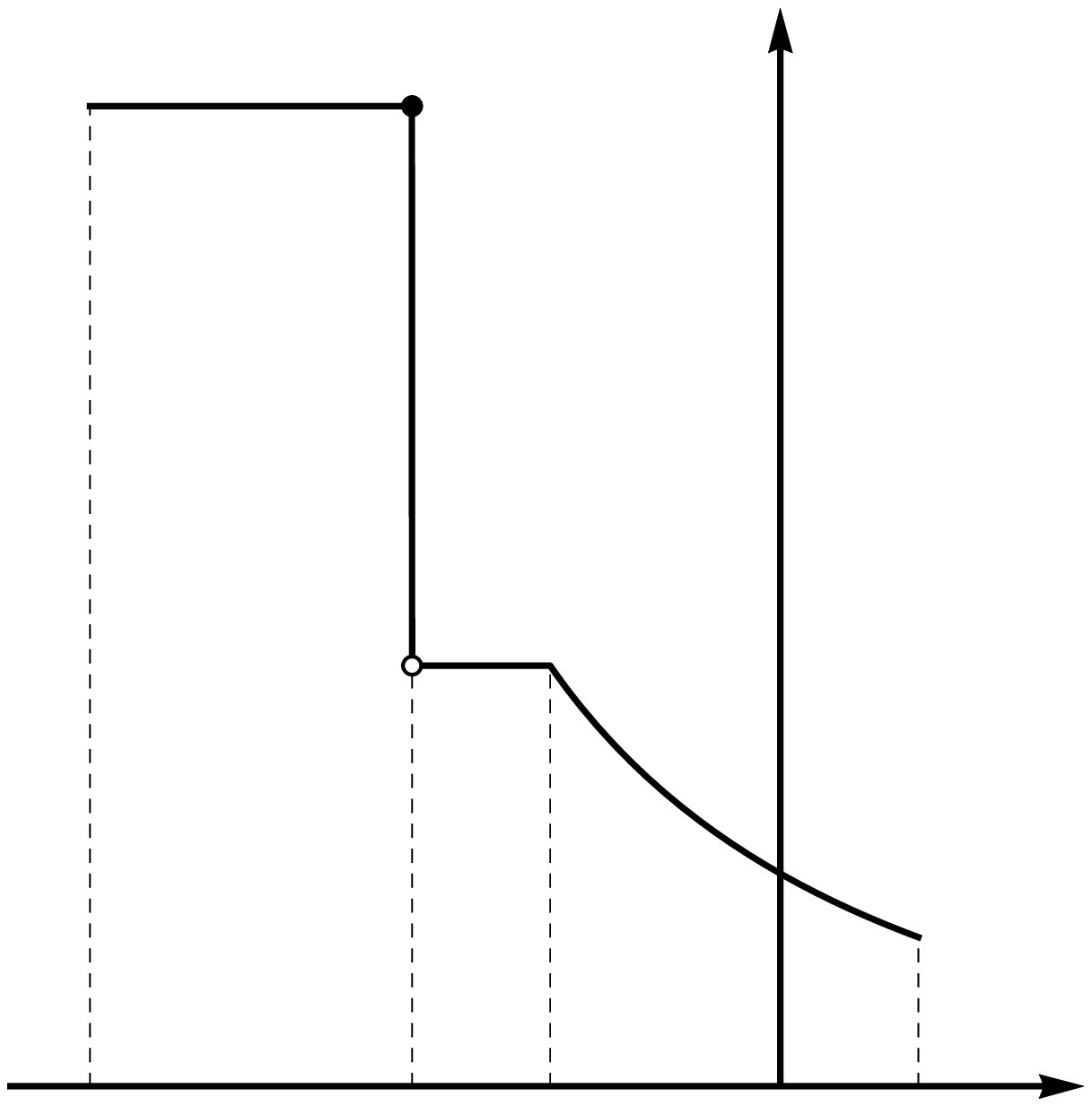}};
\node at (21,-1) {\strut $w_F$};
\node at (28,-1) {\strut $w^-$};
\node at (4,-1) {$\strut w^--1$};
\node at (44,-1) {\strut $w^+$};
\node at (52,-1) {\strut $w$};
\node at (35,52) {$\hat{\mathtt v}$};
\node at (48,10) {$v_F^-$};
\node at (25,24) {$v_F^+$};
\node at (15,50) {$V$};
\end{tikzpicture}
\begin{tikzpicture}[every node/.style={anchor=south west,inner sep=0pt},x=1mm, y=1mm]
\node at (4,4) {\includegraphics[height=50mm]{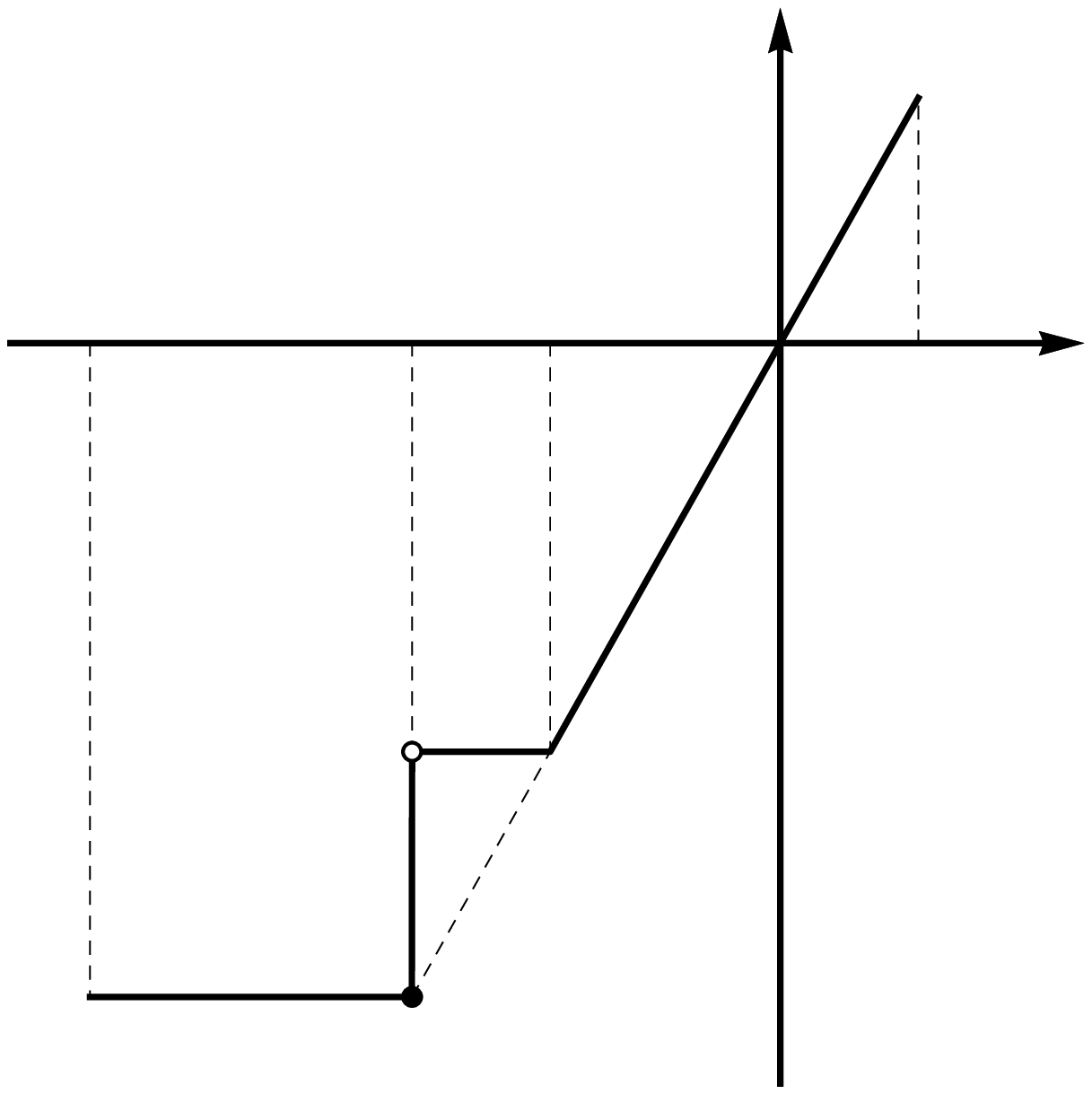}};
\node at (21,39) {\strut $w_F$};
\node at (28,39) {\strut $w^-$};
\node at  (4,39) {$\strut w^--1$};
\node at (44,32) {\strut $w^+$};
\node at (52,32) {\strut $w$};
\node at (35,52) {$\hat{\mathtt w}$};
\node at (13,10) {$w_F$};
\node at (24,21) {$w^-$};
\node at (48,50) {$w^+$};
\node at (52,-1) {\strut \color{white}{$\rho$}};
\end{tikzpicture}
\\
\begin{tikzpicture}[every node/.style={anchor=south west,inner sep=0pt},x=1mm, y=1mm]
\node at (4,4) {\includegraphics[height=50mm]{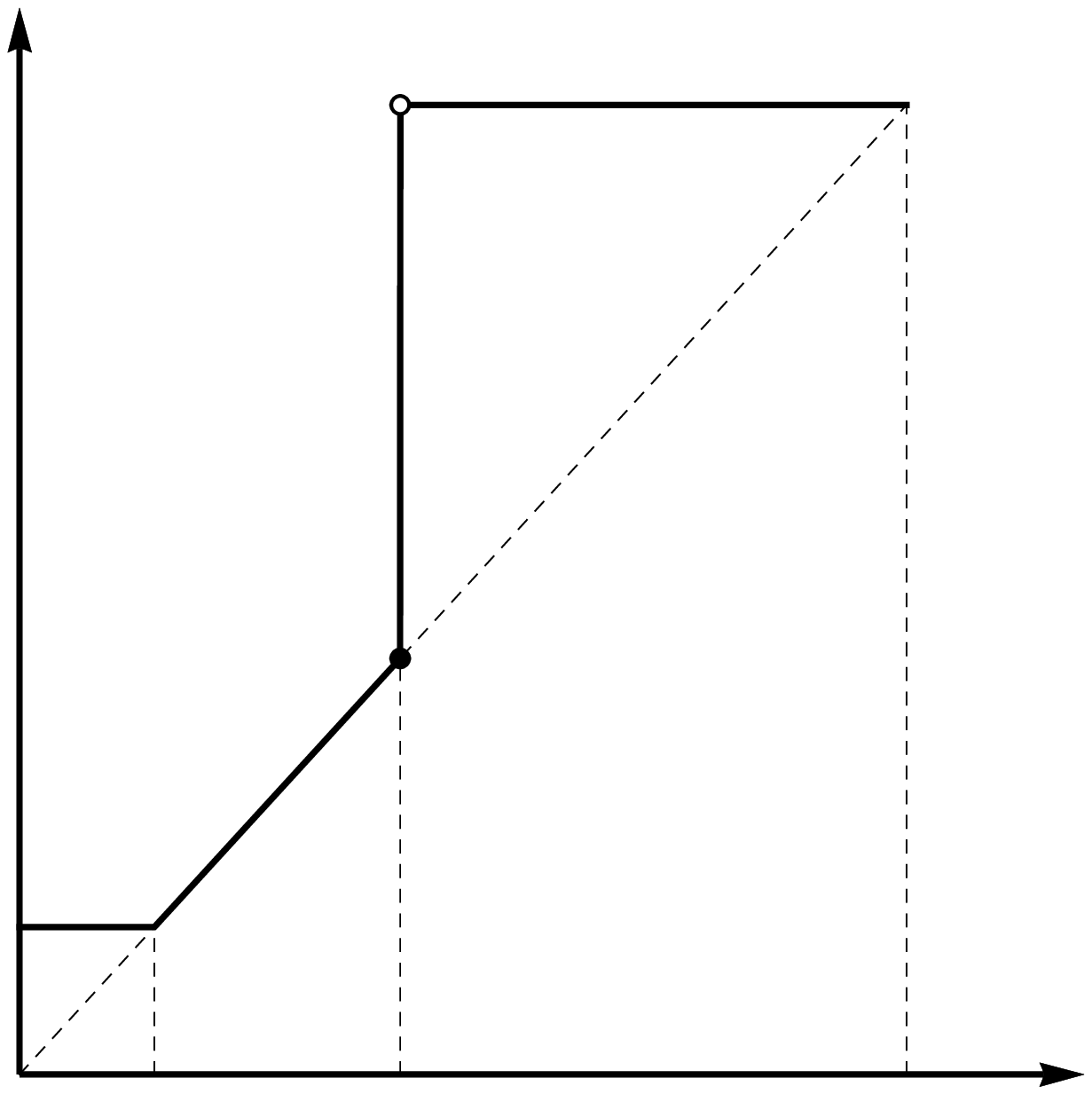}};
\node at (52,-1) {\strut $v$};
\node at (41,-1) {\strut $V$};
\node at (21,-1) {\strut $v_F^+$};
\node at  (9,-1) {\strut $v_F^-$};
\node at (0,52) {$\check{\mathtt v}$};
\node at (30,50) {$V$};
\node at (6.5,13) {$v_F^-$};
\node at (23.5,20) {$v_F^+$};
\end{tikzpicture}
\begin{tikzpicture}[every node/.style={anchor=south west,inner sep=0pt},x=1mm, y=1mm]
\node at (4,4) {\includegraphics[height=50mm]{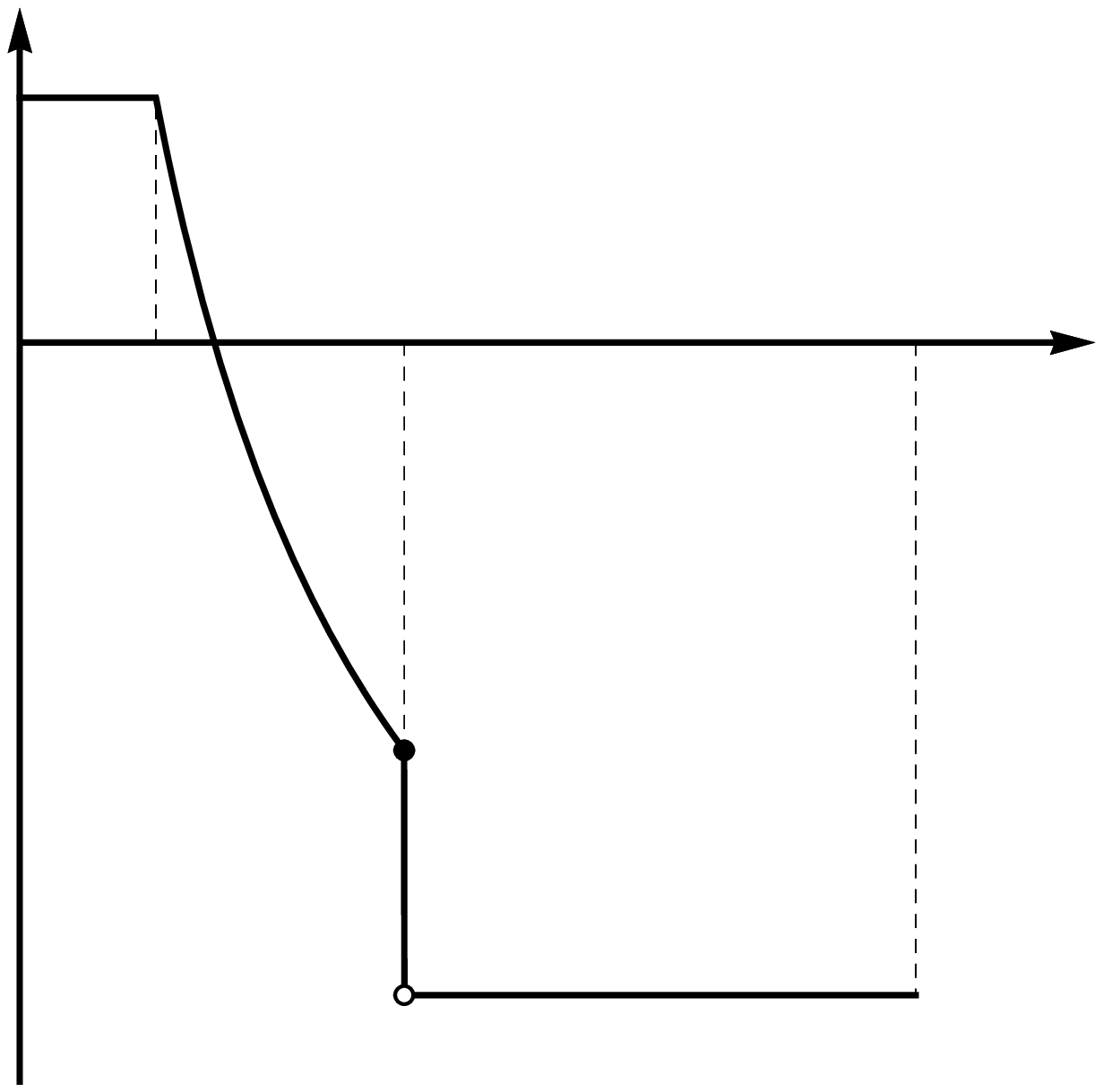}};
\node at (52,33) {\strut $v$};
\node at (44,39) {\strut $V$};
\node at (20,39) {\strut $v_F^+$};
\node at  (9,33) {\strut $v_F^-$};
\node at (0,52) {$\check{\mathtt w}$};
\node at (31,10) {$w_F$};
\node at (7,50) {$w^+$};
\node at (24,19) {$w^-$};
\node at (52,-1) {\strut \color{white}{$\rho$}};
\end{tikzpicture}
\end{flushright}
\caption{Geometrical meaning of $\hat{\mathtt u}$ and $\check{\mathtt u}$ defined in \eqref{e:hat-check} in the case $F\in(0,f_{\rm c}^-)$.
}
\label{f:uhatucheck2}
\end{figure}

Denote by $\tv_+$ and $\tv_-$ the positive and negative total variations, respectively.
For any $u\colon \R\to\Omega$ let
\begin{subequations}\label{e:upsilons}
\begin{align}
\hat{\Upsilon}(u)
&\doteq
\tv_+\Bigl(\hat{\mathtt v}\bigl(\mathtt{w}(u),F\bigr);(-\infty,0)\Bigr)
+\tv_-\Bigl(\hat{\mathtt w}\bigl(\mathtt{w}(u),F\bigr);(-\infty,0)\Bigr),
\\
\check{\Upsilon}(u)
&\doteq
\tv_+\Bigl(\check{\mathtt v}(v,F);(0,\infty)\Bigr)
+\tv_-\Bigl(\check{\mathtt w}(v,F);(0,\infty)\Bigr).
\end{align}
\end{subequations}
For any $u\in\Omega$ and $k\in[0,V]$ let
\begin{align*}
\mathtt{N}^k_F(u) &\doteq \begin{cases}
f(u) \left[\dfrac{k}{F}-\dfrac{1}{p^{-1}\bigl(\mathtt{W}(u)-k\bigr)}\right]_+ & \hbox{if } F \neq 0,\\
k & \hbox{if } F = 0,
\end{cases}&
[w]_+ &\doteq
\begin{cases}
w&\text{if }w>0,
\\
0&\text{otherwise}.
\end{cases}
\end{align*}
We are now in the position to state the main result of the paper. 
\begin{theorem}\label{t:mainF}
Let $u^o \in \L1\cap\BV(\R;\Omega)$ and $F\in [0,f_{\rm c}^+]$ satisfy one of the following conditions:
\begin{enumerate}[label={\textbf{(H.\arabic*)}},leftmargin=*]\setlength{\itemsep}{0cm}%
\item\label{H.1} $F\in[f_{\rm c}^-,f_{\rm c}^+]$;
\item\label{H.2} $F\in[0,f_{\rm c}^-)$ and $\hat{\Upsilon}(u^o) + \check{\Upsilon}(u^o)$ is bounded.
\end{enumerate}
Then the approximate solutions $u_n$ constructed in Section~\ref{s:approxsolR} converge to a solution $u \in \C0(\R_+;\BV(\R;\Omega))$ of constrained Cauchy problem \eqref{eq:system}, \eqref{eq:initdat}, \eqref{eq:const} in the sense of Definition~\ref{def:solconstmod}.
Moreover for all $t$, $s\in\R_+$ the following estimates hold
\begin{align}\label{e:estimates}
\tv\bigl(u(t)\bigr) &\leq C_F^o,&
\|u(t)-u(s)\|_{\L1(\R;\Omega)} &\leq L_F^o\, |t-s|,&
\|u(t)\|_{\L\infty(\R;\Omega)} &\leq R+V,&
\end{align}
where $C^o_F$ and $L^o_F$ are constants that depend on $u^o$ and $F$.
Furthermore,  non-classical discontinuities of $u$ can occur only at the constraint location $x=0$, and if for any $k \in [0,V]$ and $\phi \in \Cc\infty((0,\infty)\times \R;\R)$ such that $\phi \geq0$ we have
\begin{equation}\label{e:Opeth}
\lim_{n\to\infty} \int_0^T \mathtt{N}^k_F\bigl(u_n(t,0_-)\bigr) \, \phi(t,0) \, \d t = 
\int_0^T \mathtt{N}^k_F\bigl(u(t,0_-)\bigr) \, \phi(t,0) \, \d t,
\end{equation}
then the (density) flow at $x=0$ is the maximal flow $F$ allowed by the constraint.
\end{theorem}
\noindent
As in~\cite{BCM2order, BenyahiaRosini01, ColomboGoatinPriuli}, the proof of the above theorem is based on the wave-front tracking algorithm, see~\cite{bressanbook, HoldenRisebroWFT} and the references therein.
The details of the proof are deferred to Section~\ref{s:mainF}.

\begin{remark}
If $F  \in [f_{\rm c}^-,f_{\rm c}^+]$, then $w\mapsto\hat{\mathtt u}(w,F)$ and $v\mapsto\check{\mathtt u}(v,F)$ are Lipschitz continuous and therefore $\hat{\Upsilon}(u^o) + \hat{\Upsilon}(u^o)$ is obviously bounded if $u^o$ has bounded total variation.
\end{remark}

\subsection{The constrained Riemann problem}\label{s:RiemannSs}

For completeness, we conclude this section by giving the definitions of the Riemann solvers $\mathcal{R}$ and $\mathcal{R}_F$ introduced in~\cite{BenyahiaRosini01} and~\cite{edda-nikodem-mohamed}, associated to Riemann problem \eqref{eq:system}, \eqref{eq:Rdata} and to constrained Riemann problem \eqref{eq:system}, \eqref{eq:const}, \eqref{eq:Rdata}, respectively, and used in Section~\ref{s:mainF} to prove Theorem~\ref{t:mainF}.

We recall that Riemann problems for \eqref{eq:system} are Cauchy problems with initial condition of the form
\begin{equation}\label{eq:Rdata}
u(0,x)=\begin{cases}
u_\ell &\hbox{if }x<0,\\ 
u_r &\hbox{if }x>0.
\end{cases}
\end{equation}

\begin{definition}
The Riemann solver $\mathcal{R} \colon \Omega^2 \to {\mathbf{L^{\infty}}}(\mathbb{R};\Omega)$  associated to Riemann problem \eqref{eq:system}, \eqref{eq:Rdata} is defined as follows.

\begin{enumerate}[label={(R.\arabic*)},leftmargin=*]\setlength{\itemsep}{0cm}%

\item If $u_\ell,u_r\in \Omega_{\rm f}$, then $\mathcal{R}[u_\ell,u_r]$ consists of a contact discontinuity $(u_\ell,u_r)$ with speed of propagation $V$.

\item If $u_\ell,u_r\in\Omega_{\rm c}$, then $\mathcal{R}[u_\ell,u_r]$ consists of a $1$-wave $(u_\ell,{\mathtt u}_*(u_\ell,u_r))$ and of a $2$-contact discontinuity $({\mathtt u}_*(u_\ell,u_r),u_r)$.

\item If $u_\ell\in\Omega_{\rm c}^-$ and $u_r\in\Omega_{\rm f}^-$, then $\mathcal{R}[u_\ell,u_r]$ consists of a $1$-rarefaction $(u_\ell,\omega(u_\ell))$ and a contact discontinuity $(\omega(u_\ell),u_r)$.

\item If $u_\ell\in\Omega_{\rm f}^-$ and $u_r\in \Omega_{\rm c}^-$, then $\mathcal{R}[u_\ell,u_r]$ consists of a phase transition $(u_\ell,\mathtt{v}^-(u_r))$ and a $2$-contact discontinuity $(\mathtt{v}^-(u_r),u_r)$.

\end{enumerate}
\end{definition}

Since $(t,x)\mapsto\mathcal{R}[u_\ell,u_r](x/t)$ does not in general satisfy constraint condition \eqref{eq:const}, we introduce
\begin{align*}
\mathcal{D}_1 \doteq \hphantom{\cup}&\,
\bigl\{ (u_\ell,u_r) \in \Omega\times\Omega : f\bigl(\mathcal{R}[u_\ell,u_r](t,0_\pm)\bigr) \le F \bigr\}
\\=\hphantom{\cup}
&\,\bigl\{ (u_\ell , u_r) \in \Omega_{\rm f}\times\Omega_{\rm f} : f(u_\ell) \le F \bigr\}
\\
\cup&
\,\bigl\{ (u_\ell , u_r) \in \Omega_{\rm c} \times \Omega : f\bigl({\mathtt u}_*(u_\ell,u_r)\bigr) \le F \bigr\}
\\
\cup&
\,\bigl\{ (u_\ell , u_r) \in \Omega_{\rm f}^- \times \Omega_{\rm c}^- : \min\bigl\{f(u_\ell),f\bigl(\mathtt{v}^-(u_r)\bigr)\bigr\} \le F \bigr\},
\end{align*}
$\mathcal{D}_2 \doteq \Omega^2 \setminus \mathcal{D}_1$ and the constrained Riemann solver $\mathcal{R}_F$ in the following

\begin{definition}\label{def:01}
The constrained Riemann solver $\mathcal{R}_F \colon \Omega^2 \to {\mathbf{L^{\infty}}}(\mathbb{R};\Omega)$ associated to constrained Riemann problem \eqref{eq:system}, \eqref{eq:const}, \eqref{eq:Rdata} is defined as
\[
\mathcal{R}_F[u_\ell,u_r](x)\doteq \begin{cases}
\mathcal{R}[u_\ell,u_r](x) &\hbox{if }(u_\ell,u_r) \in \mathcal{D}_1,\\[5pt]
\begin{cases}
\mathcal{R}[u_\ell,\hat{\mathtt u}_\ell](x) &\hbox{if }x<0,\\
\mathcal{R}[\check{\mathtt u}_r,u_r](x) &\hbox{if }x>0,
\end{cases}
&\hbox{if }(u_\ell,u_r) \in \mathcal{D}_2,
\end{cases}
\]
where $\hat{\mathtt u}_\ell \doteq \hat{\mathtt u}(\mathtt{w}(u_\ell),F) \in \Omega_{\rm c}$ and $\check{\mathtt u}_r \doteq \check{\mathtt u}(v_r,F)  \in \Omega$ are defined by \eqref{e:hat-check}.
\end{definition}
\noindent
In \figurename~\ref{f:uhatucheck0} we clarify the selection criterion \eqref{e:hat-check} for $\hat{\mathtt u}_\ell$ and $\check{\mathtt u}_r$.
\begin{figure}[!ht]
\begin{center}
\begin{tikzpicture}[every node/.style={anchor=south west,inner sep=0pt},x=1mm, y=1mm]
\node at (4,4) {\includegraphics[height=50mm]{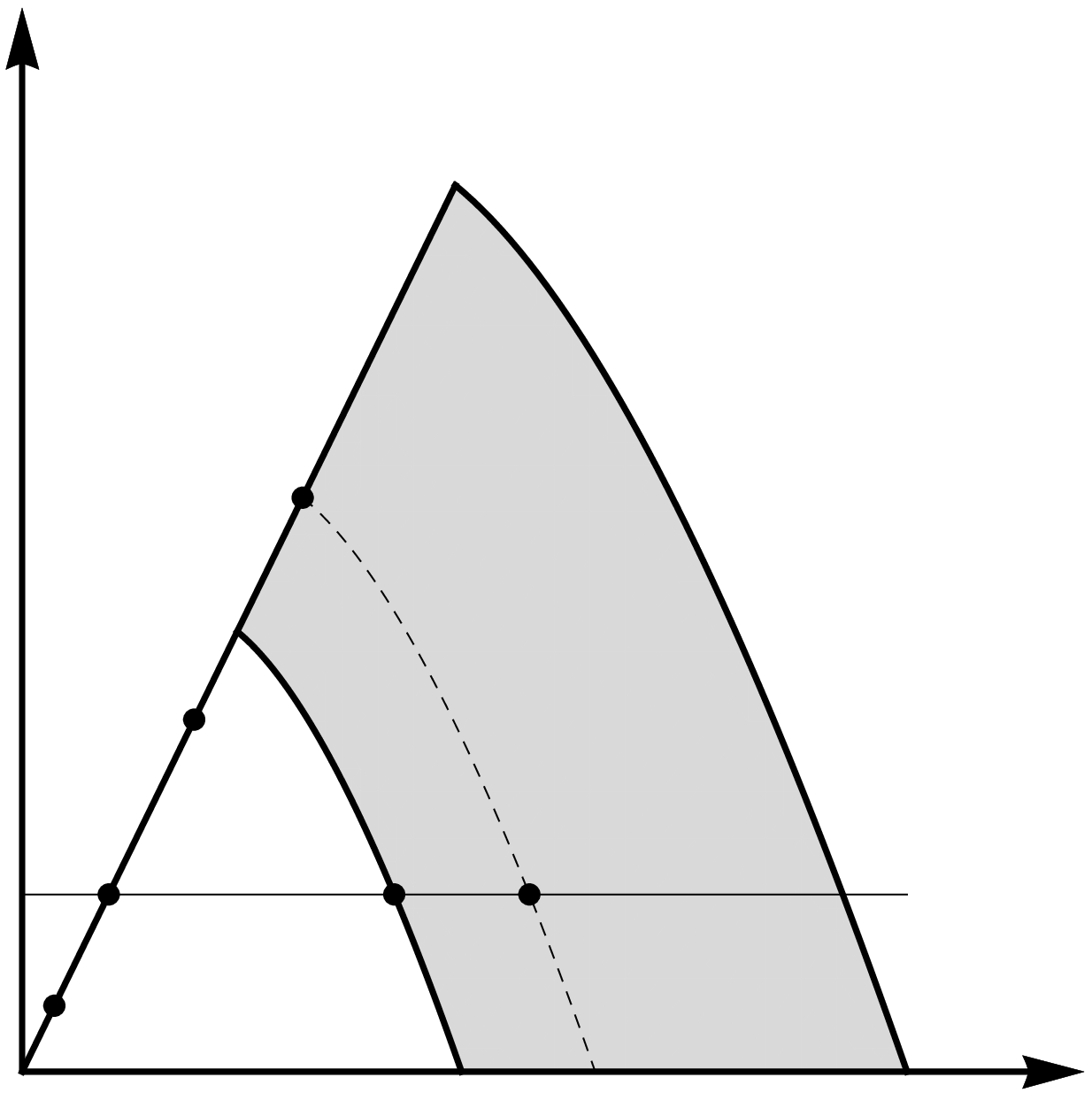}};
\node at (52,0) {$\rho$};
\node at (0,50) {$f$};
\node at (0,12) {$F$};
\node at (6,14.5) {$\check{\mathtt u}_r$};
\node at (18,8) {$\hat{\mathtt u}_\ell^1$};
\node at (29,13.5) {$\hat{\mathtt u}_\ell^2$};
\node at (8,20) {$u_\ell^1$};
\node at (12.5,30) {$u_\ell^2$};
\node at (8,6) {$u_r$};
\end{tikzpicture}\qquad
\begin{tikzpicture}[every node/.style={anchor=south west,inner sep=0pt},x=1mm, y=1mm]
\node at (4,4) {\includegraphics[height=50mm]{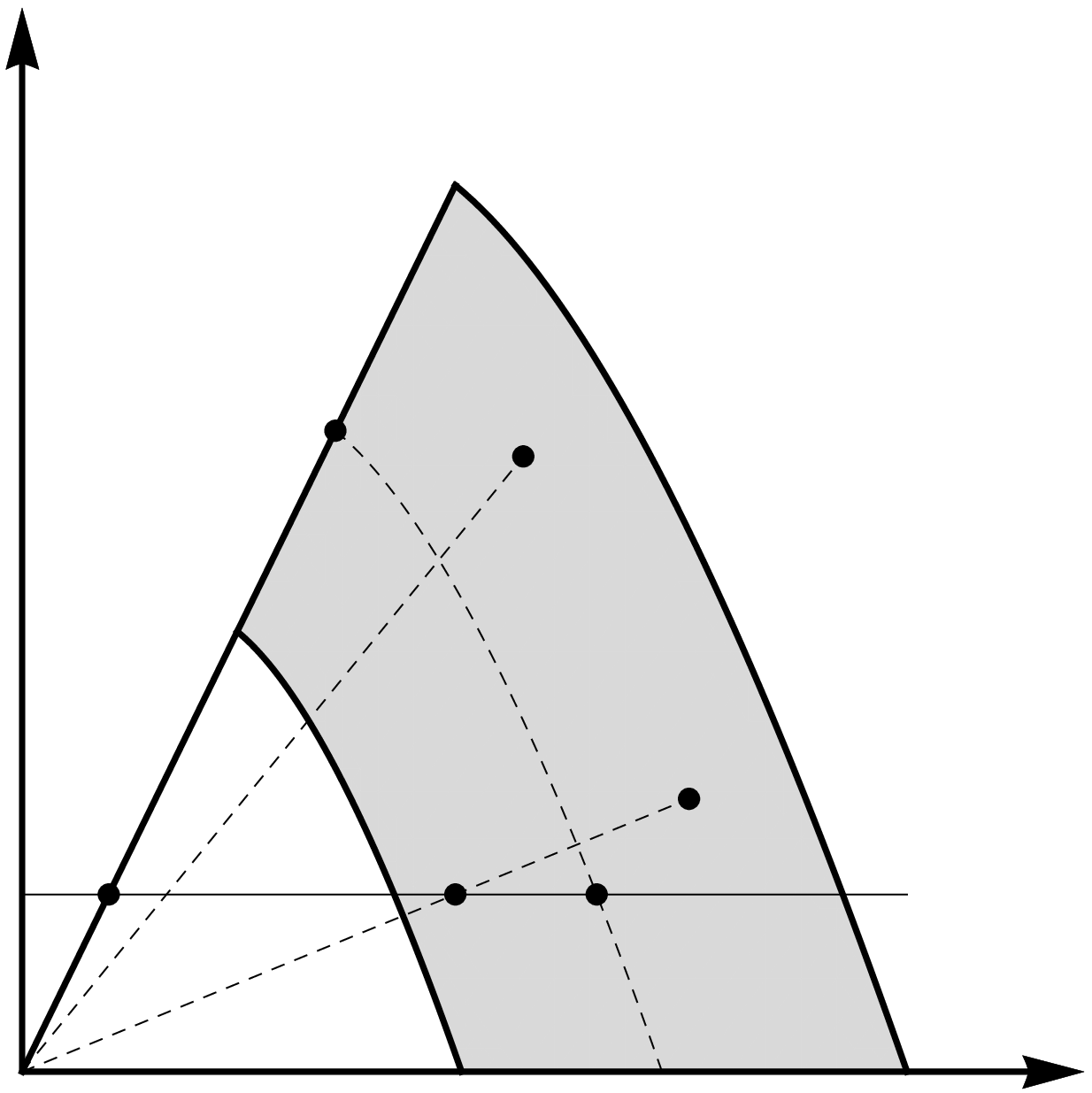}};
\node at (52,0) {$\rho$};
\node at (0,50) {$f$};
\node at (0,12) {$F$};
\node at (6,14.5) {$\check{\mathtt u}_r^2$};
\node at (25,8.5) {\strut $\check{\mathtt u}_r^1$};
\node at (32,8.5) {\strut $\hat{\mathtt u}_\ell$};
\node at (14.5,34) {$u_\ell$};
\node at (35,18) {$u_r^1$};
\node at (27,33.5) {$u_r^2$};
\end{tikzpicture}
\\[10pt]
\begin{tikzpicture}[every node/.style={anchor=south west,inner sep=0pt},x=1mm, y=1mm]
\node at (4,4) {\includegraphics[height=50mm]{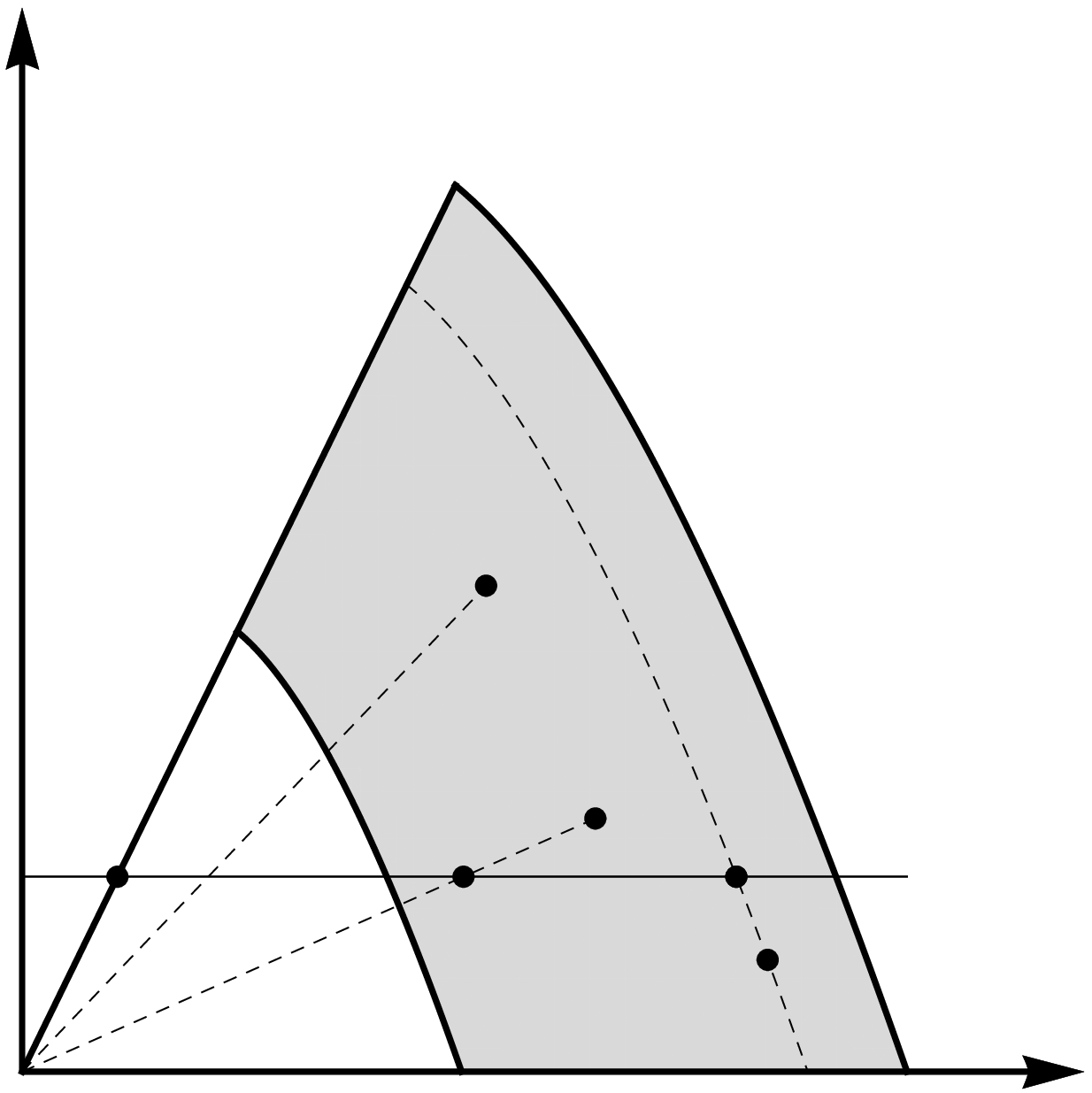}};
\node at (52,0) {$\rho$};
\node at (0,50) {$f$};
\node at (0,13) {$F$};
\node at (6,14.5) {$\check{\mathtt u}_r^2$};
\node at (25,8.5) {$\check{\mathtt u}_r^1$};
\node at (37.5,14) {$\hat{\mathtt u}_\ell$};
\node at (35,7) {$u_\ell$};
\node at (30,17) {$u_r^1$};
\node at (25.5,27.5) {$u_r^2$};
\end{tikzpicture}\qquad
\begin{tikzpicture}[every node/.style={anchor=south west,inner sep=0pt},x=1mm, y=1mm]
\node at (4,4) {\includegraphics[height=50mm]{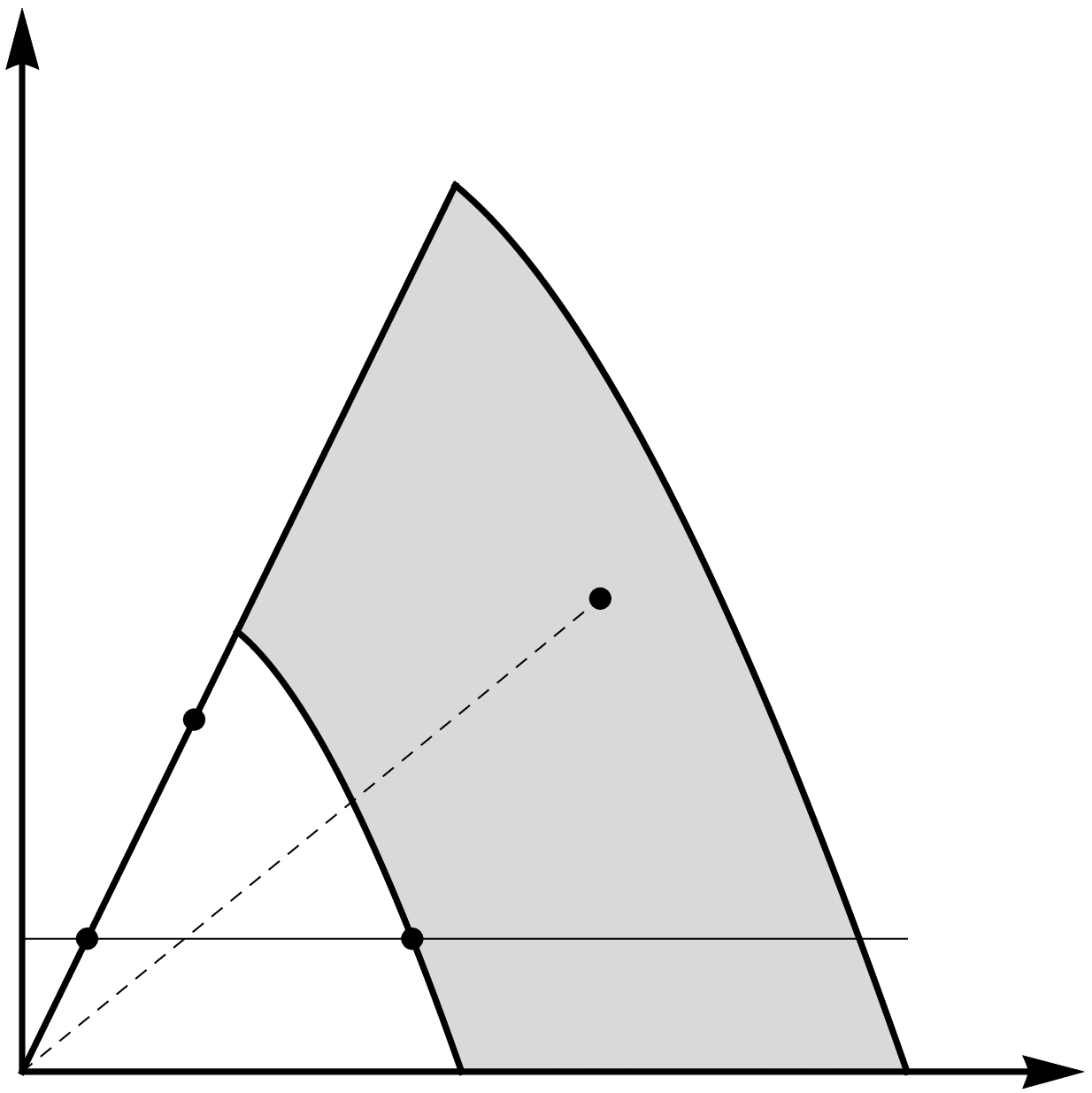}};
\node at (52,0) {$\rho$};
\node at (0,50) {$f$};
\node at (0,10) {$F$};
\node at (8,20) {$u_\ell$};
\node at (23,12) {$\hat{\mathtt u}_\ell$};
\node at (30,28) {$u_r$};
\node at (5.5,13) {$\check{\mathtt u}_r$};
\end{tikzpicture}
\end{center}
\caption{The selection criterion \eqref{e:hat-check} for $\hat{\mathtt u}_\ell \doteq \hat{\mathtt u}(\mathtt{w}(u_\ell),F)$ and $\check{\mathtt u}_r \doteq \check{\mathtt u}(v_r,F)$ exploited in Definition~\ref{def:01} in the case $(u_\ell,u_r)\in\mathcal{D}_2$ and $F \in (0,f_{\rm c}^-)$.
In the first picture $u_\ell^1$, $u_\ell^2$ represent the left state in two different cases and $\hat{\mathtt u}_\ell^1$, $\hat{\mathtt u}_\ell^2$ are the corresponding $\hat{\mathtt u}_\ell$.
Analogously in the second and third pictures for $u_r^1$, $u_r^2$ and $\check{\mathtt u}_r^1$, $\check{\mathtt u}_r^2$.
}
\label{f:uhatucheck0}
\end{figure}
We point out that $\hat{\mathtt u}_\ell$ and $\check{\mathtt u}_r$ satisfy the following general properties.
\[
\begin{minipage}{.43\textwidth}
If $(u_\ell,u_r) \in \mathcal{D}_2$, then $\mathtt{w}(u_\ell) > \mathtt{w}(\check{\mathtt u}_r)$ and $v_r > \hat{\mathtt v}_\ell$.\\
If $(u_\ell,u_r) \in \mathcal{D}_2$ and $u_\ell \in \Omega_{\rm f}^-$, then $\mathtt{w}(\hat{\mathtt u}_\ell) = w^-$.\\
If $(u_\ell,u_r) \in \mathcal{D}_2$ and $u_r \in \Omega_{\rm f}$, then $\check{\mathtt v}_r = V$.
\end{minipage}
\]

It is easy to prove that $(t,x)\mapsto\mathcal{R}[u_\ell,u_r](x/t)$ and $(t,x)\mapsto\mathcal{R}_F[u_\ell,u_r](x/t)$ are solutions to Riemann problems \eqref{eq:system}, \eqref{eq:Rdata} and \eqref{eq:system}, \eqref{eq:const}, \eqref{eq:Rdata} in the sense of Definitions~\ref{def:solunconstmod} and~\ref{def:solconstmod}, respectively.

We recall that both $\mathcal{R}$ and $\mathcal{R}_F$ are ${\mathbf{L^{1}_{loc}}}$-continuous, see~\cite[Propositions~2 and~3]{edda-nikodem-mohamed}.

\section{Example}\label{s:exPT}

In this section we apply model \eqref{eq:system}, \eqref{eq:initdat}, \eqref{eq:const} to simulate the traffic across, for instance, a toll gate located at $x = 0$ and with capacity $F$.
More specifically, let $w^-$ and $w^+$ be the Lagrangian markers corresponding to vehicles that are initially at rest in $[x_A,x_B)$ and $[x_B,0)$, respectively.
The resulting initial condition is
\[
u^o(x) \doteq
\begin{cases}
u_\ell&\text{if }x \in [x_A,x_B),\\
u_r&\text{if }x \in [x_B,0),\\
u_0&\text{if }x \in{{\mathbb{R}}} \setminus [x_A,0),
\end{cases}
\]
where $u_0 \doteq (0,V)$, $u_\ell \doteq (p^{-1}(w^-),0)$ and $u_r \doteq (p^{-1}(w^+),0)$, see \figurename~\ref{f:PTfundiawaves}.

\begin{figure}[!ht]
\begin{center}
\begin{tikzpicture}[every node/.style={anchor=south west,inner sep=0pt},x=1mm, y=1mm]
\node at (4,4) {\includegraphics[height=50mm]{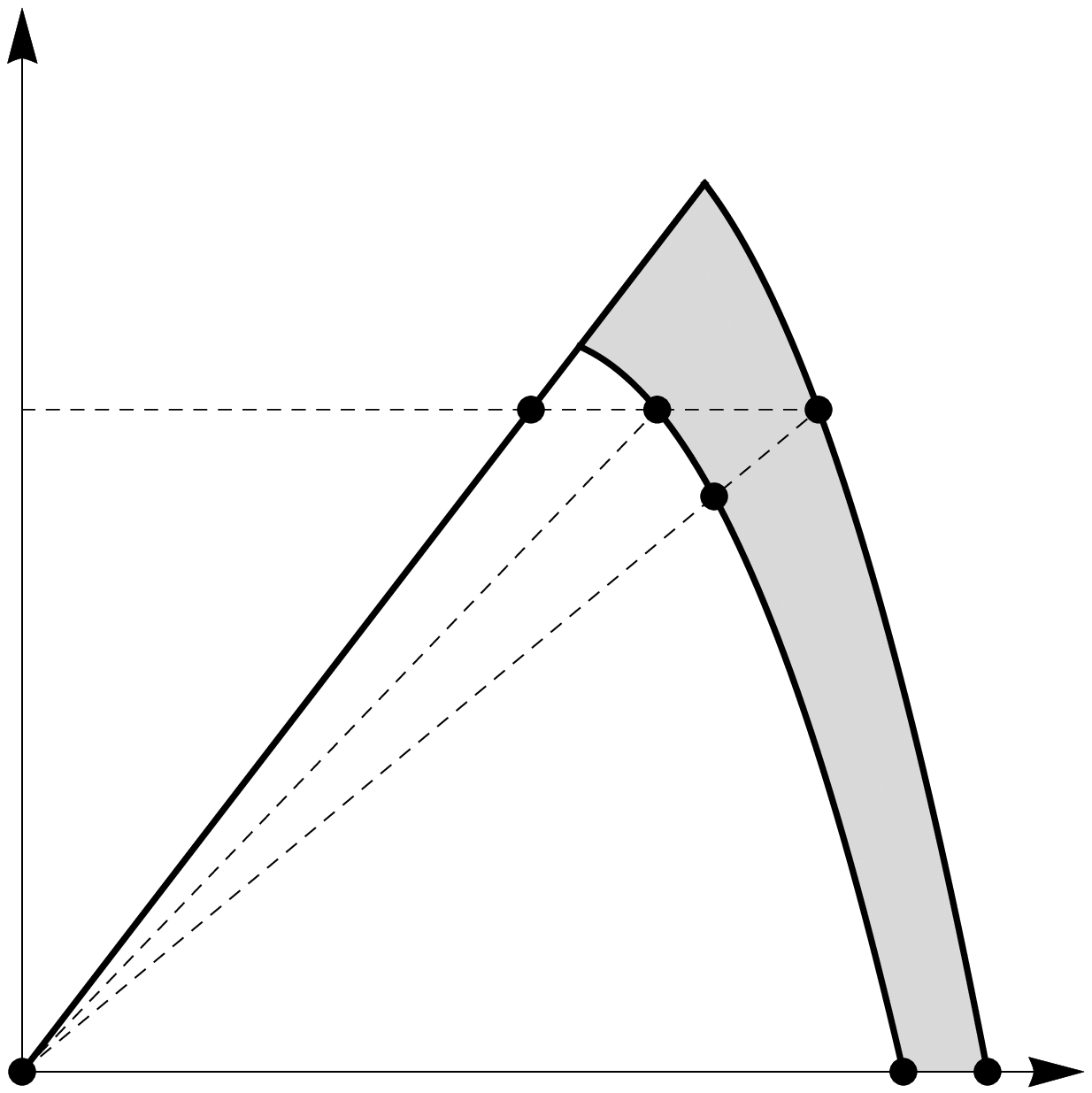}};
\node at (52,-1) {\strut $\rho$};
\node at (3,-1) {\strut $u_0$};
\node at (43,-1) {\strut $u_\ell$};
\node at (48,-1) {\strut $u_r$};
\node at (34.3,26.5) {${\mathtt u}_*$};
\node at (26,35) {\strut $\check{\mathtt u}$};
\node at (42,35) {\strut $\hat{\mathtt u}_r$};
\node at (33,35) {\strut $\hat{\mathtt u}_\ell$};
\node at (0,50) {$f$};
\node at (0,34) {$F$};
\end{tikzpicture}
\qquad
\begin{tikzpicture}[every node/.style={anchor=south west,inner sep=0pt},x=1mm, y=1mm]
\node at (4,4) {\includegraphics[height=50mm]{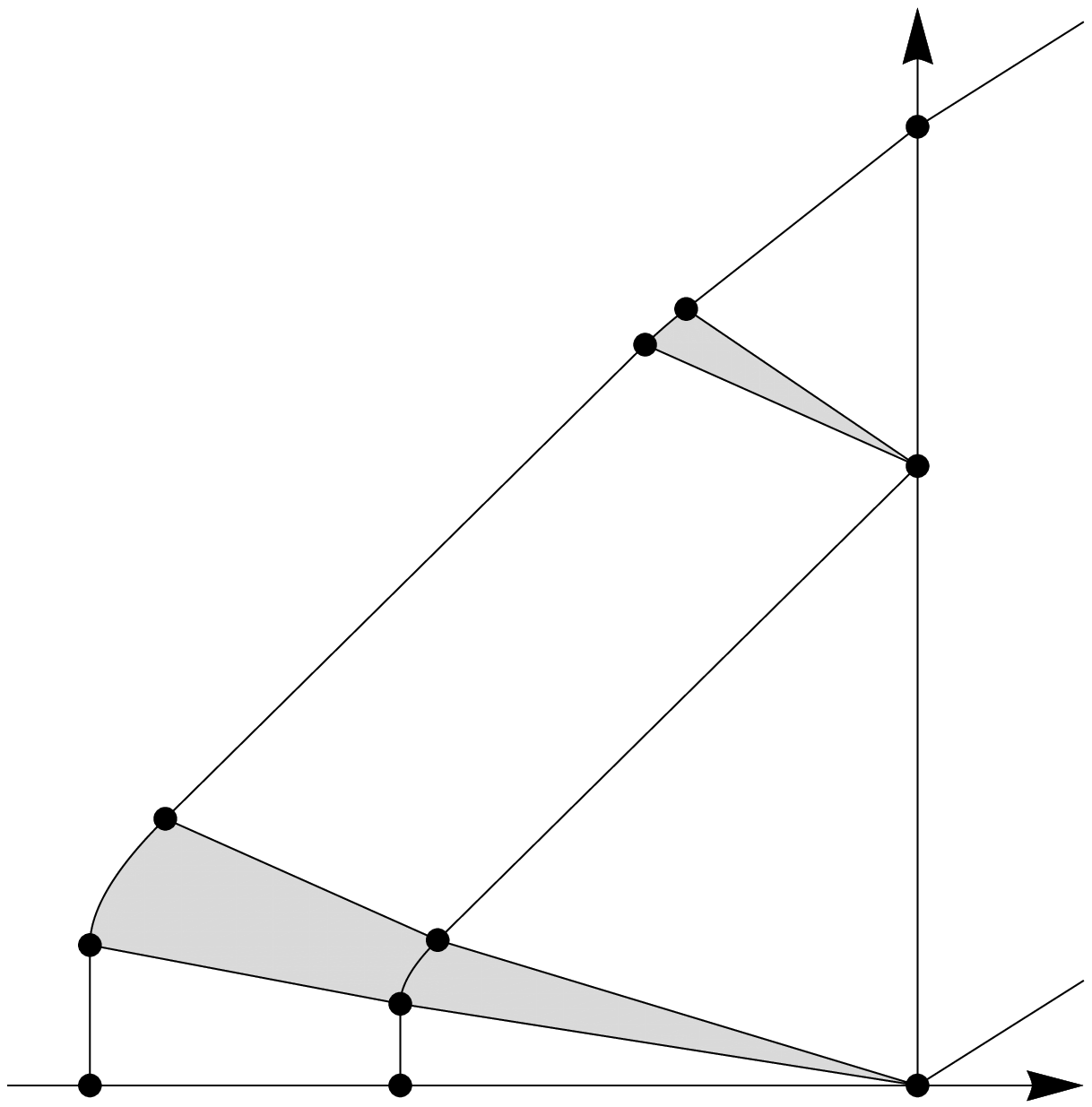}};
\node at (52,-1) {\strut $x$};
\node at (41.2,50) {$t$};
\node at (5,4.2) {\strut $A$};
\node at (19.5,4.2) {\strut $B$};
\node at (19.7,9.1) {$C$};
\node at (5,11.5) {$D$};
\node at (23,12.5) {$E$};
\node at (8,17) {$F$};
\node at (46.7,31.5) {$G$};
\node at (29.2,38) {$H$};
\node at (33.2,40) {$I$};
\node at (46.7,45.6) {$L$};
\node at (12,35) {$u_0$};
\node at (13,6) {$u_\ell$};
\node at (24.2,5.3) {$u_r$};
\node at (37,15) {$\hat{\mathtt u}_r$};
\node at (27,24) {${\mathtt u}_*$};
\node at (41,39) {$\hat{\mathtt u}_\ell$};
\node at (51,25) {$\check{\mathtt u}$};
\end{tikzpicture}
\end{center}
\caption{Notations used to describe the solution constructed in Section~\ref{s:exPT}.}
\label{f:PTfundiawaves}
\end{figure}

The resulting solution can be constructed by solving the Riemann problems corresponding to the discontinuities of $u^o$ and by considering the interactions of the waves between themselves or with the point constraint $x=0$.
We describe below the solution and its construction in more details.
Let
\begin{align*}
\hat{\mathtt u}_\ell &\doteq \hat{\mathtt u}(w^-,F),&
\hat{\mathtt u}_r &\doteq \hat{\mathtt u}(w^+,F),&
\check{\mathtt u} &\doteq \check{\mathtt u}(V,F),&
{\mathtt u}_* &\doteq {\mathtt u}_*(u_\ell,\hat{\mathtt u}_r).
\end{align*}
At $x=0$ we apply $\mathcal{R}_F$ and obtain a backward rarefaction ${\rm R}_0(u_r,\hat{\mathtt u}_r)$, a stationary non-classical shock ${\rm NS}_0(\hat{\mathtt u}_r,\check{\mathtt u})$ and a forward contact discontinuity ${\rm CD}_0(\check{\mathtt u},u_0)$, which moves with speed $V$.
At $x=x_B$ we apply $\mathcal{R}$ and obtain a stationary contact discontinuity ${\rm CD}_B(u_\ell,u_r)$.
Let $C$ and $E$ be the starting and final interaction points between ${\rm CD}_B$ and ${\rm R}_0$.
During such interaction we have that ${\rm CD}_B$ accelerates, while ${\rm R}_0$ crosses ${\rm CD}_B$ and eventually changes its values.
After time $t=t_E$ we have that ${\rm CD}_B$ moves with speed $\hat{\mathtt v}_r > 0$ and interacts with ${\rm NS}_0$ at $G$.
At $G$ we apply $\mathcal{R}_F$ and obtain a backward rarefaction ${\rm R}_G(\mathtt{u}_*,\hat{\mathtt u}_\ell)$ and a stationary non-classical shock ${\rm NS}_G(\hat{\mathtt u}_\ell,\check{\mathtt u})$.

At $x=x_A$ we apply $\mathcal{R}$ and obtain a stationary phase transition ${\rm PT}_A(u_0,u_\ell)$.
Let $D$ and $F$ be the starting and final interaction points between ${\rm PT}_A$ and ${\rm R}_0$.
During the time interval $(t_D,t_F)$ we have that ${\rm PT}_A$ accelerates and ${\rm R}_0$ starts to disappear.
After time $t=t_F$ we have that ${\rm PT}_A$ moves with speed $\hat{\mathtt v}_r > 0$.
Let $H$ and $I$ be the starting and final interaction points between ${\rm PT}_A$ and ${\rm R}_G$.
Then, during the time interval $(t_H,t_I)$ we have that ${\rm PT}_A$ accelerates and ${\rm R}_G$ starts to disappear.
After time $t=t_I$ we have that ${\rm PT}_A$ moves with speed $\hat{\mathtt v}_\ell > 0$.
Finally, ${\rm PT}_A$ interacts with ${\rm NS}_G$ at $L$ and then moves with speed $V$.

In \figurename~\ref{f:PTsol} we represent in different coordinates the quantitative evolution of the solution corresponding to $p(\rho) \doteq \rho^2$ and to the data
\begin{align*}
&x_A = -8,&
&x_B = -5,&
&w^-=1,&
&w^+=6/5,&
&V=3/5,&
&F=\sqrt{3}/5.
\end{align*}
Such solution is obtained by the explicit analysis of the wave-fronts interactions with computer-assisted computation of the interaction times and front slopes.
\begin{figure}[!ht]
\centering{
\begin{tabular}{c@{\qquad}c}
\begin{tikzpicture}[every node/.style={anchor=south west,inner sep=0pt},x=1mm, y=1mm]
\node at (4,4) {\includegraphics[height=50mm]{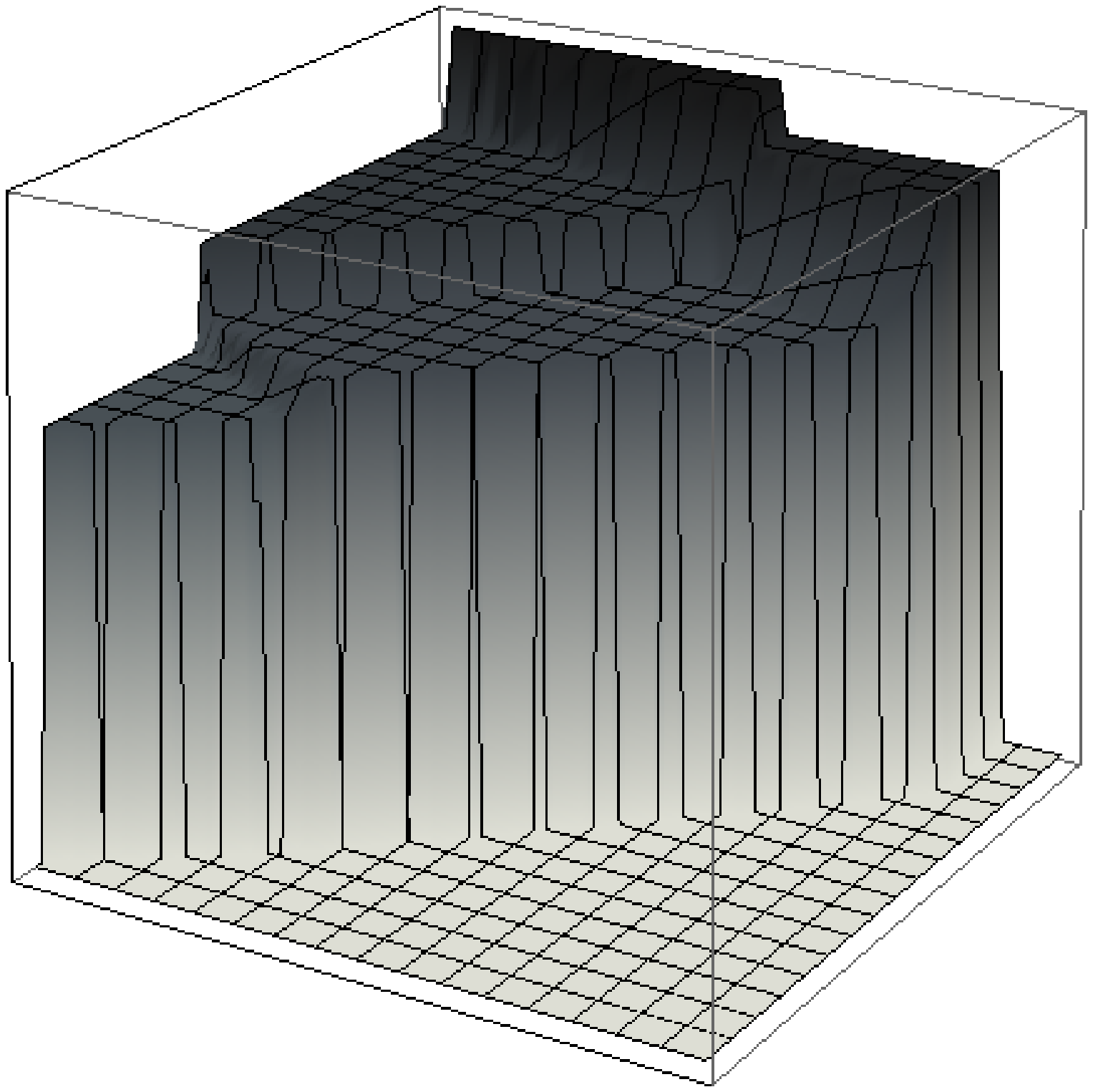}};
\node at (8,10) {$0$};
\node at (14,8) {$x_H$};
\node at (22,6) {$x_B$};
\node at (31,4) {$x_A$};
\node at (5,28) {$\check{\mathtt r}$};
\node at (5,32) {$\hat{\mathtt r}_\ell$};
\node at (5,36) {$\hat{\mathtt r}_r$};
\node at (5,39) {$\rho_\ell$};
\node at (5,42) {$\rho_r$};
\node at (8,44) {\strut $t_L$};
\node at (12,46) {\strut $t_H$};
\node at (18.5,48.5) {\strut $t_F$};
\node at (22,49.8) {\strut $t_D$};
\node at (25,51) {\strut $0$};
\end{tikzpicture}
&
\begin{tikzpicture}[every node/.style={anchor=south west,inner sep=0pt},x=1mm, y=1mm]
\node at (4,4) {\includegraphics[height=50mm]{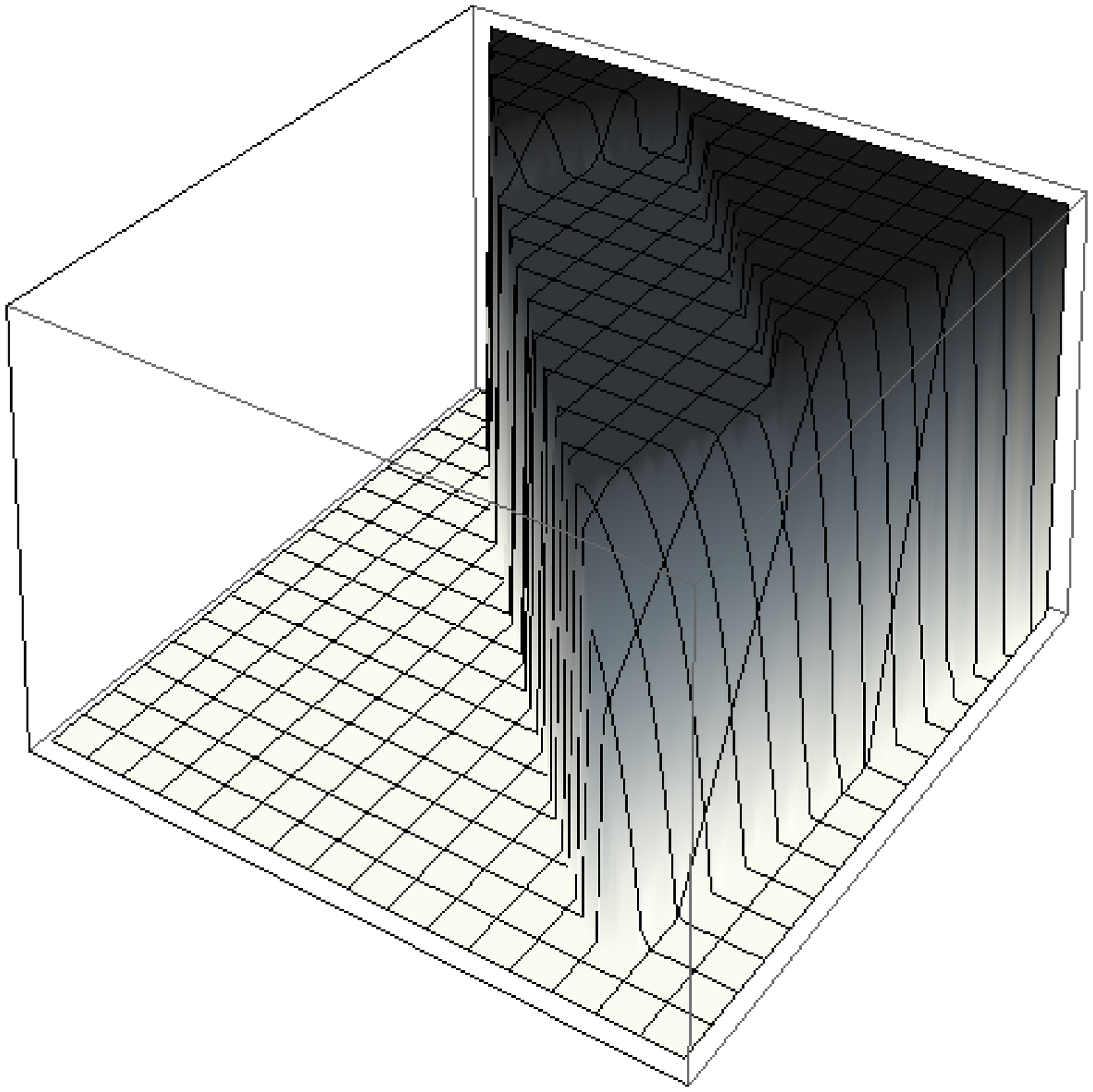}};
\node at (34.5,2.5) {$0$};
\node at (30,4) {$t_D$};
\node at (26,6) {$t_F$};
\node at (14,12) {$t_H$};
\node at (9.5,14) {\strut $t_L$};
\node at (3.5,20) {$0$};
\node at (3,33) {$\mathtt{f}_*$};
\node at (3,37) {$F$};
\node at (7,43) {$x_A$};
\node at (13,46.5) {$x_B$};
\node at (17.5,49.5) {$x_H$};
\node at (26,53.5) {$0$};
\end{tikzpicture}
\\
$(t,x)\mapsto\rho(t,x)$&
$(t,x)\mapsto f(t,x)$
\\[10pt]
\begin{tikzpicture}[every node/.style={anchor=south west,inner sep=0pt},x=1mm, y=1mm]
\node at (4,4) {\includegraphics[height=50mm]{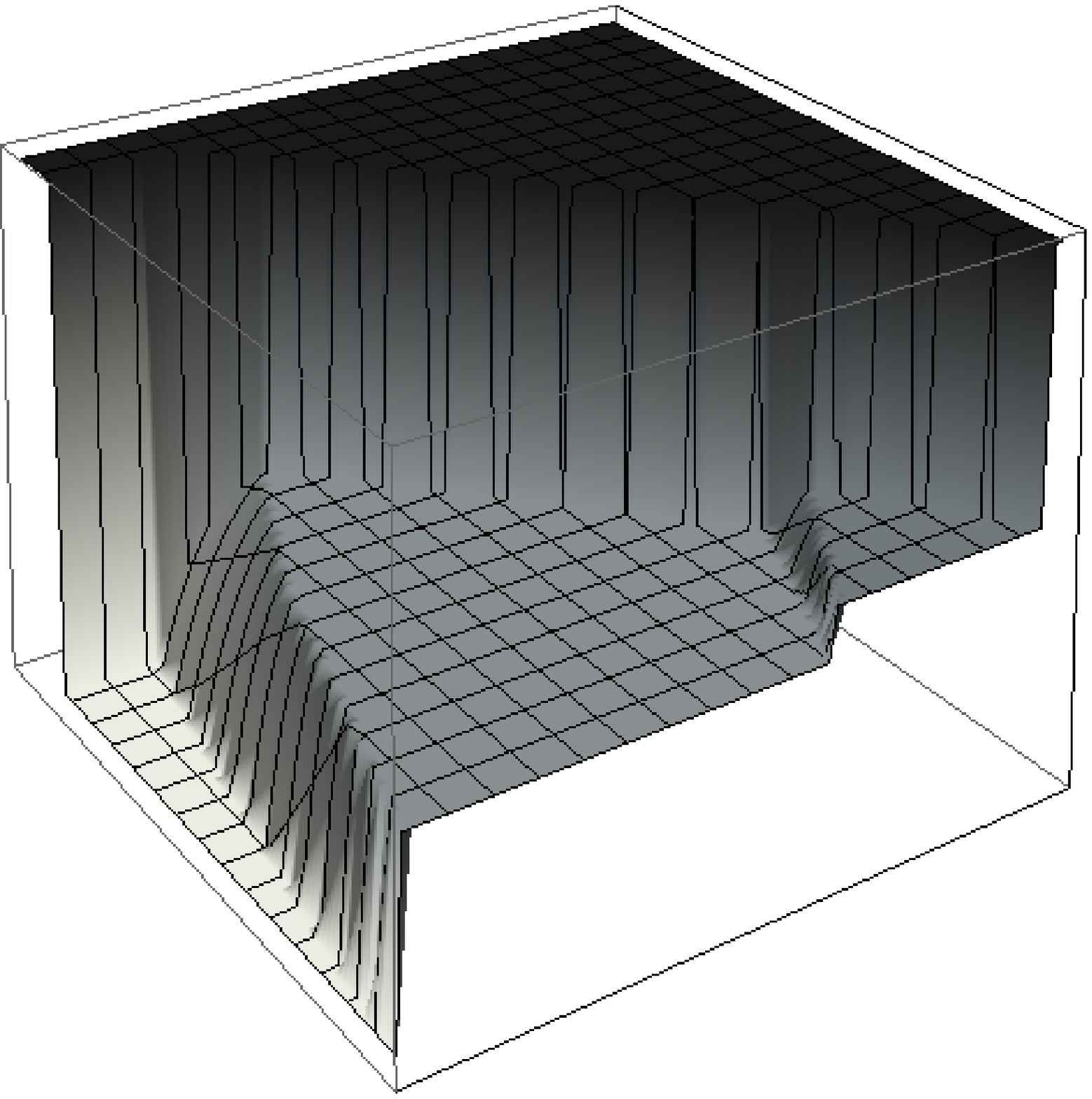}};
\node at (20,1) {$0$};
\node at (25,2.5) {$t_D$};
\node at (30,5) {$t_F$};
\node at (41,9.5) {$t_H$};
\node at (49.5,13) {$t_L$};
\node at (52,17.5) {$0$};
\node at (52,26) {$\hat{\mathtt v}_r = \mathtt{v}_*$};
\node at (52,29) {$\hat{\mathtt v}_\ell$};
\node at (52,40) {$w^-$};
\node at (32,52) {$x_A$};
\node at (40,48) {$x_B$};
\node at (45,45.5) {$x_H$};
\node at (50,43.5) {$0$};
\end{tikzpicture}&
\begin{tikzpicture}[every node/.style={anchor=south west,inner sep=0pt},x=1mm, y=1mm]
\node at (4,4) {\includegraphics[height=50mm]{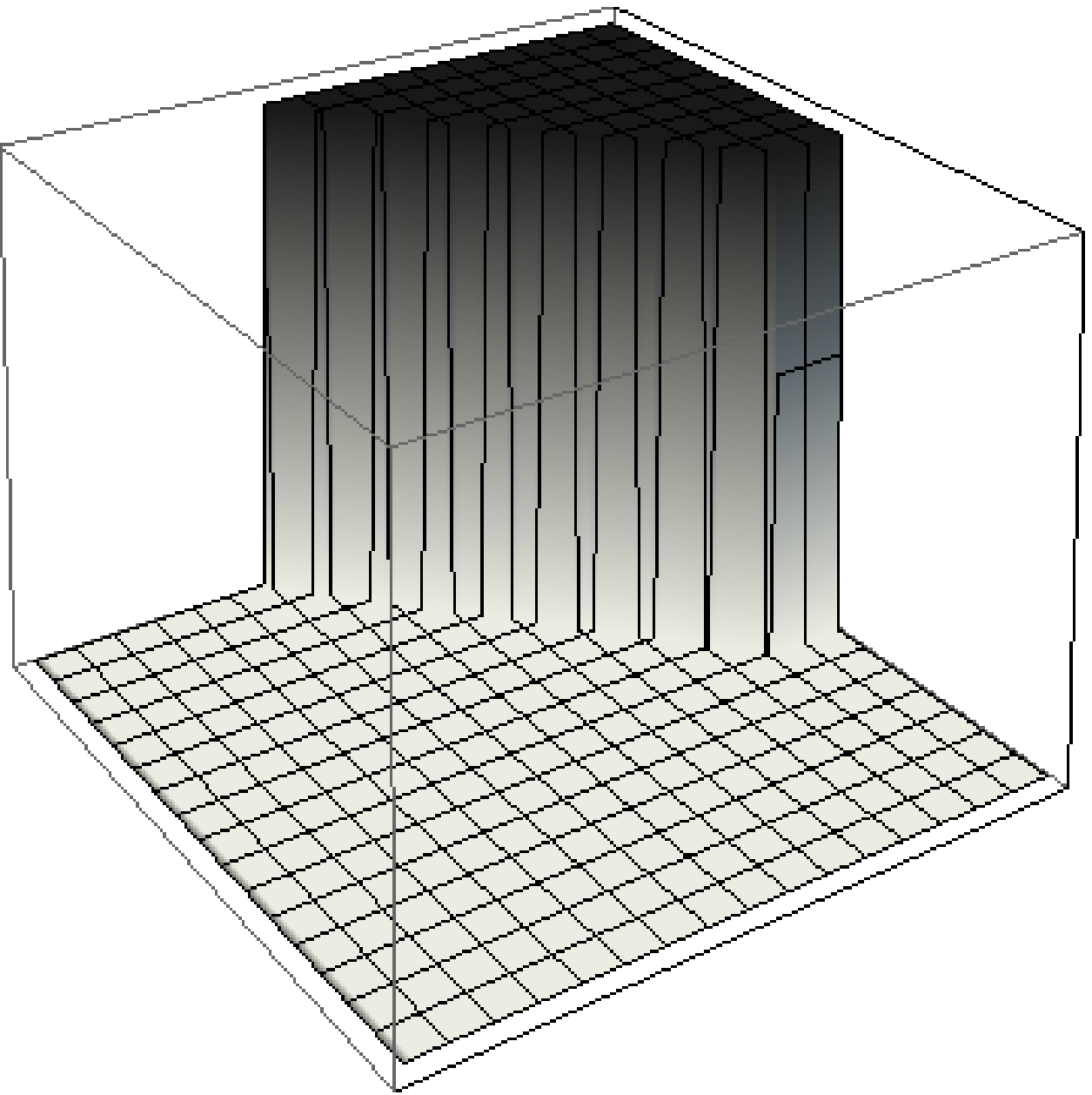}};
\node at (24,.5) {$t_L$};
\node at (38,6) {$t_G$};
\node at (52,14) {$0$};
\node at (54,17) {$w^-$};
\node at (55,40) {$w^+$};
\node at (49,46.5) {$x_A$};
\node at (43,49) {$x_B$};
\node at (33,54.5) {$0$};
\end{tikzpicture}
\\
$(t,x)\mapsto v(t,x)$&
$(t,x)\mapsto w(t,x)$
\end{tabular}
}
\caption{The solution constructed in Section~\ref{s:exPT}.
Above we let $\mathtt{f}_*=f(\mathtt{u}_*)$.}
\label{f:PTsol}
\end{figure}

We finally observe that, once the overall picture of the solution is known, it is possible to express in a closed form the time at which the last vehicle passes through $x = 0$, indeed $t_L = [(x_B-x_A) \, \rho_\ell-x_B \, \rho_r]/F \approx 24.4716$.

\section{Proof of Theorem~\ref{t:mainF}}\label{s:mainF}

In this section we prove Theorem~\ref{t:mainF}.
More precisely, in Section~\ref{s:approxsolR} we construct a grid $\mathcal{G}_n$, approximate Riemann solvers $\mathcal{R}_n$, $\mathcal{R}_{F,n}$ and an approximate solution $u_n = (\rho_n,v_n)$ to constrained Cauchy problem \eqref{eq:system}, \eqref{eq:initdat}, \eqref{eq:const}.
In Section~\ref{s:Gf} we prove that the approximate solution $u_n$ is well defined globally in time by introducing a non-increasing Temple functional $\mathcal{T}_n$, which strictly decreases any time the number of the discontinuities of $u_n$ increases.
In Section~\ref{s:conv} we prove that $u_n$ converges to $u$, which is a solution to \eqref{eq:system}, \eqref{eq:initdat}, \eqref{eq:const} and satisfies the estimates listed in \eqref{e:estimates}.
At last in Section~\ref{s:opt} we consider the flux density of the non-classical shocks.

We choose to study the total variation in the $(v,w)$-coordinates rather than in the $(\rho,v)$-coordinates.
This choice is in fact convenient to describe the grid, the approximate Riemann solvers and ease the forthcoming analysis, because the total variation of $u_n$ in these coordinates does not increase after any interaction away from $x=0$.
Furthermore, the entropy pairs in the $(v,w)$-coordinates are well defined, but in the $(\rho,v)$-coordinates are multi-valued at the vacuum.

For simplicity we assume below that $n\in\N$ is sufficiently large.
Moreover we simplify the notation by letting 
\begin{align*}
\mathtt{w}_\ell &\doteq \mathtt{w}(u_\ell),&
\hat{\mathtt u}_\ell &\doteq \hat{\mathtt u}(\mathtt{w}_\ell,F),&
\check{\mathtt u}_\ell &\doteq \check{\mathtt u}(v_\ell,F)
\end{align*}
and so on, where $\hat{\mathtt{u}}$ and $\check{\mathtt{u}}$ are defined in \eqref{e:hat-check}.

\subsection{The approximate solution}\label{s:approxsolR}

In this section we apply the wave-front tracking algorithm to construct an approximate solution $u_n$ in the space $\mathbf{PC}$ of piecewise constant functions taking finitely many values.
To do so we introduce a grid $\mathcal{G}_n$ in $\Omega$ and approximate Riemann solvers $\mathcal{R}_n$, $\mathcal{R}_{F,n} : \mathcal{G}_n \times \mathcal{G}_n \rightarrow \mathbf{PC}(\mathbb{R};\mathcal{G}_n)$.

\subsubsection*{The grid}

We introduce in $\Omega$ a grid $\mathcal{G}_n \doteq \Omega \cap \mathcal{P}$, see \figurename~\ref{f:grid}, with $\mathcal{P}$ given in the $(v,w)$-coordinates by
\[
\left( \cup_{i=0}^{M\cdot2^n} \left\{v^i\right\} \right) \times \left( \cup_{i=0}^{N\cdot2^n} \left\{w^i\right\} \right),
\]
where $M$, $N$, $v^i$ and $w^i$, are defined as follows:
\begin{figure}[!ht]
\begin{center}
\begin{tikzpicture}[every node/.style={anchor=south west,inner sep=0pt},x=1mm, y=1mm]
\node at (4,4) {\includegraphics[height=50mm]{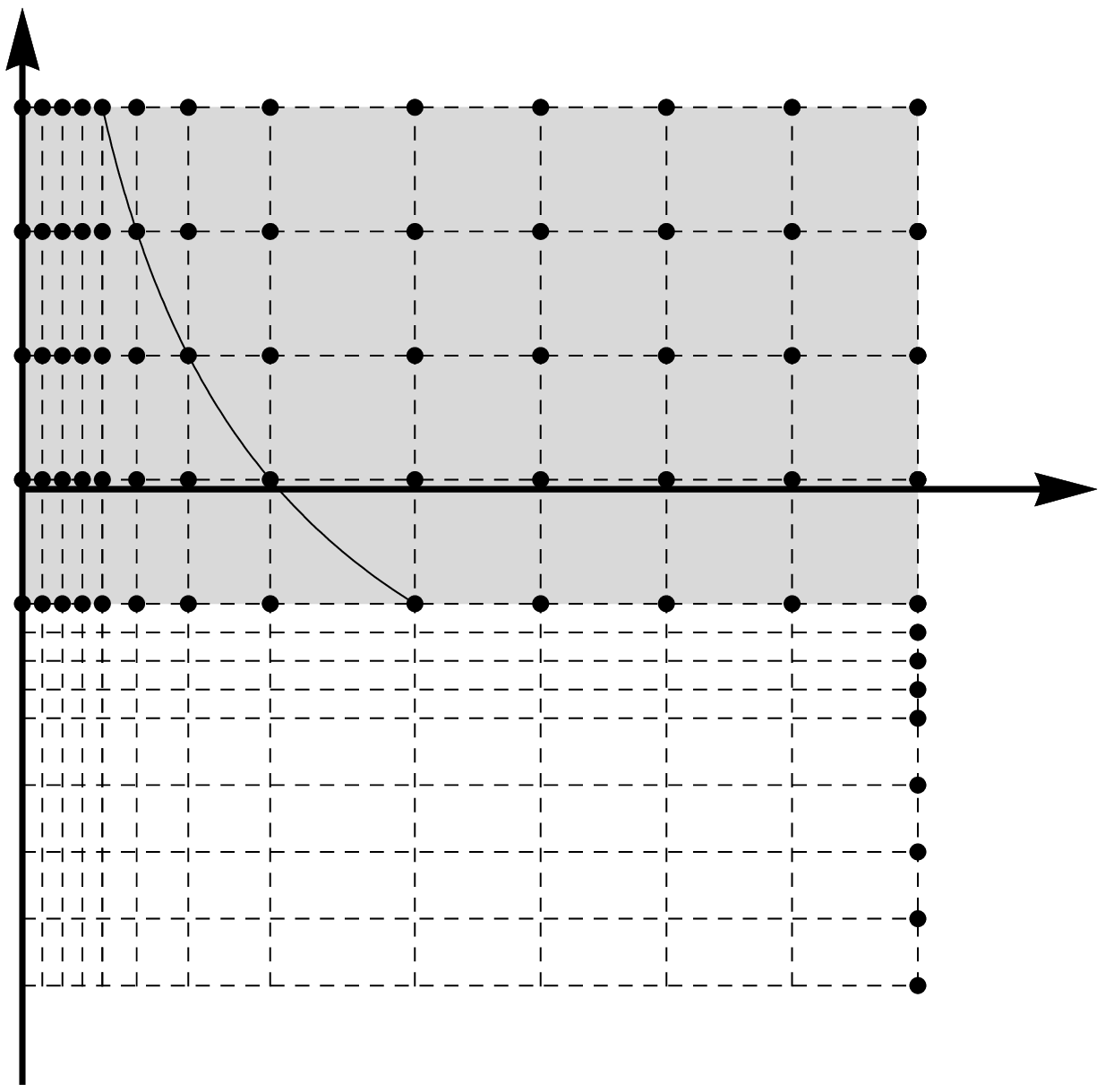}};
\node at (0,50) {\strut $w$};
\node at (52,27) {\strut $v$};
\node at (0,7.5) {\strut $w^0$};
\node at (0,19) {\strut $w^4$};
\node at (0,24.5) {\strut $w^8$};
\node at (-1,46) {\strut $w^{12}$};
\node at (46.5,46) {\strut $w^+$};
\node at (46.5,24.5) {\strut $w^-$};
\node at (46.5,19) {\strut $w_F$};
\node at (46.5,7.5) {\strut $w^--1$};
\node at (7.5,48.5) {\strut $v^4$};
\node at (21.5,48.5) {\strut $v^8$};
\node at (44,48.5) {\strut $v^{12}$};
\node at (7.5,4.5) {\strut $v_F^-$};
\node at (21.5,4.5) {\strut $v_F^+$};
\node at (44,4.5) {\strut $V$};
\node at (52,-1) {\color{white}{\strut $\rho$}};
\end{tikzpicture}
\qquad
\begin{tikzpicture}[every node/.style={anchor=south west,inner sep=0pt},x=1mm, y=1mm]
\node at (4,4) {\includegraphics[height=50mm]{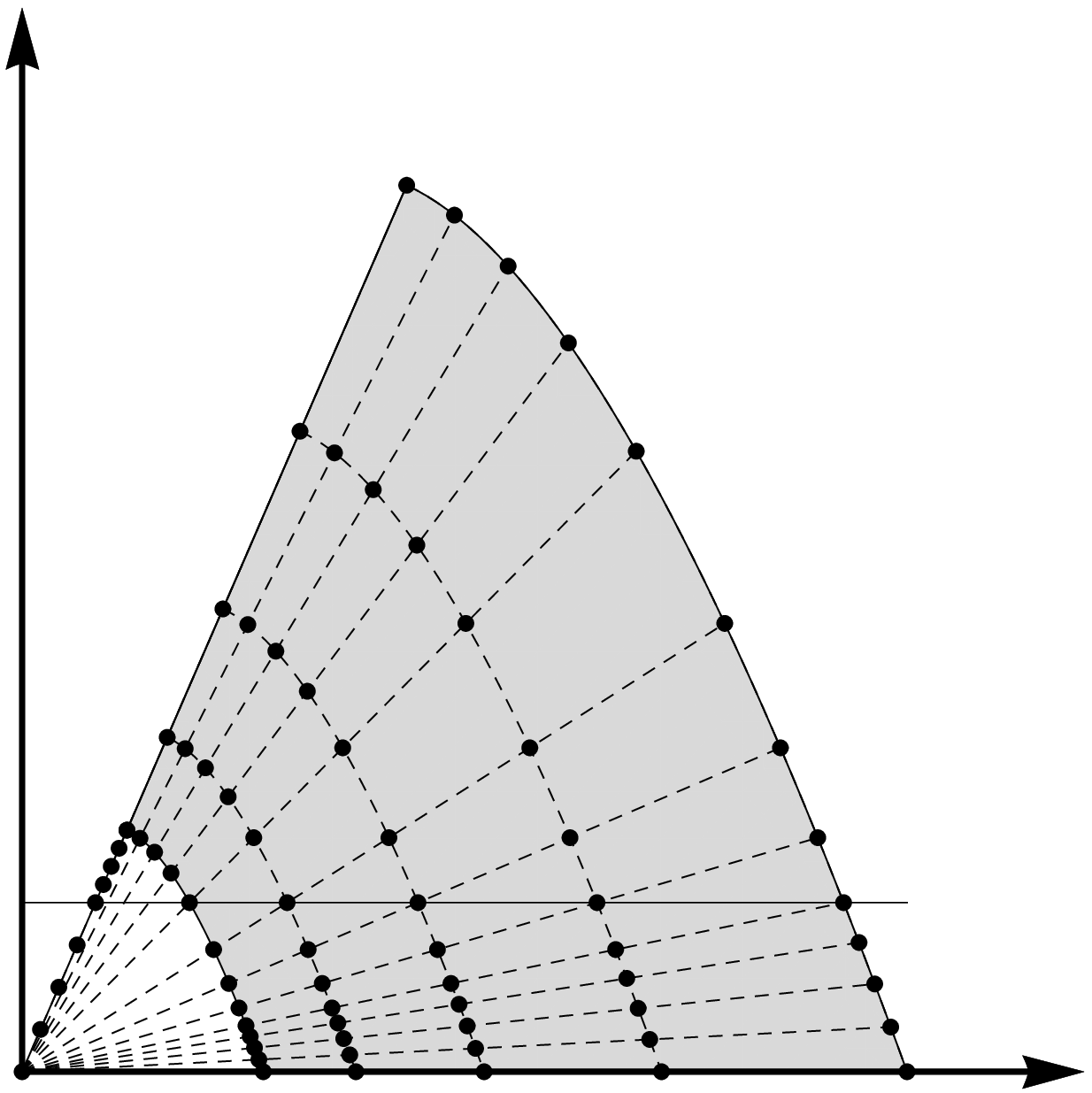}};
\node at (52,-1) {\strut $\rho$};
\node at (0,10) {\strut $F$};
\node at (0,50) {\strut $f$};
\end{tikzpicture}
\end{center}
\caption{The grid $\mathcal{G}_n$ corresponding to $F \in (0,f_{\rm c}^-)$ and $n = 2$.
The curve in the figure on the left is the support of $\Xi_F$, which corresponds to (a portion of) the horizontal line in the figure on the right.}
\label{f:grid}
\end{figure}
\begin{itemize}[leftmargin=*]\setlength{\itemsep}{0cm}%
\item
If $F = 0$, then we let $M=1$, $N=2$,
\begin{align*}
w^i &\doteq 
\begin{cases}
w^- - 1 + i \, 2^{ - n}&\hbox{if } i \in \left\{ 0, \ldots ,  2^n \right\}, \\ 
w^- + \left(i - 2^n\right) \, 2^{ - n} \, \left( w^+ - w^- \right)&\hbox{if } i \in \left\{ 2^n+1, \ldots ,  2\cdot 2^{n} \right\},
\end{cases}
\intertext{and}
v^i &\doteq 
i \, 2^{ - n} \, V \quad \hbox{if } i \in \left\{ 0, \ldots ,  2^n \right\}.
\end{align*}
\item
If $F \in (0,f_{\rm c}^-)$, then we let $M=3$, $N=3$,
\begin{align*}
w^i &\doteq 
\begin{cases}
w^- - 1 + i \, 2^{ - n} \, (w_F-w^-+1)&\hbox{if } i \in \left\{ 0, \ldots ,  2^n \right\}, \\ 
w_F + \left(i - 2^n\right) \, 2^{ - n} \, (w^--w_F)&\hbox{if } i \in \left\{ 2^n+1, \ldots ,  2\cdot 2^{n} \right\}, \\ 
w^- + \left(i - 2 \cdot 2^n\right) \, 2^{ - n} \, \left( w^+ - w^- \right)&\hbox{if } i \in \left\{ 2\cdot 2^n+1, \ldots ,  3\cdot 2^{n} \right\},
\end{cases}
\intertext{and}
v^i &\doteq 
\begin{cases}
i \, 2^{ - n} \, v_F^-&\hbox{if } i \in \left\{ 0, \ldots ,  2^n \right\}, \\ 
\Xi_F^{-1}(w^{4\cdot 2^{n} - i})&\hbox{if } i \in \left\{ 2^n+1, \ldots ,  2\cdot 2^{n} \right\}, \\ 
v_F^+ + \left(i - 2\cdot 2^{n}\right) \, 2^{ - n} \, ( V - v_F^+ )&\hbox{if } i \in \left\{ 2\cdot 2^{n}+1, \ldots , 3\cdot 2^{n}\right\}.
\end{cases}
\end{align*}
\item
If $F \in [f_{\rm c}^-,f_{\rm c}^+]$, then we let $M=2$, $N=3$,
\begin{align*}
w^i &\doteq 
\begin{cases}
w^- - 1 + i \, 2^{ - n}&\hbox{if } i \in \left\{ 0, \ldots ,  2^n \right\}, \\ 
w^- + \left(i - 2^n\right) \, 2^{ - n} \, (w_F-w^-)&\hbox{if } i \in \left\{ 2^n+1, \ldots ,  2\cdot 2^{n} \right\}, \\ 
w_F + \left(i - 2 \cdot 2^n\right) \, 2^{ - n} \, \left( w^+ - w_F \right)&\hbox{if } i \in \left\{ 2\cdot 2^n+1, \ldots ,  3\cdot 2^{n} \right\},
\end{cases}
\intertext{and}
v^i &\doteq 
\begin{cases}
i \, 2^{ - n} \, v_F^-&\hbox{if } i \in \left\{ 0, \ldots ,  2^n \right\}, \\ 
\Xi_F^{-1}(w^{4\cdot 2^{n} - i})&\hbox{if } i \in \left\{ 2^n+1, \ldots ,  2\cdot 2^{n} \right\}.
\end{cases}
\end{align*}
Notice that if $F \in \{f_{\rm c}^-,f_{\rm c}^+\}$, then not necessarily $w^i \neq w^{i+1}$.
\end{itemize}

\subsubsection*{The approximate Riemann solvers}

An approximate solution $u_n \in \mathbf{PC}(\R;\mathcal{G}_n)$ to \eqref{eq:system}, \eqref{eq:initdat}, \eqref{eq:const} is constructed by applying the approximate Riemann solvers $\mathcal{R}_n$, $\mathcal{R}_{F,n} : \mathcal{G}_n \times \mathcal{G}_n \rightarrow \mathbf{PC}(\mathbb{R};\mathcal{G}_n)$, which are obtained by approximating the rarefactions. 
More precisely, for any $(u_\ell, u_r) \in \mathcal{G}_n \times \mathcal{G}_n$ such that $\mathtt{w}_\ell = \mathtt{w}_r$ and $v_\ell = v^h<v_r = v^{h+k}$, we let
\[
\mathcal{R}_n[u_\ell, u_r](\xi) \doteq \begin{cases}
u_\ell &\text{if } \xi \le \Lambda(u_\ell, u_1),\\
u_j &\text{if } \Lambda(u_{j - 1}, u_j)<\xi\leq \Lambda(u_j, u_{j + 1}),\ 1 \leq j \leq k - 1,\\
u_r &\text{if } \xi > \Lambda(u_{k - 1}, u_r),
\end{cases}
\]
where $u_0 \doteq u_\ell$, $u_{k} \doteq u_r$ and $u_j \in \mathcal{G}_n$ is such that $v_j \doteq v^{h+j}$ and $w_j =\mathtt{w}_\ell$. 
The Riemann solver $\mathcal{R}_{F,n}$ is defined as follows:
\begin{enumerate}
\item If $f\left(\mathcal{R}_n[u_\ell, u_r](0_\pm)\right) \le F$, then $\mathcal{R}_{F,n}[u_\ell, u_r] \dot\equiv \mathcal{R}_n[u_\ell, u_r]$.
\item If $f\left(\mathcal{R}_n[u_\ell, u_r](0_\pm)\right) > F$, then 
\[
\mathcal{R}_{F,n}[u_\ell, u_r](\xi) \doteq \begin{cases}
\mathcal{R}_n[u_\ell,\hat{\mathtt u}_\ell](\xi) &\text{if } \xi<0,\\
\mathcal{R}_n[\check{\mathtt u}_r,u_r](\xi) &\text{if } \xi \geq 0.
\end{cases}
\]
\end{enumerate}

\subsubsection*{The approximate solution}

An approximate solution $u_n \in \mathbf{PC}(\R_+\times\R;\mathcal{G}_n)$ to \eqref{eq:system}, \eqref{eq:initdat}, \eqref{eq:const} can be constructed as follows.
As a first step we approximate the initial datum $u^o$ with $u^o_n \in \mathbf{PC}(\R;\mathcal{G}_n)$ such that
\begin{equation}\label{e:apprini}
\begin{aligned}
\|v^o_n\|_{\L\infty} &\leq \|v^o\|_{\L\infty},&
\tv(v_n^o) &\leq \tv(v^o),&
\lim_{n\to\infty}\|v^o_n-v^o\|_{\Lloc1} &= 0,&
\hat{\Upsilon}(u^o_n) &\leq C \, \hat{\Upsilon}(u^o),
\\[5pt]
\|\mathtt{w}^o_n\|_{\L\infty} &\leq \|\mathtt{w}^o\|_{\L\infty},&
\tv(\mathtt{w}_n^o) &\leq \tv(\mathtt{w}^o),&
\lim_{n\to\infty}\|\mathtt{w}^o_n-\mathtt{w}^o\|_{\Lloc1} &= 0,&
\check{\Upsilon}(u^o_n) &\leq C \, \check{\Upsilon}(u^o),
\end{aligned}
\end{equation}
for a constant $C$.
The approximate solution $u_n$ is then obtained by gluing together the approximate solutions computed by applying $\mathcal{R}_{F,n}$ at $x = 0$ at time $t = 0$ and at any time a wave-front reaches $x = 0$, and by applying $\mathcal{R}_n$ at any discontinuity of $u^o_n$ away from $x = 0$ or at any interaction between wave-fronts away from $x = 0$.
As usual, in order to extend the construction globally in time we have to ensure that only finitely many interactions may occur in finite time. 
In Section~\ref{s:Gf} we prove that $u_n(t,\cdot)$ is well defined for all $t > 0$ and belongs to $\mathbf{PC}(\R_+\times\R;\mathcal{G}_n)$.
Finally, in Section~\ref{s:conv} we prove that $u_n$ converges (up to a subsequence) in $\Lloc1$ to a limit $u$, which results to be a constrained solution to \eqref{eq:system}, \eqref{eq:initdat}, \eqref{eq:const} in the sense of Definition~\ref{def:solconstmod}.

\subsection{A priori estimates}\label{s:Gf}

In this section we prove the main a priori estimates on the sequence of approximate solutions $(u_n)_n$.
We prove in Proposition~\ref{p:interest} that $u_n$ takes values in $\mathcal{G}_n$ and we estimate $\tv(u_n(t,\cdot))$ uniformly in $n$ and $t$.
This together with Proposition~\ref{p:inter} guarantee that the number of interactions and the number of the discontinuities of $u_n$ are both bounded globally in time.

Observe that any Contact Discontinuity (CD) has non-negative speed (of propagation), any Shock (S) or Rarefaction Shock (RS) has negative speed, all the Non-classical Shocks (NSs) are stationary and the speed of all the possible Phase Transitions (PTs) ranges in the interval $(-f_{\rm c}^-/(p^{-1}(w^-)-\rho^-),V)$.
Below we say that $(u_\ell,u_r)$ is a null wave if $u_\ell = u_r$.
Notice that if $(u_\ell,u_r)$ is a PT then $u_\ell \in \Omega_{\rm f}^-$ and $u_r\in\Omega_{\rm c}^-$, moreover if $(u_\ell,u_r)$ is a PT with $\mathtt{w}_r>w^-$ then $\rho_\ell = 0$.

Let $u_n$ be an approximate solution.
Let $\sharp(t)$ be the number of waves/discontinuities of $u_n(t,\cdot)$ and introduce $\mathcal{T}_n \colon \R_+ \to \mathbb{R}_+$ defined as
\[
\mathcal{T}_n(t) \doteq \tv\bigl(v_n(t,\cdot)\bigr)+\tv\bigl(\mathtt{w}_n(t,\cdot)\bigr)+2\hat{\Upsilon}_n(t)+2\check{\Upsilon}_n(t),
\]
where $\hat{\Upsilon}_n(t) \doteq \hat{\Upsilon}(u_n(t,\cdot))$ and $\check{\Upsilon}_n(t\bigr) \doteq \check{\Upsilon}(u_n(t,\cdot))$.
Conventionally, we assume that $u_n$ is left continuous in time, i.e.~$u_n(t,\cdot) \equiv u_n(t_-,\cdot)$.
Then also $\mathcal{T}_n$ is left continuous in time.
By the monotonicity of $w\mapsto\hat{\mathtt v}(w)$, $w\mapsto\hat{\mathtt w}(w)$, $v\mapsto\check{\mathtt v}(v)$, $v\mapsto\check{\mathtt w}(v)$, see Remark~\ref{r:hatcheck}, and the definitions of $\hat{\Upsilon}$ and $\check{\Upsilon}$ given in \eqref{e:upsilons}, we have that
\begin{align*}
\hat{\Upsilon}_n(t)&
= \tv_+\Bigl(\hat{\mathtt v}\bigl(\mathtt{w}_n(t,\cdot)\bigr);(-\infty,0)\Bigr) + \tv_-\Bigl(\hat{\mathtt w}\bigl(\mathtt{w}_n(t,\cdot)\bigr);(-\infty,0)\Bigr)
\\&=
\sum_{x\in\mathsf{CD}_n} \Bigl\{ \bigl[\hat{\mathtt v}\bigl(\mathtt{w}_n(t,x_+)\bigr)-\hat{\mathtt v}\bigl(\mathtt{w}_n(t,x_-)\bigr)\bigr]_+ + \bigl[\hat{\mathtt w}\bigl(\mathtt{w}_n(t,x_-)\bigr)-\hat{\mathtt w}\bigl(\mathtt{w}_n(t,x_+)\bigr)\bigr]_+ \Bigr\},
\\
\check{\Upsilon}_n(t)&
= \tv_+\Bigl(\check{\mathtt v}\bigl(v_n(t,\cdot),F\bigr);(0,\infty)\Bigr) + \tv_-\Bigl(\check{\mathtt w}\bigl(v_n(t,\cdot),F\bigr);(0,\infty)\Bigr)
\\&=
\sum_{x\in\mathsf{RS}_n} \Bigl\{ \bigl[\check{\mathtt v}\bigl(v_n(t,x_+),F\bigr)-\check{\mathtt v}\bigl(v_n(t,x_-),F\bigr)\bigr]_+ + \bigl[\check{\mathtt w}\bigl(v_n(t,x_-),F\bigr)-\check{\mathtt w}\bigl(v_n(t,x_+),F\bigr)\bigr]_+ \Bigr\},
\end{align*}
where
\begin{align*}
\mathsf{CD}_n &\doteq \Bigl\{ x \in \R : \bigl(u_n(t,x_-),u_n(t,x_+)\bigr) \text{ is a CD in } x<0 \text{ such that } \mathtt{w}_n(t,x_-) > \max\{\mathtt{w}_n(t,x_+),w_F\} \Bigr\},
\\
\mathsf{RS}_n &\doteq \Bigl\{ x \in \R : \bigl(u_n(t,x_-),u_n(t,x_+)\bigr) \text{ is a RS in } x>0 \text{ such that } v_n(t,x_+) > \max\{v_n(t,x_-),v_F^-\} \Bigr\}.
\end{align*}
Let $\varepsilon_n>0$ be the minimal $(v,w)$-distance between two points in the grid $\mathcal{G}_n$, namely
\[
\varepsilon_n \doteq \min_{\substack{u^1,\, u^2 \in \mathcal{G}_n\\ u^1 \neq u^2}} \max \Bigl\{|v^1 - v^2| , |\mathtt{w}(u^1) - \mathtt{w}(u^2)|\Bigr\}.
\]

The next proposition ensures that the number of discontinuities of $u_n$ is uniformly bounded in time.
Moreover, it gives uniform bounds on the total variation of the approximate solution, which allows us to use Helly's Theorem.

\begin{proposition}\label{p:interest}
For any fixed $n \in \N$ sufficiently large and $u^o_n \in \mathbf{PC}(\R;\mathcal{G}_n)$, we have that:
\begin{enumerate}[label={{\rm\textbf{(\alph*)}}},leftmargin=*]
\item\label{TommyEmmanuel}
the map $t \mapsto \mathcal{T}_n(t)$ is non-increasing and decreases by at least $\varepsilon_n$ any time the number of waves increases;
\item
$u_n(t,\cdot) \in \mathbf{PC}(\R;\mathcal{G}_n)$ for all $t>0$.
\end{enumerate}
\end{proposition}

\begin{proof}
By construction for $t>0$ sufficiently small $u_n(t,\cdot)$ belongs to $\mathbf{PC}(\R;\mathcal{G}_n)$, more precisely it is piecewise constant with jumps along a finite number of straight lines.
If at time $t>0$ an interaction occurs, namely two waves meet or a wave reaches $x=0$, then the involved waves may change speed or strength, while new waves may be created.
To prove that $u_n(t,\cdot)$ belongs to $\mathbf{PC}(\R;\mathcal{G}_n)$ we have to provide an a priori upper bound for the number of waves, which follows from \ref{TommyEmmanuel}.

Clearly, if at time $t>0$ no interaction occurs then $\mathcal{T}_n(t) = \mathcal{T}_n(t_+)$.
For this reason we consider below all the possible interactions and distinguish the following main cases:
\begin{itemize}
\item
a single wave reaches $x=0$ and no NS is involved;
\item
a single wave reaches $x=0$ and a NS is involved;
\item
two waves interact away from $x=0$;
\item
two waves interact at $x=0$ and no NS is involved;
\item
two waves interact at $x=0$ and a NS is involved.
\end{itemize}

For completeness we estimate
\begin{align*}
\Delta\tv_v &\doteq \tv\bigl(v_n(t_+,\cdot)\bigr) - \tv\bigl(v_n(t,\cdot)\bigr),&
\Delta\hat{\Upsilon}_n &\doteq \hat{\Upsilon}_n\bigl(\mathtt{w}_n(t_+,\cdot)\bigr) - \hat{\Upsilon}_n\bigl(\mathtt{w}_n(t,\cdot)\bigr),
\\
\Delta\tv_w &\doteq \tv\bigl(\mathtt{w}_n(t_+,\cdot)\bigr) - \tv\bigl(\mathtt{w}_n(t,\cdot)\bigr),&
\Delta\check{\Upsilon}_n &\doteq \check{\Upsilon}_n\bigl(v_n(t_+,\cdot)\bigr) - \check{\Upsilon}_n\bigl(v_n(t,\cdot)\bigr),
\end{align*}
and
\begin{align*}
\Delta\sharp &\doteq \sharp(t_+) - \sharp(t_-),&
\Delta\mathcal{T}_n &\doteq \mathcal{T}_n(t_+)-\mathcal{T}_n(t_-).
\end{align*}

For simplicity in the exposition, whenever a NS is involved we consider separately the cases $F \in [f_{\rm c}^-,f_{\rm c}^+)$ and $F \in [0,f_{\rm c}^-)$.
Notice that $w^-> w_F$ if and only if $F < f_{\rm c}^-$, or equivalently $V \neq v_F^+$.
Notice also that if $F = f_{\rm c}^+$ then $\mathcal{D}_2 = \emptyset$, while if $F = 0$ then $\mathcal{D}_1 = \emptyset$.
At last, notice that if $F \in [f_{\rm c}^-,f_{\rm c}^+)$ and $(u_\ell,u_r) \in \mathcal{D}_2$, then $\hat{\mathtt w}_\ell = \mathtt{w}_\ell$ and $\check{\mathtt v}_r = v_r$.

\medskip\noindent
We start with the interaction estimates.

\begin{itemize}[wide=0pt]

\item
If a wave $(u_\ell,u_r)$ reaches $x=0$, $u_n(t,0_-) = u_n(t,0_+)$ and $(u_\ell,u_r) \in \mathcal{D}_1$, then the constraint has no influence on the wave and $0 = \Delta\tv_v = \Delta\tv_w = \Delta\sharp$.
Since any CD has non-negative speed, we have that $\Delta\hat{\Upsilon}_n \leq 0$.
Since any RS has negative speed, we have that $\Delta\check{\Upsilon}_n \leq 0$.
As a consequence $\Delta\mathcal{T}_n \leq 0$.

\item
Assume that a wave $(u_\ell,u_r)$ reaches $x=0$, $u_n(t,0_-) = u_n(t,0_+)$ and $(u_\ell,u_r) \in \mathcal{D}_2$.\\
If $F \in [f_{\rm c}^-,f_{\rm c}^+)$, then one of the following cases occurs:

\begin{enumerate}[wide=0pt]

\item[\textbf{CD$\pmb{_F^+}$}]
$(u_\ell,u_r)$ is a CD.
In this case $\hat{\mathtt v}_r \geq v_\ell = v_r = \check{\mathtt v}_r > \hat{\mathtt v}_\ell$, $\mathtt{w}_\ell = \hat{\mathtt w}_\ell > \check{\mathtt w}_r \geq \hat{\mathtt w}_r \geq \mathtt{w}_r$ and $f(u_\ell) > F \geq f(u_r)$.
$\mathcal{R}_{F,n}[u_\ell,u_r]$ has at most three waves $(u_\ell,\hat{\mathtt u}_\ell)$, $(\hat{\mathtt u}_\ell,\check{\mathtt u}_r)$ and $(\check{\mathtt u}_r,u_r)$ that are a S, a NS and a possibly null CD, respectively. 
As a consequence
\begin{align*}
\Delta\tv_v &= 2(v_\ell - \hat{\mathtt v}_\ell) > 0,& 
\Delta\hat{\Upsilon}_n &= -[\hat{\mathtt v}_r-\hat{\mathtt v}_\ell]_+ - [\hat{\mathtt w}_\ell-\hat{\mathtt w}_r]_+ < -(\hat{\mathtt v}_r-\hat{\mathtt v}_\ell) < 0,
\\
\Delta\tv_w &= 0,&
\Delta\check{\Upsilon}_n &= 0,
\end{align*}
therefore $\Delta\sharp \in \{1,2\}$ and $\Delta\mathcal{T}_n < -2(\hat{\mathtt v}_r-v_\ell) \leq 0$.

\item[\textbf{RS$\pmb{_F^+}$}]
$(u_\ell,u_r)$ is a RS. 
In this case $v_\ell = \check{\mathtt v}_\ell < v_r = \check{\mathtt v}_r$, $\check{\mathtt w}_r < \mathtt{w}_\ell = \check{\mathtt w}_\ell = \mathtt{w}_r$, $f(u_\ell) = F < f(u_r)$ and $u_\ell$, $u_r \in \Omega_{\rm c}$.
$\mathcal{R}_{F,n}[u_\ell,u_r]$ has two waves $(u_\ell,\check{\mathtt u}_r)$ and $(\check{\mathtt u}_r,u_r)$ that are a NS and a CD, respectively. 
As a consequence
\begin{align*}
\Delta\tv_v &= 0,& 
\Delta\hat{\Upsilon}_n &= 0,
\\
\Delta\tv_w &= 2(\mathtt{w}_\ell - \check{\mathtt w}_r) > 0,&
\Delta\check{\Upsilon}_n &= -[\check{\mathtt v}_r-\check{\mathtt v}_\ell]_+ -[\check{\mathtt w}_\ell-\check{\mathtt w}_r]_+ = -(v_r-v_\ell) -(\mathtt{w}_\ell - \check{\mathtt w}_r) < 0,
\end{align*}
therefore $\Delta\sharp = 1$ and $\Delta\mathcal{T}_n = -2(v_r-v_\ell) < 0$.
\end{enumerate}
\noindent
If $F \in [0,f_{\rm c}^-)$, then one of the following cases occurs:

\begin{enumerate}[wide=0pt]

\item[\textbf{CD$\pmb{_F^-}$}]
$(u_\ell,u_r)$ is a CD. 
In this case $\hat{\mathtt v}_r \geq v_\ell = v_r = \check{\mathtt v}_r > \hat{\mathtt v}_\ell$, $\hat{\mathtt w}_\ell \geq \mathtt{w}_\ell > \check{\mathtt w}_r \geq \hat{\mathtt w}_r \geq \mathtt{w}_r$ and $f(u_\ell) > F \geq f(u_r)$.
$\mathcal{R}_{F,n}[u_\ell,u_r]$ has at most three waves $(u_\ell,\hat{\mathtt u}_\ell)$, $(\hat{\mathtt u}_\ell,\check{\mathtt u}_r)$ and $(\check{\mathtt u}_r,u_r)$ that are a S or a PT, a NS and a possibly null CD, respectively. 
As a consequence
\begin{align*}
\Delta\tv_v &= 2(v_\ell - \hat{\mathtt v}_\ell) > 0,& 
\Delta\hat{\Upsilon}_n &= -[\hat{\mathtt v}_r-\hat{\mathtt v}_\ell]_+ -[\hat{\mathtt w}_\ell-\hat{\mathtt w}_r]_+
= -(\hat{\mathtt v}_r-\hat{\mathtt v}_\ell) -(\hat{\mathtt w}_\ell-\hat{\mathtt w}_r) < 0,
\\
\Delta\tv_w &= 2(\hat{\mathtt w}_\ell-\mathtt{w}_\ell) \geq 0,&
\Delta\check{\Upsilon}_n &= 0,
\end{align*}
therefore $\Delta\sharp \in \{1,2\}$ and $\Delta\mathcal{T}_n = -2(\hat{\mathtt v}_r-v_\ell) -2(\mathtt{w}_\ell-\hat{\mathtt w}_r) < 0$.

\item[\textbf{RS$\pmb{_F^-}$}]
$(u_\ell,u_r)$ is a RS. 
In this case $v_\ell < v_r \leq \check{\mathtt v}_r$, $w_F \leq \check{\mathtt w}_r < \mathtt{w}_r = \mathtt{w}_\ell = \check{\mathtt w}_\ell$, $f(u_\ell) = F < f(u_r)$ and $u_\ell$, $u_r \in \Omega_{\rm c}$.
$\mathcal{R}_{F,n}[u_\ell,u_r]$ has two waves $(u_\ell,\check{\mathtt u}_r)$ and $(\check{\mathtt u}_r,u_r)$ that are a NS and a PT or a CD, respectively. 
As a consequence
\begin{align*}
\Delta\tv_v &= 2(\check{\mathtt v}_r-v_r) \geq 0,& 
\Delta\hat{\Upsilon}_n &= 0,
\\
\Delta\tv_w &= 2(\mathtt{w}_\ell-\check{\mathtt w}_r) > 0,&
\Delta\check{\Upsilon}_n &= -[\check{\mathtt v}_r-\check{\mathtt v}_\ell]_+ -[\check{\mathtt w}_\ell-\check{\mathtt w}_r]_+
= -(\check{\mathtt v}_r-v_\ell) -(\mathtt{w}_\ell-\check{\mathtt w}_r) < 0,
\end{align*}
therefore $\Delta\sharp = 1$ and $\Delta\mathcal{T}_n= -2(v_r-v_\ell) < 0$.
Notice that $\check{\mathtt v}_r > v_r$ if and only if $\mathtt{w}_\ell = \mathtt{w}_r = w^-$ and $v_r > v_\ell = v_F^+$.

\end{enumerate}

\end{itemize}

Assume that two waves $(u_\ell,u_m)$ and $(u_m,u_r)$ interact at time $t>0$.
Let ${\mathtt u}_* \doteq {\mathtt u}_*(u_\ell,u_r)$.
Notice that ${\mathtt u}_* = u_r$ if and only if $(u_\ell,u_m)$ is a S or a RS, while ${\mathtt u}_* = u_\ell$ if and only if $(u_\ell,u_m)$ is a CD.

\begin{itemize}[wide=0pt]

\item
If the interaction occurs at $x\neq0$, then one of the following cases occurs:

\begin{enumerate}[wide=0pt]

\item[\textbf{CD-S}]
$(u_\ell,u_m)$ is a CD and $(u_m,u_r)$ is a S. 
In this case $v_\ell = v_m > v_r = v_*$, $\mathtt{w}_m = \mathtt{w}_r$, $\mathtt{w}_*$ belongs to the closed interval between $\mathtt{w}_\ell$ and $\mathtt{w}_r$, $\mathtt{W}(u_\ell) = \mathtt{W}({\mathtt u}_*)$, $\mathtt{w}_m = \mathtt{w}_r$, $f(u_m) > f(u_r)$ and $u_m$, $u_r \in \Omega_{\rm c}$.
$\mathcal{R}_n[u_\ell,u_r]$ has at most two waves $(u_\ell,{\mathtt u}_*)$ and $({\mathtt u}_*,u_r)$ that are respectively either a S and a CD, or a PT and a possibly null CD.
As a consequence $0 = \Delta\tv_v = \Delta\tv_w = \Delta\hat{\Upsilon}_n = \Delta\check{\Upsilon}_n$, therefore $\Delta\sharp \leq 0$ and $\Delta\mathcal{T}_n = 0$.

\item[\textbf{CD-RS}]
$(u_\ell,u_m)$ is a CD and $(u_m,u_r)$ is a RS. 
In this case $v_\ell = v_m < v_r = v_*$, $\mathtt{w}_\ell = \mathtt{w}_*$, $\mathtt{w}_m = \mathtt{w}_r$, $f(u_\ell) < f({\mathtt u}_*)$, $f(u_m) < f(u_r)$ and $u_\ell$, $u_m$, ${\mathtt u}_*$, $u_r \in \Omega_{\rm c}$.
$\mathcal{R}_n[u_\ell,u_r]$ has two waves $(u_\ell,{\mathtt u}_*)$ and $({\mathtt u}_*,u_r)$ that are a RS and a CD, respectively. 
As a consequence $0 = \Delta\tv_v = \Delta\tv_w = \Delta\hat{\Upsilon}_n = \Delta\check{\Upsilon}_n$, therefore $\Delta\sharp = 0$ and $\Delta\mathcal{T}_n = 0$.

\item[\textbf{CD-PT}]
$(u_\ell,u_m)$ is a CD and $(u_m,u_r)$ is a PT. 
In this case $v_\ell = v_m = V > v_r = v_*$, $\mathtt{w}_m < w^- \leq \mathtt{w}_r$, $\mathtt{w}_*$ belongs to the closed interval between $\mathtt{w}_\ell$ and $\mathtt{w}_r$, $u_\ell \in \Omega_{\rm f}$, $u_m \in \Omega_{\rm f}^-$ and ${\mathtt u}_*$, $u_r \in \Omega_{\rm c}$.
$\mathcal{R}_n[u_\ell,u_r]$ has at most two waves $(u_\ell,{\mathtt u}_*)$ and $({\mathtt u}_*,u_r)$ that are either a PT or a S and a possibly null CD, respectively. 
As a consequence
\begin{align*}
\Delta\tv_v &= 0,& 
\Delta\hat{\Upsilon}_n &\leq 0,
\\
\Delta\tv_w &= |\mathtt{w}_\ell-\mathtt{w}_r| - (|\mathtt{w}_\ell-\mathtt{w}_m| + |\mathtt{w}_m-\mathtt{w}_r|) \leq 0,&
\Delta\check{\Upsilon}_n &= 0,
\end{align*}
therefore $\Delta\sharp \leq 0$ and $\Delta\mathcal{T}_n \leq 0$.

\item[\textbf{S-S}]
$(u_\ell,u_m)$ and $(u_m,u_r)$ are Ss.
In this case $v_\ell > v_m > v_r$, $\mathtt{w}_\ell = \mathtt{w}_m = \mathtt{w}_r \geq w^-$ and $u_\ell$, $u_m$, $u_r \in \Omega_{\rm c}$.
$\mathcal{R}_n[u_\ell,u_r]$ has one wave $(u_\ell,u_r)$, which is a S. 
As a consequence $0 = \Delta\tv_v = \Delta\tv_w = \Delta\hat{\Upsilon}_n = \Delta\check{\Upsilon}_n$, therefore $\Delta\sharp = -1$ and $\Delta\mathcal{T}_n = 0$.

\item[\textbf{S-RS}]
$(u_\ell,u_m)$ is a S and $(u_m,u_r)$ is a RS. 
In this case $v_\ell > v_r > v_m$, $\mathtt{w}_\ell = \mathtt{w}_m = \mathtt{w}_r \geq w^-$ and $u_\ell$, $u_m$, $u_r \in \Omega_{\rm c}$.
$\mathcal{R}_n[u_\ell,u_r]$ has one wave $(u_\ell,u_r)$, which is a S. 
As a consequence
\begin{align*}
\Delta\tv_v &= -2 (v_r-v_m) < 0,& 
\Delta\hat{\Upsilon}_n &= 0,
\\
\Delta\tv_w &= 0,&
\Delta\check{\Upsilon}_n &\leq 0,
\end{align*}
therefore $\Delta\sharp = -1$ and $\Delta\mathcal{T}_n < 0$.

\item[\textbf{RS-S}]
$(u_\ell,u_m)$ is a RS and $(u_m,u_r)$ is a S. 
In this case $v_m > v_\ell > v_r$, $\mathtt{w}_\ell = \mathtt{w}_m = \mathtt{w}_r \geq w^-$ and $u_\ell$, $u_m$, $u_r \in \Omega_{\rm c}$.
$\mathcal{R}_n[u_\ell,u_r]$ has one wave $(u_\ell,u_r)$, which is a S. 
As a consequence
\begin{align*}
\Delta\tv_v &= -2 (v_m-v_\ell) < 0,& 
\Delta\hat{\Upsilon}_n &= 0,
\\
\Delta\tv_w &= 0,&
\Delta\check{\Upsilon}_n &\leq 0,
\end{align*}
therefore $\Delta\sharp = -1$ and $\Delta\mathcal{T}_n < 0$.

\item[\textbf{PT-S}]
$(u_\ell,u_m)$ is a PT and $(u_m,u_r)$ is a S. 
In this case $v_\ell = V > v_m > v_r$, $\mathtt{w}_\ell < w^- \leq \mathtt{w}_m = \mathtt{w}_r$, $u_\ell \in \Omega_{\rm f}^-$ and $u_m$, $u_r \in \Omega_{\rm c}$.
$\mathcal{R}_n[u_\ell,u_r]$ has one wave $(u_\ell,u_r)$, which is a PT. 
As a consequence $0 = \Delta\tv_v = \Delta\tv_w = \Delta\hat{\Upsilon}_n = \Delta\check{\Upsilon}_n$, therefore $\Delta\sharp = -1$ and $\Delta\mathcal{T}_n = 0$.

\item[\textbf{PT-RS}]
$(u_\ell,u_m)$ is a PT and $(u_m,u_r)$ is a RS. 
In this case $v_\ell = V \geq v_r > v_m$, $\mathtt{w}_\ell < w^- \leq \mathtt{w}_m = \mathtt{w}_r$, $u_\ell \in \Omega_{\rm f}^-$ and $u_m$, $u_r \in \Omega_{\rm c}$.
$\mathcal{R}_n[u_\ell,u_r]$ has one wave $(u_\ell,u_r)$, which is either a PT or a CD. 
As a consequence
\begin{align*}
\Delta\tv_v &= -2 (v_r-v_m) < 0,& 
\Delta\hat{\Upsilon}_n &= 0,
\\
\Delta\tv_w &= 0,&
\Delta\check{\Upsilon}_n &\leq 0,
\end{align*}
therefore $\Delta\sharp = -1$ and $\Delta\mathcal{T}_n < 0$.

\end{enumerate}

\item
If the interaction occurs at $x=0$ and $(u_\ell,u_r) \in \mathcal{D}_1$, then one of the following cases occurs:

\begin{enumerate}[wide=0pt]

\item[\textbf{CD-S$\pmb{_0}$}]
$(u_\ell,u_m)$ is a CD and $(u_m,u_r)$ is a S. 
In this case $v_\ell = v_m > v_r = v_*$, $\mathtt{w}_m = \mathtt{w}_r$, $\mathtt{w}_*$ belongs to the closed interval between $\mathtt{w}_\ell$ and $\mathtt{w}_r$, $\mathtt{W}(u_\ell) = \mathtt{W}({\mathtt u}_*)$, $\mathtt{w}_m = \mathtt{w}_r$, $f(u_r) < (u_m) \leq F$, $\min\{f(u_\ell), f({\mathtt u}_*)\} \leq F$ and $u_m$, $u_r \in \Omega_{\rm c}$.
$\mathcal{R}_{F,n}[u_\ell,u_r]$ has at most two waves $(u_\ell,{\mathtt u}_*)$ and $({\mathtt u}_*,u_r)$ that are respectively either a S and a CD, or a PT and a possibly null CD.
As a consequence $\Delta\hat{\Upsilon}_n \leq 0 = \Delta\tv_v = \Delta\tv_w = \Delta\check{\Upsilon}_n$, therefore $\Delta\sharp \leq 0$ and $\Delta\mathcal{T}_n \leq 0$.

\item[\textbf{CD-RS$\pmb{_0}$}]
$(u_\ell,u_m)$ is a CD and $(u_m,u_r)$ is a RS. 
In this case $v_\ell = v_m < v_r$, $\mathtt{w}_m = \mathtt{w}_r$, $\mathtt{w}_* = \mathtt{w}_\ell$, $f(u_\ell) < f({\mathtt u}_*)$, $f(u_m) < f(u_r)$, $\max\{f(u_m),f({\mathtt u}_*)\} \leq F$ and $u_\ell$, $u_m$, ${\mathtt u}_*$, $u_r \in \Omega_{\rm c}$.
$\mathcal{R}_{F,n}[u_\ell,u_r]$ has two waves $(u_\ell,{\mathtt u}_*)$ and $({\mathtt u}_*,u_r)$ that are a RS and a CD, respectively. 
As a consequence
\begin{align*}
\Delta\tv_v &= 0,& 
\Delta\hat{\Upsilon}_n &\leq 0,
\\
\Delta\tv_w &= 0,&
\Delta\check{\Upsilon}_n &\leq 0,
\end{align*}
therefore $\Delta\sharp = 0$ and $\Delta\mathcal{T}_n \leq 0$.

\item[\textbf{CD-NS$\pmb{_0}$}]
$(u_\ell,u_m)$ is a CD and $(u_m,u_r)$ is a NS. 
In this case $v_F^- \leq v_\ell = v_m < v_r \leq v_F^+$, $w^- \leq \mathtt{w}_\ell \leq \mathtt{w}_r < \mathtt{w}_m$, $f(u_\ell) < f({\mathtt u}_*) \leq F = f(u_m) = f(u_r)$ and $u_\ell$, $u_m$, ${\mathtt u}_*$, $u_r \in \Omega_{\rm c}$.
$\mathcal{R}_{F,n}[u_\ell,u_r]$ has a fan of RSs ranging from $u_\ell$ to ${\mathtt u}_*$ and a possibly null CD $({\mathtt u}_*,u_r)$. 
As a consequence $\Delta\tv_w = -2(\mathtt{w}_m-\mathtt{w}_r) < 0 = \Delta\tv_v = \Delta\hat{\Upsilon}_n = \Delta\check{\Upsilon}_n$, therefore $\Delta\sharp \in [-1,2^n-1]$ and $\Delta\mathcal{T}_n < 0$.

\item[\textbf{CD-PT$\pmb{_0}$}]
$(u_\ell,u_m)$ is a CD and $(u_m,u_r)$ is a PT.
In this case $v_\ell = v_m = V > v_r = v_*$, $\mathtt{w}_m < w^- \leq \mathtt{w}_r$, $\mathtt{w}_*$ belongs to the closed interval between $\mathtt{w}_\ell$ and $\mathtt{w}_r$, $\min\{f(u_\ell),f({\mathtt u}_*)\} \leq F$, $\max\{f(u_m),f(u_r)\} \leq F$, $u_\ell \in \Omega_{\rm f}$, $u_m \in \Omega_{\rm f}^-$ and ${\mathtt u}_*$, $u_r \in \Omega_{\rm c}$.
$\mathcal{R}_{F,n}[u_\ell,u_r]$ has at most two waves $(u_\ell,{\mathtt u}_*)$ and $({\mathtt u}_*,u_r)$ that are either a PT or a S and a possibly null CD, respectively. 
As a consequence
\begin{align*}
\Delta\tv_v &= 0,& 
\Delta\hat{\Upsilon}_n &\leq 0,
\\
\Delta\tv_w &= |\mathtt{w}_\ell-\mathtt{w}_r| - (|\mathtt{w}_\ell-\mathtt{w}_m| + |\mathtt{w}_m-\mathtt{w}_r|) \leq 0,&
\Delta\check{\Upsilon}_n &= 0,
\end{align*}
therefore $\Delta\sharp \in \{-1, 0\}$ and $\Delta\mathcal{T}_n \leq 0$.

\item[\textbf{S-S$\pmb{_0}$}]
$(u_\ell,u_m)$ and $(u_m,u_r)$ are Ss.
In this case $v_\ell > v_m > v_r$, $\mathtt{w}_\ell = \mathtt{w}_m = \mathtt{w}_r \geq w^-$, $f(u_r) < f(u_\ell) \leq F$ and $u_\ell$, $u_m$, $u_r \in \Omega_{\rm c}$.
$\mathcal{R}_{F,n}[u_\ell,u_r]$ has one wave $(u_\ell,u_r)$, which is a S. 
As a consequence $0 = \Delta\tv_v = \Delta\tv_w = \Delta\hat{\Upsilon}_n = \Delta\check{\Upsilon}_n$, therefore $\Delta\sharp = -1$ and $\Delta\mathcal{T}_n= 0$.

\item[\textbf{S-RS$\pmb{_0}$}]
$(u_\ell,u_m)$ is a S and $(u_m,u_r)$ is a RS. 
In this case $v_\ell > v_r > v_m$, $\mathtt{w}_\ell = \mathtt{w}_m = \mathtt{w}_r \geq w^-$, $f(u_r) < f(u_\ell) \leq F$ and $u_\ell$, $u_m$, $u_r \in \Omega_{\rm c}$.
$\mathcal{R}_{F,n}[u_\ell,u_r]$ has one wave $(u_\ell,u_r)$, which is a S. 
As a consequence
\begin{align*}
\Delta\tv_v &= -2 (v_r-v_m) < 0,& 
\Delta\hat{\Upsilon}_n &= 0,
\\
\Delta\tv_w &= 0,&
\Delta\check{\Upsilon}_n &\leq 0,
\end{align*}
therefore $\Delta\sharp = -1$ and $\Delta\mathcal{T}_n < 0$.

\item[\textbf{RS-S$\pmb{_0}$}]
$(u_\ell,u_m)$ is a RS and $(u_m,u_r)$ is a S. 
In this case $v_m > v_\ell > v_r$, $\mathtt{w}_\ell = \mathtt{w}_m = \mathtt{w}_r \geq w^-$, $f(u_r) < f(u_\ell) \leq F$ and $u_\ell$, $u_m$, $u_r \in \Omega_{\rm c}$.
$\mathcal{R}_{F,n}[u_\ell,u_r]$ has one wave $(u_\ell,u_r)$, which is a S. 
As a consequence
\begin{align*}
\Delta\tv_v &= -2 (v_m-v_\ell) < 0,& 
\Delta\hat{\Upsilon}_n &= 0,
\\
\Delta\tv_w &= 0,&
\Delta\check{\Upsilon}_n &\leq 0,
\end{align*}
therefore $\Delta\sharp = -1$ and $\Delta\mathcal{T}_n < 0$.

\item[\textbf{NS-S$\pmb{_0}$}]
$(u_\ell,u_m)$ is a NS and $(u_m,u_r)$ is a S. 
In this case $v_m > v_\ell \geq v_r = v_*$, $\mathtt{w}_* = \mathtt{w}_\ell > \mathtt{w}_m = \mathtt{w}_r \geq w^-$, $f(u_\ell) = f(u_m) = F > f(u_r)$ and $u_\ell$, $u_m$, $u_r \in \Omega_{\rm c}$.
$\mathcal{R}_{F,n}[u_\ell,u_r]$ has at most two waves $(u_\ell,{\mathtt u}_*)$ and $({\mathtt u}_*,u_r)$ that are a possibly null S and CD, respectively. 
As a consequence $\Delta\tv_v = -2 (v_m-v_\ell) < 0 = \Delta\tv_w = \Delta\hat{\Upsilon}_n = \Delta\check{\Upsilon}_n$, therefore $\Delta\sharp \in \{-1,0\}$ and $\Delta\mathcal{T}_n < 0$.

\item[\textbf{NS-PT$\pmb{_0}$}]
$(u_\ell,u_m)$ is a NS and $(u_m,u_r)$ is a PT. 
In this case $v_m = V > v_\ell \geq v_r = v_*$, $\mathtt{w}_\ell = \mathtt{w}_* \geq \mathtt{w}_r = w^- > \mathtt{w}_m$, $f_{\rm c}^- > f(u_\ell) = f(u_m) = F > f(u_r)$, $f({\mathtt u}_*) \leq F$, $u_\ell$, $u_r \in \Omega_{\rm c}$ and $u_m \in \Omega_{\rm f}^-$.
$\mathcal{R}_{F,n}[u_\ell,u_r]$ has at most two waves $(u_\ell,{\mathtt u}_*)$ and $({\mathtt u}_*,u_r)$ that are a possibly null S and a possibly null CD (but not both null), respectively.
As a consequence
\begin{align*}
\Delta\tv_v &= -2(V-v_\ell) < 0,& 
\Delta\hat{\Upsilon}_n &= 0,
\\
\Delta\tv_w &= -2(w^--\mathtt{w}_m) < 0,&
\Delta\check{\Upsilon}_n &= 0,
\end{align*}
therefore $\Delta\sharp \in \{-1,0\}$ and $\Delta\mathcal{T}_n < 0$.

\item[\textbf{PT-S$\pmb{_0}$}]
$(u_\ell,u_m)$ is a PT and $(u_m,u_r)$ is a S. 
In this case $v_\ell = V > v_m > v_r$, $\mathtt{w}_\ell < w^- \leq \mathtt{w}_m = \mathtt{w}_r$, $f(u_r) < f(u_m) \leq \max\{f(u_\ell),f(u_m)\} \leq F$, $u_\ell \in \Omega_{\rm f}^-$ and $u_m$, $u_r \in \Omega_{\rm c}^-$.
$\mathcal{R}_{F,n}[u_\ell,u_r]$ has one wave $(u_\ell,u_r)$, which is a PT. 
As a consequence $0 = \Delta\tv_v = \Delta\tv_w = \Delta\hat{\Upsilon}_n = \Delta\check{\Upsilon}_n$, therefore $\Delta\sharp = -1$ and $\Delta\mathcal{T}_n = 0$.

\item[\textbf{PT-RS$\pmb{_0}$}]
$(u_\ell,u_m)$ is a PT and $(u_m,u_r)$ is a RS. 
In this case $v_\ell = V \geq v_r > v_m$, $\mathtt{w}_\ell < w^- \leq \mathtt{w}_m = \mathtt{w}_r$, $u_\ell \in \Omega_{\rm f}^-$ and $u_m$, $u_r \in \Omega_{\rm c}$.
$\mathcal{R}_{F,n}[u_\ell,u_r]$ has one wave $(u_\ell,u_r)$, which is either a PT or a CD. 
As a consequence
\begin{align*}
\Delta\tv_v &= -2 (v_r-v_m) < 0,& 
\Delta\hat{\Upsilon}_n &= 0,& 
\\
\Delta\tv_w &= 0,&
\Delta\check{\Upsilon}_n &\leq 0,
\end{align*}
therefore $\Delta\sharp = -1$ and $\Delta\mathcal{T}_n < 0$.

\item[\textbf{PT-NS$\pmb{_0}$}]
$(u_\ell,u_m)$ is a PT and $(u_m,u_r)$ is a NS. 
In this case $v_\ell = V \geq v_r > v_m$, $\mathtt{w}_\ell < w^- \leq \mathtt{w}_r < \mathtt{w}_m$, $f(u_\ell) < f(u_m) = f(u_r) = F$, $u_\ell \in \Omega_{\rm f}^-$ and $u_m$, $u_r \in \Omega_{\rm c}$.
$\mathcal{R}_{F,n}[u_\ell,u_r]$ has one wave $(u_\ell,u_r)$, which is either a CD or a PT. 
As a consequence
\begin{align*}
\Delta\tv_v &= -2 (v_r-v_m) < 0,& 
\Delta\hat{\Upsilon}_n &= 0,
\\
\Delta\tv_w &= -2 (\mathtt{w}_m-\mathtt{w}_r) < 0,&
\Delta\check{\Upsilon}_n &= 0,
\end{align*}
therefore $\Delta\sharp = -1$ and $\Delta\mathcal{T}_n < 0$.

\end{enumerate}

\item
Assume that two waves $(u_\ell,u_m)$ and $(u_m,u_r)$ interact at $x=0$ and $(u_\ell,u_r) \in \mathcal{D}_2$.
\\\noindent
If $F \in [f_{\rm c}^-,f_{\rm c}^+)$, then one of the following cases occurs:

\begin{enumerate}[wide=0pt]

\item[\textbf{CD-S$\pmb{_F^+}$}]
$(u_\ell,u_m)$ is a CD and $(u_m,u_r)$ is a S. 
In this case $\hat{\mathtt v}_r \geq v_\ell = v_m > v_r = \check{\mathtt v}_r > \hat{\mathtt v}_\ell$, $\mathtt{w}_\ell = \hat{\mathtt w}_\ell > \check{\mathtt w}_r > \mathtt{w}_m = \mathtt{w}_r$, $\mathtt{w}_\ell > \hat{\mathtt w}_r$, $f(u_\ell) > f({\mathtt u}_*) > F \geq f(u_m) > f(u_r)$ and $u_\ell$, $u_m$, $u_r \in \Omega_{\rm c}$.
$\mathcal{R}_{F,n}[u_\ell,u_r]$ has three waves $(u_\ell,\hat{\mathtt u}_\ell)$, $(\hat{\mathtt u}_\ell,\check{\mathtt u}_r)$ and $(\check{\mathtt u}_r,u_r)$, which are a S, a NS and a CD, respectively.
As a consequence
\begin{align*}
\Delta\tv_v &= 2(v_r-\hat{\mathtt v}_\ell) > 0,& 
\Delta\hat{\Upsilon}_n &= -(\hat{\mathtt v}_m-\hat{\mathtt v}_\ell) -(\mathtt{w}_\ell-\hat{\mathtt w}_m) < 0,
\\
\Delta\tv_w &= 0,&
\Delta\check{\Upsilon}_n &= 0,
\end{align*}
therefore $\Delta\sharp =1$ and $\Delta\mathcal{T}_n = -2(\hat{\mathtt v}_r-v_r) -2(\mathtt{w}_\ell-\hat{\mathtt w}_r) < 0$.

\item[\textbf{CD-RS$\pmb{_F^+}$}]
$(u_\ell,u_m)$ is a CD and $(u_m,u_r)$ is a RS. 
In this case $\hat{\mathtt v}_\ell \leq v_\ell = v_m = \check{\mathtt v}_m < v_r = \check{\mathtt v}_r$, $\hat{\mathtt v}_\ell < \hat{\mathtt v}_m$, $\mathtt{w}_\ell = \hat{\mathtt w}_\ell > \hat{\mathtt w}_m \geq \mathtt{w}_m = \mathtt{w}_r$, $\check{\mathtt w}_r < \check{\mathtt w}_m$, $f(u_m) < f(u_r) \leq f(u_\ell) < f({\mathtt u}_*)$, $f(u_m) \leq F < f({\mathtt u}_*)$ and $u_\ell$, $u_m$, $u_r \in \Omega_{\rm c}$.
$\mathcal{R}_{F,n}[u_\ell,u_r]$ has at most three waves $(u_\ell,\hat{\mathtt u}_\ell)$, $(\hat{\mathtt u}_\ell,\check{\mathtt u}_r)$ and $(\check{\mathtt u}_r,u_r)$, which are a possibly null S, a NS and a possibly null CD, respectively.
As a consequence
\begin{align*}
\Delta\tv_v &= 2(v_\ell-\hat{\mathtt v}_\ell) \geq 0,& 
\Delta\hat{\Upsilon}_n &= -(\hat{\mathtt v}_m-\hat{\mathtt v}_\ell) -(\hat{\mathtt w}_\ell-\hat{\mathtt w}_m) < -(\hat{\mathtt v}_m-\hat{\mathtt v}_\ell) < 0,
\\
\Delta\tv_w &= 
\left\{
\begin{array}{@{}ll@{}}
2(\mathtt{w}_r-\check{\mathtt w}_r)&\text{if }f(u_m) = F
\\
0&\text{if }f(u_m) < F
\end{array}\right\}
\geq 0,&
\Delta\check{\Upsilon}_n &= -(v_r-v_m) -(\check{\mathtt w}_m-\check{\mathtt w}_r) <  -(\check{\mathtt w}_m-\check{\mathtt w}_r) < 0,
\end{align*}
and therefore
\begin{align*}
\Delta\sharp &\in\{-1,0,1\},&
\Delta\mathcal{T}_n&< -2(\hat{\mathtt v}_m-v_\ell) + 
\left\{
\begin{array}{@{}ll@{}}
0&\text{if }f(u_m) = F
\\
-2(\check{\mathtt w}_m-\check{\mathtt w}_r)&\text{if }f(u_m) < F
\end{array}\right\}
\leq0.
\end{align*}

\item[\textbf{CD-NS$\pmb{_F^+}$}]
$(u_\ell,u_m)$ is a CD and $(u_m,u_r)$ is a NS. 
In this case $v_\ell = v_m < v_r$, $\mathtt{w}_r < \min\{\mathtt{w}_\ell , \mathtt{w}_m\}$, $\mathtt{w}_\ell \neq \mathtt{w}_m$, $f(u_m) = f(u_r) = F \neq f(u_\ell)$ and $u_\ell$, $u_m$ $u_r \in \Omega_{\rm c}$.
\\\noindent
If $f(u_\ell) < F$, then $\mathcal{R}_{F,n}[u_\ell,u_r]$ has a fan of RSs from $u_\ell$ to $\hat{\mathtt u}_\ell$ and a NS $(\hat{\mathtt u}_\ell,u_r)$; as a consequence $\Delta\tv_w = -2(\mathtt{w}_m-\mathtt{w}_\ell) < 0 = \Delta\tv_v = \Delta\hat{\Upsilon}_n = \Delta\check{\Upsilon}_n$, therefore $\Delta\sharp \in [0,2^n-2]$ and $\Delta\mathcal{T}_n = -2(\mathtt{w}_m-\mathtt{w}_\ell) < 0$.
\\
If $f(u_\ell) > F$, then $\hat{\mathtt v}_\ell < v_\ell = v_m = \hat{\mathtt v}_m$, $\hat{\mathtt w}_\ell = \mathtt{w}_\ell > \mathtt{w}_m = \hat{\mathtt w}_m$, $\mathcal{R}_{F,n}[u_\ell,u_r]$ has a two waves $(u_\ell,\hat{\mathtt u}_\ell)$ and $(\hat{\mathtt u}_\ell,u_r)$, that are a S and a NS, respectively; as a consequence
\begin{align*}
\Delta\tv_v &= 2(v_\ell-\hat{\mathtt v}_\ell) > 0,& 
\Delta\hat{\Upsilon}_n &= -(v_\ell-\hat{\mathtt v}_\ell) -(\mathtt{w}_\ell-\mathtt{w}_m) < 0,
\\
\Delta\tv_w &= 0,&
\Delta\check{\Upsilon}_n &= 0,
\end{align*}
therefore $\Delta\sharp = 0$ and $\Delta\mathcal{T}_n = -2(\mathtt{w}_\ell-\mathtt{w}_m) < 0$.

\item[\textbf{CD-PT$\pmb{_F^+}$}]
$(u_\ell,u_m)$ is a CD and $(u_m,u_r)$ is a PT. 
In this case $\mathtt{w}_m < w^- \leq w_F = \hat{\mathtt w}_m < \mathtt{w}_\ell$, $v_\ell = v_m = \hat{\mathtt v}_m = V > v_r = \check{\mathtt v}_r > \hat{\mathtt v}_\ell$, $\mathtt{w}_m < w^- \leq \mathtt{w}_r < \check{\mathtt w}_r < \mathtt{w}_\ell = \hat{\mathtt w}_\ell$, $f(u_\ell) > f({\mathtt u}_*) > F \geq \max\{f(u_m) , f(u_r)\}$, $u_\ell \in \Omega_{\rm f}^+$, $u_m \in \Omega_{\rm f}^-$ and $u_r \in \Omega_{\rm c}$.
$\mathcal{R}_{F,n}[u_\ell,u_r]$ has three waves $(u_\ell,\hat{\mathtt u}_\ell)$, $(\hat{\mathtt u}_\ell,\check{\mathtt u}_r)$ and $(\check{\mathtt u}_r,u_r)$ that are a S, a NS and a CD, respectively. 
As a consequence
\begin{align*}
\Delta\tv_v &= 2(v_r-\hat{\mathtt v}_\ell) > 0,& 
\Delta\hat{\Upsilon}_n &= -(V-\hat{\mathtt v}_\ell) -(\mathtt{w}_\ell-w_F) < 0,
\\
\Delta\tv_w &= -2(\mathtt{w}_r-\mathtt{w}_m) < 0,&
\Delta\check{\Upsilon}_n &= 0,
\end{align*}
therefore $\Delta\sharp = 1$ and $\Delta\mathcal{T}_n < -2(V-v_r) -2(\mathtt{w}_\ell-w_F) < 0$.

\item[\textbf{NS-S$\pmb{_F^+}$}]
$(u_\ell,u_m)$ is a NS and $(u_m,u_r)$ is a S. 
In this case $v_m > v_r = \check{\mathtt v}_r > v_\ell$, $\mathtt{w}_\ell > \check{\mathtt w}_r > \mathtt{w}_m = \mathtt{w}_r \geq w_F \geq w^-$, $f({\mathtt u}_*) > F = f(u_\ell) = f(u_m) > f(u_r)$ and $u_\ell$, $u_m$, $u_r \in \Omega_{\rm c}$.
$\mathcal{R}_{F,n}[u_\ell,u_r]$ has two waves $(u_\ell , \check{\mathtt u}_r)$ and $(\check{\mathtt u}_r , u_r)$ that are a NS and a CD, respectively. 
As a consequence $\Delta\tv_v = -2(v_m-v_r) < 0 = \Delta\tv_w = \Delta\hat{\Upsilon}_n = \Delta\check{\Upsilon}_n$, therefore $\Delta\sharp = 0$ and $\Delta\mathcal{T}_n = -2(v_m-v_r) < 0$.

\item[\textbf{NS-RS$\pmb{_F^+}$}]
$(u_\ell,u_m)$ is a NS and $(u_m,u_r)$ is a RS. 
In this case $v_\ell < v_m = \check{\mathtt v}_m < v_r = \check{\mathtt v}_r$, $\mathtt{w}_\ell > \mathtt{w}_m =  \check{\mathtt w}_m = \mathtt{w}_r > \check{\mathtt w}_r$, $f(\mathtt{u}_*) > f(u_r) > F = f(u_\ell) = f(u_m)$ and $u_\ell$, $u_m$, $u_r \in \Omega_{\rm c}$.
$\mathcal{R}_{F,n}[u_\ell,u_r]$ has two waves $(u_\ell,\check{\mathtt u}_r)$ and $(\check{\mathtt u}_r,u_r)$ that are a NS and a CD, respectively.
As a consequence
\begin{align*}
\Delta\tv_v &= 0,& 
\Delta\hat{\Upsilon}_n &= 0,& 
\\
\Delta\tv_w &= 2(\mathtt{w}_r-\check{\mathtt w}_r) > 0,&
\Delta\check{\Upsilon}_n &= -(v_r-v_m) -(\mathtt{w}_r-\check{\mathtt w}_r) < 0,
\end{align*}
therefore $\Delta\sharp = 0$ and $\Delta\mathcal{T}_n = -2(v_r-v_m) < 0$.

\end{enumerate}

\noindent
If $F \in [0,f_{\rm c}^-)$, then one of the following cases occurs:

\begin{enumerate}[wide=0pt]

\item[\textbf{CD-S$\pmb{_F^-}$}]
$(u_\ell,u_m)$ is a CD and $(u_m,u_r)$ is a S. 
In this case $\hat{\mathtt v}_m \geq v_\ell = v_m > v_r = \check{\mathtt v}_r > \hat{\mathtt v}_\ell$, $\mathtt{w}_\ell = \hat{\mathtt w}_\ell > \check{\mathtt w}_r > \mathtt{w}_m = \hat{\mathtt w}_m = \mathtt{w}_r \geq w^-$, $f(u_\ell) > f({\mathtt u}_*) > F \geq f(u_m) > f(u_r)$ and $u_\ell$, $u_m$, $u_r \in \Omega_{\rm c}$.
$\mathcal{R}_{F,n}[u_\ell,u_r]$ has three waves $(u_\ell,\hat{\mathtt u}_\ell)$, $(\hat{\mathtt u}_\ell,\check{\mathtt u}_r)$ and $(\check{\mathtt u}_r,u_r)$, which are a S, a NS and a CD, respectively.
As a consequence
\begin{align*}
\Delta\tv_v &= 2(v_r-\hat{\mathtt v}_\ell) > 0,& 
\Delta\hat{\Upsilon}_n &= -(\hat{\mathtt v}_m-\hat{\mathtt v}_\ell) -(\mathtt{w}_\ell-\mathtt{w}_m) < 0,
\\
\Delta\tv_w &= 0,&
\Delta\check{\Upsilon}_n &= 0,
\end{align*}
therefore $\Delta\sharp = 1 $ and $\Delta\mathcal{T}_n = -2(\hat{\mathtt v}_m-v_r) -2(\mathtt{w}_\ell-\mathtt{w}_m) < 0$.

\item[\textbf{CD-RS$\pmb{_F^-}$}]
$(u_\ell,u_m)$ is a CD and $(u_m,u_r)$ is a RS. 
In this case $\hat{\mathtt v}_\ell \leq v_\ell = v_m \leq \hat{\mathtt v}_m$, $v_m = \check{\mathtt v}_m < v_r \leq \check{\mathtt v}_r$, $\mathtt{w}_\ell = \hat{\mathtt w}_\ell \geq \check{\mathtt w}_m \geq \mathtt{w}_m = \mathtt{w}_r \geq w^-$, $\mathtt{w}_\ell > \mathtt{w}_m$, $f(u_m) < f(u_r) \leq f(u_\ell) < f({\mathtt u}_*)$, $f(u_m) \leq F < f({\mathtt u}_*)$ and $u_\ell$, $u_m$, $u_r \in \Omega_{\rm c}$.
$\mathcal{R}_{F,n}[u_\ell,u_r]$ has at most three waves $(u_\ell,\hat{\mathtt u}_\ell)$, $(\hat{\mathtt u}_\ell,\check{\mathtt u}_r)$ and $(\check{\mathtt u}_r,u_r)$, which are a possibly null S, a NS and a possibly null CD or PT, respectively.
As a consequence
\begin{align*}
\Delta\tv_v &= 2(v_\ell-\hat{\mathtt v}_\ell) + 2(\check{\mathtt v}_r-v_r) \geq 0,& 
\Delta\hat{\Upsilon}_n &= -(\hat{\mathtt v}_m-\hat{\mathtt v}_\ell) -(\mathtt{w}_\ell-\mathtt{w}_m) < 0,
\\
\Delta\tv_w &=
\left\{
\begin{array}{@{}ll@{}}
2(\mathtt{w}_r-\check{\mathtt w}_r)&\text{if }f(u_m) = F
\\
0&\text{if }f(u_m) < F
\end{array}\right\}
\geq 0,&
\Delta\check{\Upsilon}_n &= -(\check{\mathtt v}_r-v_m) 
- \left\{
\begin{array}{@{}ll@{}}
(\mathtt{w}_r-\check{\mathtt w}_r)&\text{if }f(u_m) = F
\\
(\check{\mathtt w}_m-\check{\mathtt w}_r)&\text{if }f(u_m) < F
\end{array}\right\}
< 0,
\end{align*}
and therefore
\begin{align*}
\Delta\sharp &\in \{-1,0,1\},&
\Delta\mathcal{T}_n&=
-2(\hat{\mathtt v}_m-v_\ell) 
-2(v_r-v_m)
-2(\mathtt{w}_\ell-\mathtt{w}_m)
-\left\{
\begin{array}{@{}ll@{}}
0&\text{if }f(u_m) = F
\\
2(\check{\mathtt w}_m-\check{\mathtt w}_r)&\text{if }f(u_m) < F
\end{array}\right\}
<0.
\end{align*}

\item[\textbf{CD-NS$\pmb{_F^-}$}]
$(u_\ell,u_m)$ is a CD and $(u_m,u_r)$ is a NS. 
In this case $v_\ell = v_m = \hat{\mathtt v}_m < v_r $, $\mathtt{w}_r < \min\{\mathtt{w}_\ell , \mathtt{w}_m\}$, $\mathtt{w}_\ell = \hat{\mathtt w}_\ell$, $\mathtt{w}_m = \hat{\mathtt w}_m$, $f(u_m) = f(u_r) = F \neq f(u_\ell)$ and $u_\ell$, $u_m \in \Omega_{\rm c}$.
\\
If $f(u_\ell) < F$, then $v_r > \hat{\mathtt v}_\ell > v_\ell = v_m$, $\mathtt{w}_r < \mathtt{w}_\ell < \mathtt{w}_m$ and $\mathcal{R}_{F,n}[u_\ell,u_r]$ has a fan of RSs from $u_\ell$ to $\hat{\mathtt u}_\ell$ and a NS $(\hat{\mathtt u}_\ell,u_r)$; as a consequence $\Delta\tv_w = -2(\mathtt{w}_m-\mathtt{w}_\ell) < 0 = \Delta\tv_v  = \Delta\hat{\Upsilon}_n = \Delta\check{\Upsilon}_n$, therefore $\Delta\sharp \in [0,2^n-2]$ and $\Delta\mathcal{T}_n = -2(\mathtt{w}_m-\mathtt{w}_\ell) < 0$.
\\
If $f(u_\ell) > F$, then $v_r > v_\ell = v_m = \hat{\mathtt v}_m > \hat{\mathtt v}_\ell$, $\mathtt{w}_\ell > \mathtt{w}_m > \mathtt{w}_r$ and $\mathcal{R}_{F,n}[u_\ell,u_r]$ has a two waves $(u_\ell,\hat{\mathtt u}_\ell)$ and $(\hat{\mathtt u}_\ell,u_r)$, that are a S and a NS, respectively; as a consequence
\begin{align*}
\Delta\tv_v &= 2(v_\ell-\hat{\mathtt v}_\ell) > 0,& 
\Delta\hat{\Upsilon}_n &= -(v_\ell-\hat{\mathtt v}_\ell) -(\mathtt{w}_\ell-\mathtt{w}_m) < 0,
\\
\Delta\tv_w &= 0,&
\Delta\check{\Upsilon}_n &= 0,
\end{align*}
therefore $\Delta\sharp = 0$ and $\Delta\mathcal{T}_n = -2(\mathtt{w}_\ell-\mathtt{w}_m) < 0$.

\item[\textbf{CD-PT$\pmb{_F^-}$}]
$(u_\ell,u_m)$ is a CD and $(u_m,u_r)$ is a PT. 
In this case $v_\ell = v_m = \hat{\mathtt v}_m = V > v_r = \check{\mathtt v}_r > \hat{\mathtt v}_\ell$, $\mathtt{w}_\ell = \hat{\mathtt w}_\ell > \check{\mathtt w}_r \geq \mathtt{w}_r \geq w^- > w_F = \hat{\mathtt w}_m \geq \mathtt{w}_m$, $f(u_\ell) > f({\mathtt u}_*) > F \geq \max\{f(u_m) , f(u_r)\}$, $u_\ell \in \Omega_{\rm f}^+$, $u_m \in \Omega_{\rm f}^-$ and $u_r \in \Omega_{\rm c}$.
$\mathcal{R}_{F,n}[u_\ell,u_r]$ has at most three waves $(u_\ell,\hat{\mathtt u}_\ell)$, $(\hat{\mathtt u}_\ell,\check{\mathtt u}_r)$ and $(\check{\mathtt u}_r,u_r)$ that are a S, a NS and a possibly null CD, respectively. 
As a consequence
\begin{align*}
\Delta\tv_v &= 2(v_r-\hat{\mathtt v}_\ell) > 0,& 
\Delta\hat{\Upsilon}_n &= -(v_m-\hat{\mathtt v}_\ell) -(\mathtt{w}_\ell-w_F) < 0,
\\
\Delta\tv_w &= -2(\mathtt{w}_r-\mathtt{w}_m) < 0,&
\Delta\check{\Upsilon}_n &= 0,
\end{align*}
therefore $\Delta\sharp \in \{0,1\}$ and $\Delta\mathcal{T}_n = -2(v_m-v_r) -2(\mathtt{w}_r-\mathtt{w}_m) -2(\mathtt{w}_\ell-w_F) <0$.

\item[\textbf{NS-S$\pmb{_F^-}$}]
$(u_\ell,u_m)$ is a NS and $(u_m,u_r)$ is a S. 
In this case $v_m > v_r = \check{\mathtt v}_r > v_\ell$, $\mathtt{w}_\ell > \check{\mathtt w}_r > \mathtt{w}_m = \mathtt{w}_r \geq w^-$, $f(u_\ell) = f(u_m) = F > f(u_r)$ and $u_\ell$, $u_m$, $u_r \in \Omega_{\rm c}$.
$\mathcal{R}_{F,n}[u_\ell,u_r]$ has two waves $(u_\ell , \check{\mathtt u}_r)$ and $(\check{\mathtt u}_r , u_r)$ that are a NS and a CD, respectively. 
As a consequence $\Delta\tv_v = -2(v_m-v_r) < 0 = \Delta\tv_w = \Delta\hat{\Upsilon}_n = \Delta\check{\Upsilon}_n$, therefore $\Delta\sharp = 0$ and $\Delta\mathcal{T}_n = -2(v_m-v_r) < 0$.

\item[\textbf{NS-RS$\pmb{_F^-}$}]
$(u_\ell,u_m)$ is a NS and $(u_m,u_r)$ is a RS. 
In this case $v_\ell < v_m = \check{\mathtt v}_m < v_r \leq \check{\mathtt v}_r$, $\mathtt{w}_\ell > \mathtt{w}_m = \check{\mathtt w}_m = \mathtt{w}_r > \check{\mathtt w}_r$, $f(u_r) > F = f(u_\ell) = f(u_m)$ and $u_\ell$, $u_m$, $u_r \in \Omega_{\rm c}$.
$\mathcal{R}_{F,n}[u_\ell,u_r]$ has two waves $(u_\ell,\check{\mathtt u}_r)$ and $(\check{\mathtt u}_r,u_r)$ that are a NS and either a PT or a CD, respectively.
As a consequence
\begin{align*}
\Delta\tv_v &= 2(\check{\mathtt v}_r-v_r) \geq 0,& 
\Delta\hat{\Upsilon}_n &= 0,
\\
\Delta\tv_w &= 2(\mathtt{w}_r-\check{\mathtt w}_r) > 0,&
\Delta\check{\Upsilon}_n &= -(\check{\mathtt v}_r-v_m) -(\mathtt{w}_r-\check{\mathtt w}_r) < 0,
\end{align*}
therefore $\Delta\sharp = 0$ and $\Delta\mathcal{T}_n = -2(v_r-v_m) < 0$.
Notice that $\check{\mathtt v}_r > v_r$ if and only if $\mathtt{w}_m = \mathtt{w}_r = w^-$ and $\check{\mathtt v}_r = V > v_r > v_m = v_F^+$.

\item[\textbf{NS-PT$\pmb{_F^-}$}]
$(u_\ell,u_m)$ is a NS and $(u_m,u_r)$ is a PT. 
In this case $v_m = V > v_r = \check{\mathtt v}_r > v_\ell$, $\mathtt{w}_\ell > \check{\mathtt w}_r > \mathtt{w}_r = w^- > \mathtt{w}_m$, $f({\mathtt u}_*) > f(u_\ell) = f(u_m) = F > f(u_r)$, $u_\ell$, $u_r \in \Omega_{\rm c}$ and $u_m \in \Omega_{\rm f}^-$.
$\mathcal{R}_{F,n}[u_\ell,u_r]$ has two waves $(u_\ell,\check{\mathtt u}_r)$ and $(\check{\mathtt u}_r,u_r)$ that are a NS and a CD, respectively.
As a consequence
\begin{align*}
\Delta\tv_v &= -2(V-v_r) < 0,& 
\Delta\hat{\Upsilon}_n &= 0,
\\
\Delta\tv_w &= -2(w^--\mathtt{w}_m) < 0,&
\Delta\check{\Upsilon}_n &= 0,
\end{align*}
therefore $\Delta\sharp = 0$ and $\Delta\mathcal{T}_n = -2(V-v_r) -2(w^--\mathtt{w}_m) < 0$.

\end{enumerate}

\end{itemize}
\noindent
This concludes the proof.
\end{proof}

\begin{table}
\begin{center}
\renewcommand{\arraystretch}{1.25}
\begin{tabular}{|c|l|c|c|}
\hline  
Interaction&Result&$\Delta\sharp$&$\Delta\mathcal{T}_n$\\
\hline 
CD$_F^+$&(S,NS,CD), (S,NS)&$\in\{1,2\}$&$<0$\\
\hline
RS$_F^+$&(NS,CD)&$=1$&$<0$\\
\hline
CD$_F^-$&(S,NS,CD), (S,NS), (PT,NS,CD), (PT,NS)&$\in\{1,2\}$&$<0$\\
\hline
RS$_F^-$&(NS,PT), (NS,CD)&$=1$&$<0$\\
\hline
CD-S&(S,CD), (PT,CD), PT&$\leq0$&$=0$\\
\hline
CD-RS&(RS,CD)&$=0$&$=0$\\
\hline
CD-PT&(PT,CD), PT, (S,CD), S&$\leq0$&$\leq0$\\
\hline
S-S&S&$<0$&$=0$\\
\hline
S-RS&S&$<0$&$<0$\\
\hline
RS-S&S&$<0$&$<0$\\
\hline
PT-S&PT&$<0$&$=0$\\
\hline
PT-RS&PT, CD&$<0$&$<0$\\
\hline
CD-S$_0$&(S,CD), (PT,CD), PT&$\leq0$&$\leq0$\\
\hline
CD-RS$_0$&(RS,CD)&$=0$&$\leq0$\\
\hline
CD-NS$_0$&(RSs,CD), RSs&$\in[-1,2^n-1]$&$<0$\\
\hline
CD-PT$_0$&(PT,CD), PT, (S,CD), S&$\leq0$&$\leq0$\\
\hline
S-S$_0$&S&$<0$&$=0$\\
\hline
S-RS$_0$&S&$<0$&$<0$\\
\hline
RS-S$_0$&S&$<0$&$<0$\\
\hline
NS-S$_0$&(S,CD), CD&$\leq0$&$<0$\\
\hline
NS-PT$_0$&(S,CD), S, CD&$\leq0$&$<0$\\
\hline
PT-S$_0$&PT&$<0$&$=0$\\
\hline
PT-RS$_0$&PT, CD&$<0$&$<0$\\
\hline
PT-NS$_0$&CD, PT&$<0$&$<0$\\
\hline
CD-S$_F^+$&(S,NS,CD)&$=1$&$<0$\\
\hline
CD-RS$_F^+$&(S,NS,CD), (NS,CD), (S,NS), NS&$\in\{-1,0,1\}$&$<0$\\
\hline
CD-NS$_F^+$&(RSs,NS), (S,NS)&$\in[0,2^n-2]$&$<0$\\
\hline
CD-PT$_F^+$&(S,NS,CD)&$=1$&$<0$\\
\hline
NS-S$_F^+$&(NS,CD)&$=0$&$<0$\\
\hline
NS-RS$_F^+$&(NS,CD)&$=0$&$<0$\\
\hline
CD-S$_F^-$&(S,NS,CD)&$=1$&$<0$\\
\hline
CD-RS$_F^-$&(S,NS,CD), (S,NS,PT), (NS,CD), (NS,PT), (S,NS), NS&$\in\{-1,0,1\}$&$<0$\\
\hline
CD-NS$_F^-$&(RSs,NS), (S,NS)&$\in[0,2^n-2]$&$<0$\\
\hline
CD-PT$_F^-$&(S,NS,CD), (S,NS)&$\in\{0,1\}$&$<0$\\
\hline
NS-S$_F^-$&(NS,CD)&$=0$&$<0$\\
\hline
NS-RS$_F^-$&(NS,PT), (NS,CD)&$=0$&$<0$\\
\hline
NS-PT$_F^-$&(NS,CD)&$=0$&$<0$\\
\hline 
\end{tabular}
\end{center} 
\caption{An overview of the interactions considered in the proof of Proposition~\ref{p:interest}.}
\label{tab:inter}
\end{table}

In Table~\ref{tab:inter} we collect the most relevant possible interactions considered in the proof of Proposition~\ref{p:interest} and list the corresponding possible results in terms of wave types, $\Delta\sharp$ and $\Delta\mathcal{T}_n$.

Beside the bound on the number of wave-fronts proved in Proposition~\ref{p:interest}, we need to bound also the number of interactions.
This is the aim of the next proposition, which together with Proposition~\ref{p:interest} ensure the global existence of $u_n$.
We underline that for any interaction $\Delta\sharp\leq 2^n-1$, see Table~\ref{tab:inter}.
\begin{proposition}\label{p:inter}
For any fixed $n \in \N$ sufficiently large and $u^o_n \in \mathbf{PC}(\R;\mathcal{G}_n)$, we have that the number of interactions in $(0,\infty)$ is bounded.
In particular $u_n$ is globally defined.
\end{proposition}
\begin{proof}
From what we already show in the proof of Proposition~\ref{p:interest}, see Table~\ref{tab:inter}, we deduce that
\[
t\mapsto 2^n \, \dfrac{\mathcal{T}_n(t)}{\varepsilon_n} + \sharp(t)
\]
strictly decreases after any interaction, except the following cases.
\begin{equation}\label{e:Korn}
\begin{minipage}{.85\textwidth}
A CD $(u_\ell,u_m)$ interacts with $(u_m,u_r)$ and one of the following conditions is satisfied:
\begin{itemize}[wide=0pt]
\item
$(u_m,u_r)$ is a S and $\mathtt{w}_\ell = w^--1$;
\item
$(u_m,u_r)$ is a S and $w^--1<\mathtt{w}_\ell \leq w^- = \mathtt{w}_r$;
\item
$(u_m,u_r)$ is a RS;
\item
$(u_m,u_r)$ is a PT and $\mathtt{w}_\ell > w^-$.
\end{itemize}
\end{minipage}
\end{equation}
For this reason it remains to bound the number of only the above type of interactions.
We observe that the number of waves of $u_n$ do not change after interactions as in \eqref{e:Korn}.
This implies that the number of waves is uniformly bounded.
We also observe that any interaction as in \eqref{e:Korn} has exactly one incoming CD and exactly one outgoing CD.
Since no wave can reach any CD from the left (and then possibly have with it an interaction as in \eqref{e:Korn}), we have that as long as a CD remains a CD (possible further interactions involving it have to be taken into account), it can interact only once with another wave W (or with waves generated by \emph{further} interactions involving W), moreover in this case W is slower then such CD and is not another CD.
Since furthermore we already know that the number of waves is uniformly bounded, there can be only finitely many interactions involving CDs.
It is therefore now clear that also the number of the interactions described in \eqref{e:Korn} is bounded.
\end{proof}

\subsection{Convergence}\label{s:conv}

We first observe that
\[
|\rho_\ell-\rho_r| \leq L \, \bigl( |v_\ell-v_r| + |\mathtt{w}_\ell-\mathtt{w}_r| \bigr)
\]
where $L \doteq \max \{ \rho^-, \|1/p'\|_{\L\infty([p^{-1}(w^-),p^{-1}(w^+)];\R)} \}$ because
\[
\rho_{\ell,r} =
\begin{cases}
p^{-1}(\mathtt{w}_{\ell,r}-v_{\ell,r})&\text{if }\mathtt{w}_{\ell,r}\in[w^-,w^+],
\\
(\mathtt{w}_{\ell,r}+1-w^-) \, \rho^-&\text{if }\mathtt{w}_{\ell,r}\in[w^--1,w^-).
\end{cases}
\]
As a consequence $\tv(\rho) \leq L \, (\tv(v)+\tv(\mathtt{w}))$, hence
\[
\tv(u) \leq (1+L) \, \bigl(\tv(v)+\tv(\mathtt{w})\bigr).
\]
Moreover, by Proposition~\ref{p:interest} and \eqref{e:apprini} we have that for any $t>0$
\[
\tv\bigl(v_n(t,\cdot)\bigr) + \tv\bigl(\mathtt{w}_n(t,\cdot)\bigr) \leq 
\mathcal{T}_n(t) \leq \mathcal{T}_n(0) \leq 
\tv(v^o) + \tv(\mathtt{w}^o) + 2 \, C \, \Bigl( \hat{\Upsilon}(u^o) + \check{\Upsilon}(u^o) \Bigr).
\]
As a consequence $\tv(u_n)$ is bounded by
\begin{equation}
C^o_F \doteq
(1+L) \, \Bigl[ \tv(v^o) + \tv(\mathtt{w}^o)  + 2 \, C \, \Bigl( \hat{\Upsilon}(u^o) + \check{\Upsilon}(u^o) \Bigr) \Bigr].
\label{e:KingCrimson}
\end{equation}

Since $u^o_n$ takes values in $\Omega$, for any $t>0$ we have that also $u_n(t,\cdot)$ takes values in $\Omega$, hence
\[
\|u_n(t,\cdot)\|_{\L\infty(\R;\Omega)} \leq R+V.
\]
Moreover
\begin{equation}\label{e:StevenWilson}
\|u_n(t,\cdot) - u_n(s,\cdot)\|_{\L1(\R;\Omega)} \leq L^o_F \, |t - s|,
\end{equation}
with $L^o_F \doteq C^o_F \, \max\{V, R\, p'(R) \}$.
Indeed, if no interaction occurs for times between $t$ and $s$, then
\begin{align*}
\|u_n(t,\cdot) - u_n(s,\cdot)\|_{\L1(\R;\Omega)} \leq &\
\sum_{i\in\mathsf{D}(t)}
\Bigl|(t - s) \, \dot{\delta}_n^i(t) \, \Bigl(\rho_n\bigl(t,\delta_n^i(t)_-\bigr)-\rho_n\bigl(t,\delta_n^i(t)_+\bigr)\Bigr) \Bigr|
\\&\
+\sum_{i\in\mathsf{D}(t)}
\Bigl|(t - s) \, \dot{\delta}_n^i(t) \, \Bigl(v_n\bigl(t,\delta_n^i(t)_-\bigr)-v_n\bigl(t,\delta_n^i(t)_+\bigr)\Bigr) \Bigr|
 \leq L^o_F \, |t - s|,
\end{align*}
where $\delta_n^i(t) \in \R$, $i\in\mathsf{D}(t)\subset\N$, are the positions of the discontinuities of $u_n(t,\cdot)$.
The case when one or more interactions take place for times between $t$ and $s$ is similar, because by the finite speed of propagation of the waves the map $t \mapsto u_n(t,\cdot)$ is $\L1$-continuous across interaction times.

Thus, by applying Helly's Theorem, the approximate solutions $(u_n)_n$ converge (up to a subsequence) in $\Lloc1(\R_+\times \R; \Omega)$ to a function $u \in \L\infty(\R_+ ;\BV(\R ; \Omega)) \cap \C0(\R_+; \Lloc1(\R ; \Omega))$ and the limit satisfies the estimates in \eqref{e:estimates}. 

\begin{proposition}\label{p:JacobCollier}
Let $u^o \in \L1\cap\BV(\R;\Omega)$ and $F\in [0,f_{\rm c}^+]$ satisfy \ref{H.1} or \ref{H.2}.
If $u$ is a limit of the approximate solutions $(u_n)_n$ constructed in Section~\ref{s:approxsolR}, then $u$ is a solution to constrained Cauchy problem \eqref{eq:system}, \eqref{eq:initdat}, \eqref{eq:const} in the sense of Definition~\ref{def:solconstmod}.
\end{proposition}

\begin{proof}
We consider separately the conditions listed in Definition~\ref{def:solconstmod}.
\begin{enumerate}[label={\textbf{(CS.\arabic*)}},wide=0pt]\setlength{\itemsep}{0cm}%

\item[\ref{CS1}] 
The initial condition \eqref{eq:initdat} holds by \eqref{e:estimates}, \eqref{e:StevenWilson} and the $\Lloc1$-convergence of $u_n$ to $u$.

\item[\ref{CS2}] 
We prove now \eqref{e:AdrianBelew0}, that is for any test function $\phi \in \Cc\infty((0,\infty)\times\R; \R)$ we have
\[
\int_0^\infty
\int_{\R} 
\Bigl( \rho \, \phi_t
+ f(u)\, \phi_x \Bigr) \,
\d x \, \d t
=0.
\]
Choose $T>0$ such that $\phi(t,x) = 0$ whenever $t \geq T$.
Since $u_n$ is uniformly bounded and $f$ is uniformly continuous on bounded sets, it is sufficient to prove that
\begin{equation}\label{e:BlindMelon}
\int_0^T
\int_{\R} 
\Bigl( \rho_n \, \phi_t
+ f(u_n)\, \phi_x \Bigr) \,
\d x \, \d t
\to0.
\end{equation}
By the Green-Gauss formula the double integral above can be written as
\[
\int_0^T \sum_{i\in\mathsf{D}(t)}
\Bigl( \dot{\delta}_n^i(t) \, \Delta \rho_n^i(t) - \Delta f_n^i(t) \Bigr) \, \phi\bigl(t,\delta_n^i(t)\bigr) \, \d t ,
\]
where
\begin{align*}
\Delta \rho_n^i(t) &\doteq 
\rho_n\bigl(t,\delta_n^i(t)_+\bigr) - \rho_n\bigl(t,\delta_n^i(t)_-\bigr),&
\Delta f_n^i(t) &\doteq 
f\Bigl(u_n\bigl(t,\delta_n^i(t)_+\bigr)\Bigr) - f\Bigl(u_n\bigl(t,\delta_n^i(t)_-\bigr)\Bigr).
\end{align*}
By construction any discontinuity of $u_n(t,\cdot)$ satisfies the first Rankine-Hugoniot condition \eqref{e:RH1}, therefore
\[
\dot{\delta}_n^i(t) \, \Delta \rho_n^i(t) - \Delta f_n^i(t)=0,\qquad i\in\mathsf{D}(t),
\]
and \eqref{e:BlindMelon} is trivial.
 
The proof of \eqref{e:AdrianBelew1} is analogous because by construction any discontinuity of $u_n(t,\cdot)$ away from $x=0$ satisfies also the second Rankine-Hugoniot condition \eqref{e:RH2}.

\item[\ref{CS3}] 
We prove now \eqref{e:Deftones}, namely that for any $k \in [0,V]$ and $\phi \in \Cc\infty((0,\infty)\times \R;\R)$ such that $\phi(\cdot,0)\equiv0$ and $\phi \geq0$ we have
\[
\int_0^\infty
\int_{\R} \Bigl(\mathtt{E}^k(u) \, \phi_t + \mathtt{Q}^k(u) \, \phi_x\Bigr) \, \d x \, \d t \geq 0,
\]
where
\begin{align*}
\mathtt{E}^k(u) &\doteq \begin{cases}
0 & \hbox{if } v \geq k,\\
\dfrac{\rho }{p^{-1}\bigl(\mathtt{W}(u)-k\bigr)} - 1 & \hbox{if } v < k,
\end{cases}&
\mathtt{Q}^k(u) &\doteq \begin{cases}
0 & \hbox{if } v \geq k,\\
\dfrac{f(u)}{p^{-1}\bigl(\mathtt{W}(u)-k\bigr)} - k & \hbox{if } v < k.
\end{cases}
\end{align*}
Choose $T>0$ such that $\phi(t,x) = 0$ whenever $t\geq T$.
By the a.e.~convergence of $u_n$ to $u$ and the uniform continuity of $\mathtt{E}^k$ and $\mathtt{Q}^k$, it is sufficient to prove that
\begin{equation}\label{e:PorcupineTree}
\liminf_{n\to\infty} \int_0^T
\int_{\R} \Bigl(\mathtt{E}^k(u_n) \, \phi_t + \mathtt{Q}^k(u_n) \, \phi_x\Bigr) \, \d x \, \d t \geq 0.
\end{equation}
By the Green-Gauss formula the double integral above can be written as
\[
\int_0^T \sum_{i\in\mathsf{D}(t)}
\Bigl( \dot{\delta}_n^i(t) \, \Delta\mathtt{E}_n^{k,i}(t) - \Delta\mathtt{Q}_n^{k,i}(t) \Bigr) \, \phi\bigl(t,\delta_n^i(t)\bigr) \, \d t ,
\]
where
\begin{align*}
\Delta\mathtt{E}_n^{k,i}(t) &\doteq 
\mathtt{E}^k\Bigl(u_n\bigl(t,\delta_n^i(t)_+\bigr)\Bigr)
-\mathtt{E}^k\Bigl(u_n\bigl(t,\delta_n^i(t)_-\bigr)\Bigr),&
\Delta\mathtt{Q}_n^{k,i}(t) &\doteq 
\mathtt{Q}^k\Bigl(u_n\bigl(t,\delta_n^i(t)_+\bigr)\Bigr)
-\mathtt{Q}^k\Bigl(u_n\bigl(t,\delta_n^i(t)_-\bigr)\Bigr).
\end{align*}
To estimate the above integral we have to distinguish the following cases.

\begin{itemize}[wide=0pt]

\item
If the $i$th discontinuity is a PT, then we let $x \doteq \delta_n^i(t)$ and observe that
\begin{align*}
\rho_n(t,x_-) &< \min\bigl\{\rho_n(t,x_+),p^{-1}(w^--k)\bigr\},&
\dot{\delta}_n^i(t) &= \Lambda\bigl(u_n(t,x_-),u_n(t,x_+)\bigr),
\\
v_n(t,x_-) &= V > v_n(t,x_+),&
\mathtt{W}\bigl(u_n(t,x_-)\bigr) &= w^- \leq \mathtt{w}\bigl(u_n(t,x_+)\bigr) = \mathtt{W}\bigl(u_n(t,x_+)\bigr),
\end{align*}
hence
\begin{align*}
\Delta\mathtt{E}_n^{k,i}(t) &=
\begin{cases}
\dfrac{\rho_n(t,x_+)}{\rho_{n,+}^k} - 1
&\text{if }v_n(t,x_+) < k \leq V,
\\
0
&\text{if }k \leq v_n(t,x_+),
\end{cases}
\\
-\Delta\mathtt{Q}_n^{k,i}(t) &=
\begin{cases}
k - \dfrac{f\bigl(u_n(t,x_+)\bigr)}{\rho_{n,+}^k}
&\text{if }v_n(t,x_+) < k \leq V,
\\
0
&\text{if }k \leq v_n(t,x_+),
\end{cases}
\end{align*}
where $\rho_{n,+}^k \doteq p^{-1}(\mathtt{w}(u_n(t,x_+))-k)$.
If  $v_n(t,x_+) < k \leq V$, then
\begin{align*}
&\
\dot{\delta}_n^i(t) \, \Delta\mathtt{E}_n^{k,i}(t) - \Delta\mathtt{Q}_n^{k,i}(t)
\\=&\ 
\Lambda\bigl(u_n(t,x_-),u_n(t,x_+)\bigr)
\left[
\dfrac{\rho_n(t,x_+)}{\rho_{n,+}^k} - 1
\right]
+ k - \dfrac{f\bigl(u_n(t,x_+)\bigr)}{\rho_{n,+}^k}
\\=&\
\underbrace{\left[
\dfrac{\rho_n(t,x_+)}{\rho_{n,+}^k} - 1
\right]}_{>0}
\,
\underbrace{\left[\vphantom{\dfrac{\rho_n(t,x_+)}{\rho_{n,+}^k}}
\Lambda\bigl(u_n(t,x_-),u_n(t,x_+)\bigr)
- \Lambda\bigl((\rho_{n,+}^k,k),u_n(t,x_+)\bigr) \right]}_{>0}
> 0.
\end{align*}

\item
If the $i$th discontinuity is a CD, then we let $x \doteq \delta_n^i(t)$ and observe that $\dot{\delta}_n^i(t) = v_n(t,x_-) = v_n(t,x_+)$ implies that $\dot{\delta}_n^i(t) \, \Delta\mathtt{E}_n^{k,i}(t) - \Delta\mathtt{Q}_n^{k,i}(t) = 0$.

\item
If the $i$th discontinuity is a S, then we let $x \doteq \delta_n^i(t)$ and observe that 
\begin{align*}
\rho_n(t,x_-) &< \rho_n(t,x_+),&
f\bigl(u_n(t,x_-)\bigr) &> f\bigl(u_n(t,x_+)\bigr),&
\dot{\delta}_n^i(t) &= \Lambda\bigl(u_n(t,x_-),u_n(t,x_+)\bigr)<0,
\\
v_n(t,x_-) &> v_n(t,x_+),&&&
\mathtt{w}_\pm &\doteq \mathtt{w}\bigl(u_n(t,x_-)\bigr) = \mathtt{w}\bigl(u_n(t,x_+)\bigr) \geq w^-,
\end{align*}
hence
\begin{align*}
\Delta\mathtt{E}_n^{k,i}(t) &=
\begin{cases}
\dfrac{\rho_n(t,x_+)-\rho_n(t,x_-)}{p^{-1}(\mathtt{w}_\pm-k)}
&\text{if }v_n(t,x_+) < v_n(t,x_-) < k,
\\[10pt]
\dfrac{\rho_n(t,x_+)}{p^{-1}(\mathtt{w}_\pm-k)} - 1
&\text{if }v_n(t,x_+) < k \leq v_n(t,x_-),
\\[10pt]
0
&\text{if }k \leq v_n(t,x_+) < v_n(t,x_-),
\end{cases}
\\
-\Delta\mathtt{Q}_n^{k,i}(t) &=
\begin{cases}
\dfrac{f\bigl(u_n(t,x_-)\bigr)-f\bigl(u_n(t,x_+)\bigr)}{p^{-1}(\mathtt{w}_\pm-k)}
&\text{if }v_n(t,x_+) < v_n(t,x_-) < k,
\\[10pt]
k - \dfrac{f\bigl(u_n(t,x_+)\bigr)}{p^{-1}(\mathtt{w}_\pm-k)}
&\text{if }v_n(t,x_+) < k \leq v_n(t,x_-),
\\[10pt]
0
&\text{if }k \leq v_n(t,x_+) < v_n(t,x_-).
\end{cases}
\end{align*}
If $k > v_n(t,x_-)$ or $k \leq v_n(t,x_+)$, then it is immediate to see that $\dot{\delta}_n^i(t) \, \Delta\mathtt{E}_n^{k,i}(t) - \Delta\mathtt{Q}_n^{k,i}(t) = 0$.
Furthermore, if $v_n(t,x_+) < k \leq v_n(t,x_-)$, then
\begin{align*}&\
\dot{\delta}_n^i(t) \, \Delta\mathtt{E}_n^{k,i}(t) - \Delta\mathtt{Q}_n^{k,i}(t)
\\=&\
\Lambda\bigl(u_n(t,x_-),u_n(t,x_+)\bigr)
\left[\dfrac{\rho_n(t,x_+)}{p^{-1}(\mathtt{w}_\pm-k)} - 1\right]
+k - \dfrac{f\bigl(u_n(t,x_+)\bigr)}{p^{-1}(\mathtt{w}_\pm-k)}
\\
=&\
\underbrace{\left[\dfrac{\rho_n(t,x_+)}{p^{-1}(\mathtt{w}_\pm-k)} - 1\right]}_{>0}
\,
\underbrace{\left[\vphantom{\dfrac{\rho_n(t,x_+)}{p^{-1}(\mathtt{w}_\pm-k)}}
\Lambda\bigl(u_n(t,x_-),u_n(t,x_+)\bigr)
-
\Lambda\Bigl(\bigl(p^{-1}(\mathtt{w}_\pm-k),k\bigr),u_n(t,x_+)\Bigr)
\right]}_{>0} > 0.
\end{align*}

\item
If the $i$th discontinuity is a RS, then we let $x \doteq \delta_n^i(t)$ and observe that 
\begin{align*}
\rho_n(t,x_-) &> \rho_n(t,x_+),&
f\bigl(u_n(t,x_-)\bigr) &< f\bigl(u_n(t,x_+)\bigr),&
\dot{\delta}_n^i(t) &= \Lambda\bigl(u_n(t,x_-),u_n(t,x_+)\bigr)<0,
\\
v_n(t,x_-) &< v_n(t,x_+),&&&
\mathtt{w}_\pm &\doteq \mathtt{w}\bigl(u_n(t,x_-)\bigr) = \mathtt{w}\bigl(u_n(t,x_+)\bigr) \geq w^-,
\end{align*}
hence
\begin{align*}
\Delta\mathtt{E}_n^{k,i}(t) &=
\begin{cases}
\dfrac{\rho_n(t,x_+)-\rho_n(t,x_-)}{p^{-1}(\mathtt{w}_\pm-k)}
&\text{if }v_n(t,x_-) < v_n(t,x_+) < k,
\\[10pt]
\dfrac{\rho_n(t,x_-)}{p^{-1}(\mathtt{w}_\pm-k)} - 1
&\text{if }v_n(t,x_-) < k \leq v_n(t,x_+),
\\[10pt]
0
&\text{if }k \leq v_n(t,x_-) < v_n(t,x_+),
\end{cases}
\\
-\Delta\mathtt{Q}_n^{k,i}(t) &=
\begin{cases}
\dfrac{f\bigl(u_n(t,x_-)\bigr)-f\bigl(u_n(t,x_+)\bigr)}{p^{-1}(\mathtt{w}_\pm-k)}
&\text{if }v_n(t,x_-) < v_n(t,x_+) < k,
\\[10pt]
\dfrac{f\bigl(u_n(t,x_-)\bigr)}{p^{-1}(\mathtt{w}_\pm-k)} - k
&\text{if }v_n(t,x_-) < k \leq v_n(t,x_+),
\\[10pt]
0
&\text{if }k \leq v_n(t,x_-) < v_n(t,x_+).
\end{cases}
\end{align*}
If $k > v_n(t,x_+)$ or $k \leq v_n(t,x_-)$, then it is immediate to see that $\dot{\delta}_n^i(t) \, \Delta\mathtt{E}_n^{k,i}(t) - \Delta\mathtt{Q}_n^{k,i}(t) = 0$.
Furthermore, if $v_n(t,x_-) < k \leq v_n(t,x_+)$, then
\begin{align*}&\
\dot{\delta}_n^i(t) \, \Delta\mathtt{E}_n^{k,i}(t) - \Delta\mathtt{Q}_n^{k,i}(t)
\\=&\
\Lambda\bigl(u_n(t,x_-),u_n(t,x_+)\bigr)
\left[ \dfrac{\rho_n(t,x_-)}{p^{-1}(\mathtt{w}_\pm-k)} - 1 \right]
+ \dfrac{f\bigl(u_n(t,x_-)\bigr)}{p^{-1}(\mathtt{w}_\pm-k)} - k
\\
=&\
\underbrace{\left[ \dfrac{\rho_n(t,x_-)}{p^{-1}(\mathtt{w}_\pm-k)} - 1 \right]}_{> 0}\
\underbrace{\left[ \vphantom{\dfrac{\rho_n(t,x_-)}{p^{-1}(\mathtt{w}_\pm-k)}}
\Lambda\bigl(u_n(t,x_-),u_n(t,x_+)\bigr)
+ \Lambda\Bigl(u_n(t,x_-),\bigl(p^{-1}(\mathtt{w}_\pm-k),k\bigr)\Bigr)
\right]}_{< 0}
\\
\geq&\
-\dfrac{2}{\rho^-} \, p^{-1}(\mathtt{w}_\pm) \, p'\bigl(p^{-1}(\mathtt{w}_\pm)\bigr) \,
\bigl[\rho_n(t,x_-)-\rho_n(t,x_+)\bigr]
\end{align*}
because $\rho_n(t,x_-) > p^{-1}(\mathtt{w}_\pm-k) \geq \rho_n(t,x_+) \geq \rho^-$ and because by the concavity of $\mathfrak{L}_{\mathtt{w}_\pm}(\rho) = (\mathtt{w}_\pm-p(\rho))\,\rho$ we have
\begin{align*}
0 >&\
\Lambda\bigl(u_n(t,x_-),u_n(t,x_+)\bigr) >
\Lambda\Bigl(u_n(t,x_-),\bigl(p^{-1}(\mathtt{w}_\pm-k),k\bigr)\Bigr) 
\\>&\
\mathfrak{L}_{\mathtt{w}_\pm}'\bigl(\rho_n(t,x_-)\bigr) = 
\mathtt{w}_\pm - p\bigl(\rho_n(t,x_-)\bigr) - \rho_n(t,x_-) \, p'\bigl(\rho_n(t,x_-)\bigr)
\\\geq&\
\mathfrak{L}_{\mathtt{w}_\pm}'\bigl(p^{-1}(\mathtt{w}_\pm)\bigr) =
- p^{-1}(\mathtt{w}_\pm) \, p'\bigl(p^{-1}(\mathtt{w}_\pm)\bigr).
\end{align*}
\end{itemize}
\medskip

The above case by case study shows that
\begin{align*}
&\ 
\liminf_{n\to\infty} \int_0^T
\int_{\R} \Bigl[\mathtt{E}^k(u_n) \, \phi_t + \mathtt{Q}^k(u_n) \, \phi_x\Bigr] \, \d x \, \d t
\\=&\
\liminf_{n\to\infty} 
\int_0^T \sum_{i\in\mathsf{RS}_n(t)}
\Bigl[ \dot{\delta}_n^i(t) \, \Delta\mathtt{E}_n^{k,i}(t) - \Delta\mathtt{Q}_n^{k,i}(t) \Bigr] \, \phi\bigl(t,\delta_n^i(t)\bigr) \, \d t
\\\geq&\
-\dfrac{2}{\rho^-} \, \max\limits_{\rho\in[p^{-1}(w^-),R]}\bigl|\rho \, p'(\rho)\bigr| \,
\liminf_{n\to\infty} 
\int_0^T \sum_{i\in\mathsf{RS}_n(t)}
\Bigl[\rho_n\bigl(t,\delta_n^i(t)_-\bigr)-\rho_n\bigl(t,\delta_n^i(t)_+\bigr)\Bigr] \, \phi\bigl(t,\delta_n^i(t)\bigr) \, \d t
\\\geq&\
-\dfrac{2 \, T }{\rho^-} \, \|\phi\|_{\L\infty} \, C_F^o \max\limits_{\rho\in[\rho^-,R]}\bigl|\rho\,p'(\rho)\bigr|
\doteq -M,
\end{align*}
where $\delta_n^i(t) \in \R$, $i\in\mathsf{RS}_n(t)\subset\N$, are the positions of the RSs of $u_n(t,\cdot)$ and $C_F^o$ is defined in \eqref{e:KingCrimson}.

We claim that for any fixed $h>0$, there exists a dense set $\mathcal{K}_h$ of values of $k$ in $[0,V]$ such that
\[
\liminf_{n\to\infty} 
\int_0^T \sum_{i\in\mathsf{RS}_n(t)}
\Bigl[ \dot{\delta}_n^i(t) \, \Delta\mathtt{E}_n^{k,i}(t) - \Delta\mathtt{Q}_n^{k,i}(t) \Bigr] \, \phi\bigl(t,\delta_n^i(t)\bigr) \, \d t
\geq - \dfrac{1}{h}.
\]
To prove it we fix $a$, $b\in[0,V]$ with $a<b$ and show that there exists $k \in (a,b)$ such that the above estimate is satisfied.
Let $l \doteq \left\lceil2(M\,h+1)/(b-a)\right\rceil$ and introduce the set
\[
\mathcal{K}_h \doteq \dfrac{2\,\N+1}{l} \cap (a,b).
\]
Let $\mathcal{E}_n>0$ be the maximal $(v,w)$-distance between two ``consecutive'' points in the grid $\mathcal{G}_n$ having the same $w$-coordinate, namely, with a slight abuse of notations, we let
\[
\mathcal{E}_n \doteq \max_{\substack{(v^i,w),\, (v^{i+1},w) \in \mathcal{G}_n\\ v^i \neq v^{i+1}}} (v^{i+1}-v^i).
\]
Let $\mathfrak{n}_h \in \N$ be sufficiently large so that $\mathcal{E}_{\mathfrak{n}_h} < 2/l$.
Take $n \geq \mathfrak{n}_h$.
We claim that for any $i \in \mathsf{RS}_n(t)$ we have
\[
\mathcal{K}_h \cap \Bigl(v_n\bigl(t,\delta_n^i(t)_-\bigr),v_n\bigl(t,\delta_n^i(t)_+\bigr)\Bigr)
\]
has at most one element.
Indeed, if $\mathcal{K}_h$ has more than one element then for any $i \in \mathsf{RS}_n(t)$ we have
\[
v_n\bigl(t,\delta_n^i(t)_+\bigr)-v_n\bigl(t,\delta_n^i(t)_-\bigr) \leq \mathcal{E}_n < \dfrac{2}{l} = \min_{\substack{k^1,\, k^2 \in \mathcal{K}_h\\ k^1 \neq k^2}} |k^1-k^2|.
\]
As a consequence the sum
\[
\sum_{k\in\mathcal{K}_h}
\Bigl[ \dot{\delta}_n^i(t) \, \Delta\mathtt{E}_n^{k,i}(t) - \Delta\mathtt{Q}_n^{k,i}(t) \Bigr]
\]
has at most one nonzero element; moreover
\[
-m \, \Bigl(\rho_n\bigl(t,\delta_n^i(t)_-\bigr)-\rho_n\bigl(t,\delta_n^i(t)_+\bigr)\Bigr)
\leq 
\sum_{k\in\mathcal{K}_h}
\Bigl[ \dot{\delta}_n^i(t) \, \Delta\mathtt{E}_n^{k,i}(t) - \Delta\mathtt{Q}_n^{k,i}(t) \Bigr],
\]
where
\[
m \doteq
\dfrac{2}{\rho^-} \max\limits_{\rho\in[\rho^-,R]}\bigl|\rho\,p'(\rho)\bigr|
= \dfrac{M}{T\,C^o_F\,\|\phi\|_{\L\infty}}.
\]
Therefore we find
\[
\sum_{i\in\mathsf{RS}_n(t)}\sum_{k\in\mathcal{K}_h}
\Bigl[ \dot{\delta}_n^i(t) \, \Delta\mathtt{E}_n^{k,i}(t) - \Delta\mathtt{Q}_n^{k,i}(t) \Bigr] \geq -m \, C_F^o.
\]
By exchanging the sums, multiplying by the test function and integrating in time we get
\[
\sum_{k\in\mathcal{K}_h}\int_0^T\sum_{i\in\mathsf{RS}_n(t)}
\Bigl[ \dot{\delta}_n^i(t) \, \Delta\mathtt{E}_n^{k,i}(t) - \Delta\mathtt{Q}_n^{k,i}(t) \Bigr] \, \phi\bigl(t,\delta_n^i(t)\bigr) \, \d t \geq -M.
\]
Moreover, by construction we have that $\mathcal{K}_h$ is a non-empty set with a finite number of elements (it has at most $h\,M$ elements), hence
\[
h\,M \max_{k\in\mathcal{K}_h} \left[
\int_0^T\sum_{i\in\mathsf{RS}_n(t)}
\Bigl[ \dot{\delta}_n^i(t) \, \Delta\mathtt{E}_n^{k,i}(t) - \Delta\mathtt{Q}_n^{k,i}(t) \Bigr] \, \phi\bigl(t,\delta_n^i(t)\bigr) \, \d t
\right]
\geq -M.
\]

In conclusion we proved that there exists $k \in \mathcal{K}_h \subseteq (a,b)$ such that the above estimate is satisfied for any $n \geq \mathfrak{n}_h$; therefore, since $\mathcal{K}_h$ has a finite number of elements, we have
\[
\liminf_{n\to\infty} 
\int_0^T \sum_{i\in\mathsf{RS}_n(t)}
\Bigl[ \dot{\delta}_n^i(t) \, \Delta\mathtt{E}_n^{k,i}(t) - \Delta\mathtt{Q}_n^{k,i}(t) \Bigr] \, \phi\bigl(t,\delta_n^i(t)\bigr) \, \d t
\geq - \dfrac{1}{h}.
\]
Since $a$ and $b$ are arbitrary, the above estimate holds true for a dense set of values of $k$ in $[0,V]$.

Actually, the above estimate holds for any $k$ in $[0,V]$ because the term in brackets in the above formula is continuous with respect to $k$. Finally, for the arbitrariness of $h$, we have that
\[
\liminf_{n\to\infty} 
\int_0^T \sum_{i\in\mathsf{RS}_n(t)}
\Bigl[ \dot{\delta}_n^i(t) \, \Delta\mathtt{E}_n^{k,i}(t) - \Delta\mathtt{Q}_n^{k,i}(t) \Bigr] \, \phi\bigl(t,\delta_n^i(t)\bigr) \, \d t
\geq 0
\]
and this concludes the proof of \eqref{e:PorcupineTree}.

\item[\ref{CS4}] 
We prove now that \eqref{eq:const} holds for a.e.~$t >0$, namely
\[f\bigl(u(t,0_\pm)\bigr) \le F \qquad \text{for a.e.~}t >0.\]
By construction $f(u_n(t,0_\pm)) \le F$ for any $t>0$, namely the approximate solutions satisfy \eqref{eq:const}.
Since weak convergence preserves pointwise inequalities, it is sufficient to prove that $f(u_n(t,0_\pm))$ weakly converges to $f(u(t,0_\pm))$.
If $\phi$ is a smooth test function of time with compact support in $(0,\infty)$ and $\varphi$ is a smooth test function of space with compact support and such that $\varphi(0)=1$, then
\[
\int_0^\infty f\bigl(u_n(t,0_-)\bigr) \, \phi(t) \, \d t =
\int_0^\infty \int_{-\infty}^0 \Bigl[ \rho_n(t,x) \, \dot{\phi}(t) \, \varphi(x) + f\bigl(u_n(t,x)\bigr) \, \phi(t) \,\dot{\varphi}(x) \Bigr] \d x \, \d t.
\]
The right-hand side passes to the limit, yielding the analogous expression with $u_n$ replaced by $u$.
By using again the Green-Gauss formula, one finally finds that
\[
\lim_{n\to\infty}\int_0^\infty f\bigl(u_n(t,0_-)\bigr) \, \phi(t) \, \d t =
\int_0^\infty f\bigl(u(t,0_-)\bigr) \, \phi(t) \, \d t.
\]
As a consequence $f(u_n(t,0_-))$ weakly converges to $f(u(t,0_-))$, hence $f(u(t,0_-)) \le F$ for a.e.~$t>0$.
At last, since we already proved that $u$ satisfies the first Rankine-Hugoniot condition, we have $f(u(t,0_-)) = f(u(t,0_+))$, hence $f(u(t,0_\pm)) \le F$ for a.e.~$t>0$.\qedhere

\end{enumerate}

\end{proof}

\subsection{The density flow through \texorpdfstring{$\pmb{x=0}$}{}}
\label{s:opt}

Let $u$ be the solution of constrained Cauchy problem \eqref{eq:system}, \eqref{eq:initdat}, \eqref{eq:const} constructed in the previous section.
By Propositions~\ref{p:obvious} and~\ref{p:JacobCollier} we have that non-classical shocks of $u$ can occur only at the constraint location $x=0$, and in this case the (density) flow at $x=0$ does not exceed the maximal flow $F$ allowed by the constraint.

In the case of a constrained Riemann problem \eqref{eq:system}, \eqref{eq:const}, \eqref{eq:Rdata}, we know that $u$ coincides with $(t,x)\mapsto\mathcal{R}_F[u_\ell,u_r](x/t)$, moreover if $(u_\ell,u_r) \in \mathcal{D}_2$ then the flow of the non-classical shock of $u$ coincides with $F$.
In the next proposition we show that also for a general constrained Cauchy problem the flow of the non-classical shocks of $u$ coincides with $F$ if the traces at $x=0$ of the approximate solutions $(u_n)_n$ satisfy a technical condition.

\begin{proposition}\label{p:TheWineryDogs}
Let $u^o \in \L1\cap\BV(\R;\Omega)$, $F\in [0,f_{\rm c}^+]$ satisfy \ref{H.1} or \ref{H.2} and $u$ be a limit of the approximate solutions $(u_n)_n$ constructed in Section~\ref{s:approxsolR}.
Assume that the traces at $x=0$ of $(u_n)_n$ and $u$ satisfy \eqref{e:Opeth}, that is  for any $k \in [0,V]$ and $\phi \in \Cc\infty((0,\infty)\times \R;\R)$ such that $\phi \geq0$
\[
\lim_{n\to\infty} \int_0^T \mathtt{N}^k_F\bigl(u_n(t,0_-)\bigr) \, \phi(t,0) \, \d t = 
\int_0^T \mathtt{N}^k_F\bigl(u(t,0_-)\bigr) \, \phi(t,0) \, \d t,
\]
with
\[
\mathtt{N}^k_F(u) \doteq \begin{cases}
f(u) \left[\dfrac{k}{F}-\dfrac{1}{p^{-1}\bigl(\mathtt{W}(u)-k\bigr)}\right]_+ & \hbox{if } F \neq 0,\\
k & \hbox{if } F = 0.
\end{cases}
\]
If at time $t_0>0$ the limit $u$ has a non-classical discontinuity, then $f(u(t_0,0_\pm))=F$.
\end{proposition}

\begin{proof}
We first prove that for any $k \in [0,V]$ and $\phi \in \Cc\infty((0,\infty)\times \R;\R)$ such that $\phi \geq0$ we have
\begin{equation}\label{e:Deftones2}
\int_0^\infty \left[
\int_{\R} \Bigl[\mathtt{E}^k(u) \, \phi_t + \mathtt{Q}^k(u) \, \phi_x\Bigr] \, \d x
+ \mathtt{N}^k_F\bigl(u(t,0_-)\bigr) \, \phi(t,0) \right] \d t \geq 0.
\end{equation}
Notice that \eqref{e:Deftones2} differs from \eqref{e:Deftones} not only for an extra term involving $\mathtt{N}^k_F(u(t,0_+))$, but also because here we do not require that $\phi(\cdot,0)\equiv0$.

Choose $T>0$ such that $\phi(t,x) = 0$ whenever $t\geq T$.
By \eqref{e:Opeth}, the a.e.~convergence of $u_n$ to $u$ and the uniform continuity of $\mathtt{E}^k$ and $\mathtt{Q}^k$, it is sufficient to prove that
\begin{equation}\label{e:PorcupineTreeb}
\liminf_{n\to\infty} \int_0^T \left[
\int_{\R} \Bigl[ \mathtt{E}^k(u_n) \, \phi_t + \mathtt{Q}^k(u_n) \, \phi_x \Bigr] \, \d x
+ \mathtt{N}^k_F\bigl(u_n(t,0_-)\bigr) \, \phi(t,0)
\right] \d t \geq 0.
\end{equation}
As already observed in the proof of Proposition~\ref{p:JacobCollier}, by the Green-Gauss formula the double integral above can be written as
\[
\int_0^T \sum_{i\in\mathsf{D}(t)}
\Bigl[ \dot{\delta}_n^i(t) \, \Delta\mathtt{E}_n^{k,i}(t) - \Delta\mathtt{Q}_n^{k,i}(t) \Bigr] \, \phi\bigl(t,\delta_n^i(t)\bigr) \, \d t ,
\]
where
\begin{align*}
\Delta\mathtt{E}_n^{k,i}(t) &\doteq 
\mathtt{E}^k\Bigl(u_n\bigl(t,\delta_n^i(t)_+\bigr)\Bigr)
-\mathtt{E}^k\Bigl(u_n\bigl(t,\delta_n^i(t)_-\bigr)\Bigr),&
\Delta\mathtt{Q}_n^{k,i}(t) &\doteq 
\mathtt{Q}^k\Bigl(u_n\bigl(t,\delta_n^i(t)_+\bigr)\Bigr)
-\mathtt{Q}^k\Bigl(u_n\bigl(t,\delta_n^i(t)_-\bigr)\Bigr).
\end{align*}
To estimate the above integral we can proceed as in the proof of Proposition~\ref{p:JacobCollier}, with the exception that here the $i$th discontinuity could also be a NS.
\begin{figure}[!ht]
\begin{center}
\resizebox{\textwidth}{!}{
\def\ratio{1}
\def\pic{50mm}
\begin{tikzpicture}[every node/.style={anchor=south west,inner sep=0pt},x=1mm/\ratio, y=1mm/\ratio]
\node at (4,4) {\includegraphics[height=\pic]{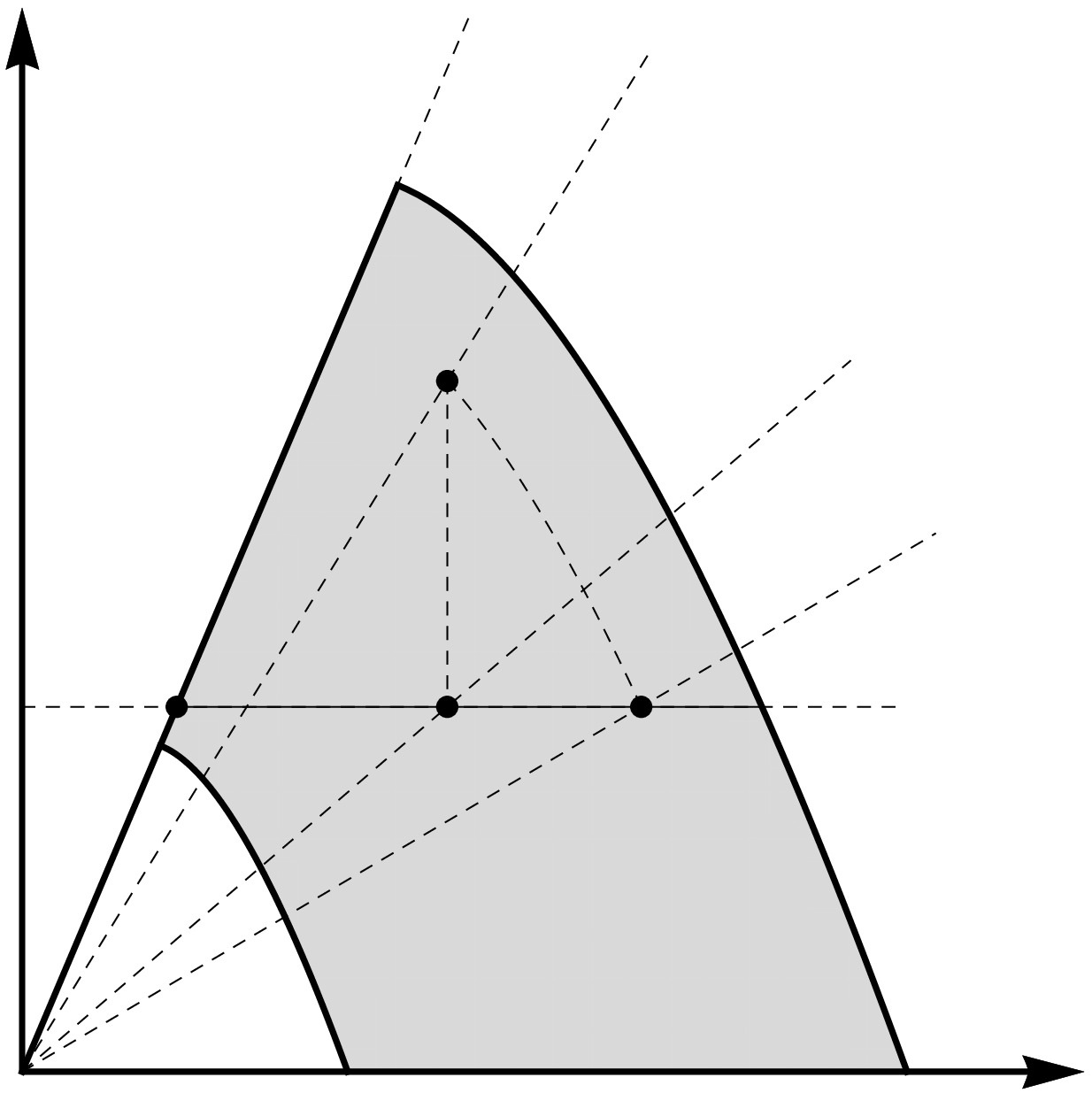}};
\node at (52,0) {$\rho$};
\node at (0,50) {$f$};
\node at (0,20.5) {$F$};
\node at (13,51) {$v_0^+=V$};
\node at (44,37) {$v_{0,F}^-$};
\node at (48,29) {$v_0^-$};
\node at (35,51) {$k$};
\end{tikzpicture}
\quad
\begin{tikzpicture}[every node/.style={anchor=south west,inner sep=0pt},x=1mm/\ratio, y=1mm/\ratio]
\node at (4,4) {\includegraphics[height=\pic]{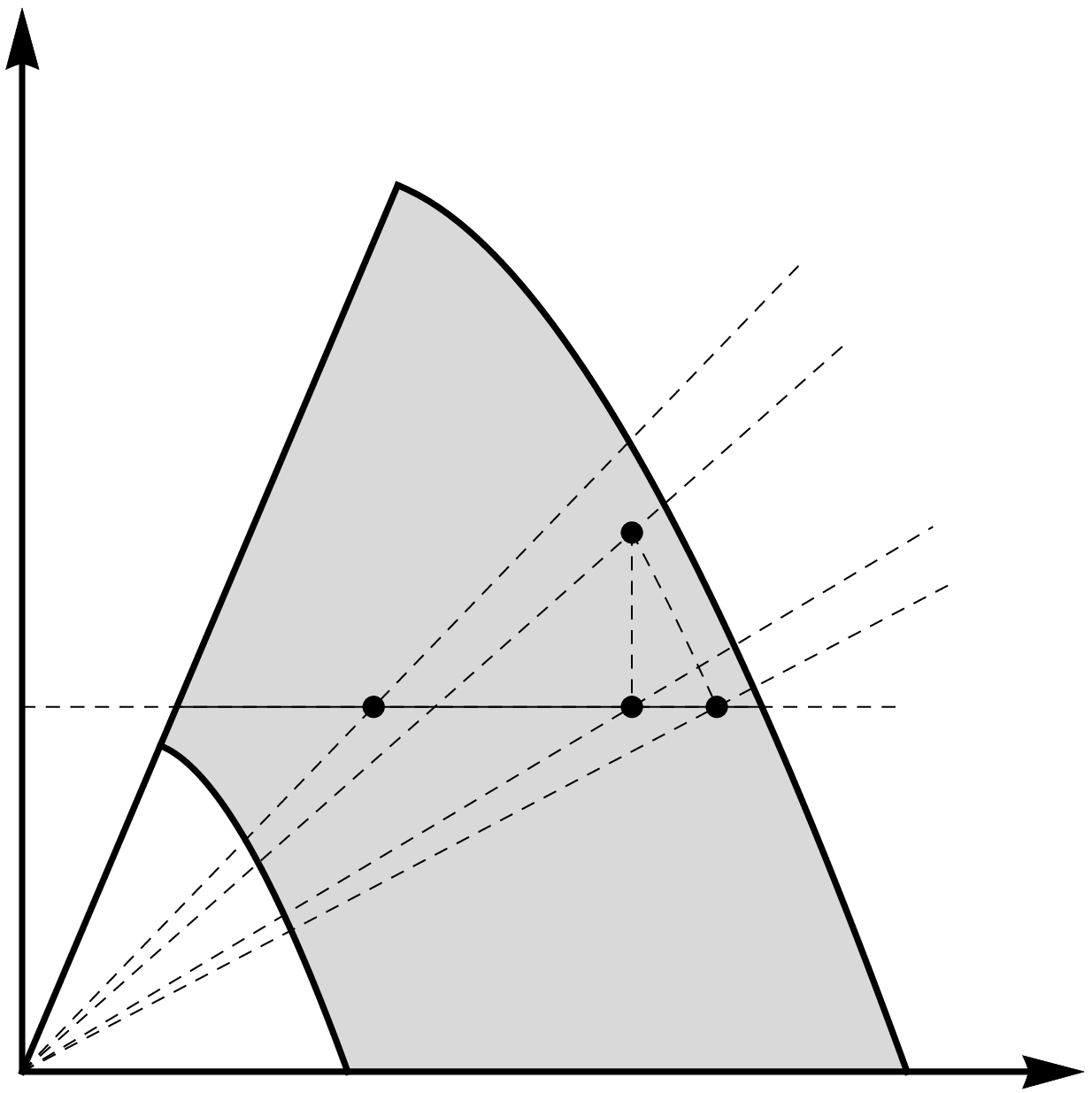}};
\node at (52,0) {$\rho$};
\node at (0,50) {$f$};
\node at (0,20.5) {$F$};
\node at (41,42) {$v_0^+$};
\node at (44,37) {$k$};
\node at (48,25) {$v_0^-$};
\node at (47,29) {$v_{0,F}^-$};
\end{tikzpicture}
\quad
\begin{tikzpicture}[every node/.style={anchor=south west,inner sep=0pt},x=1mm/\ratio, y=1mm/\ratio]
\node at (4,4) {\includegraphics[height=\pic]{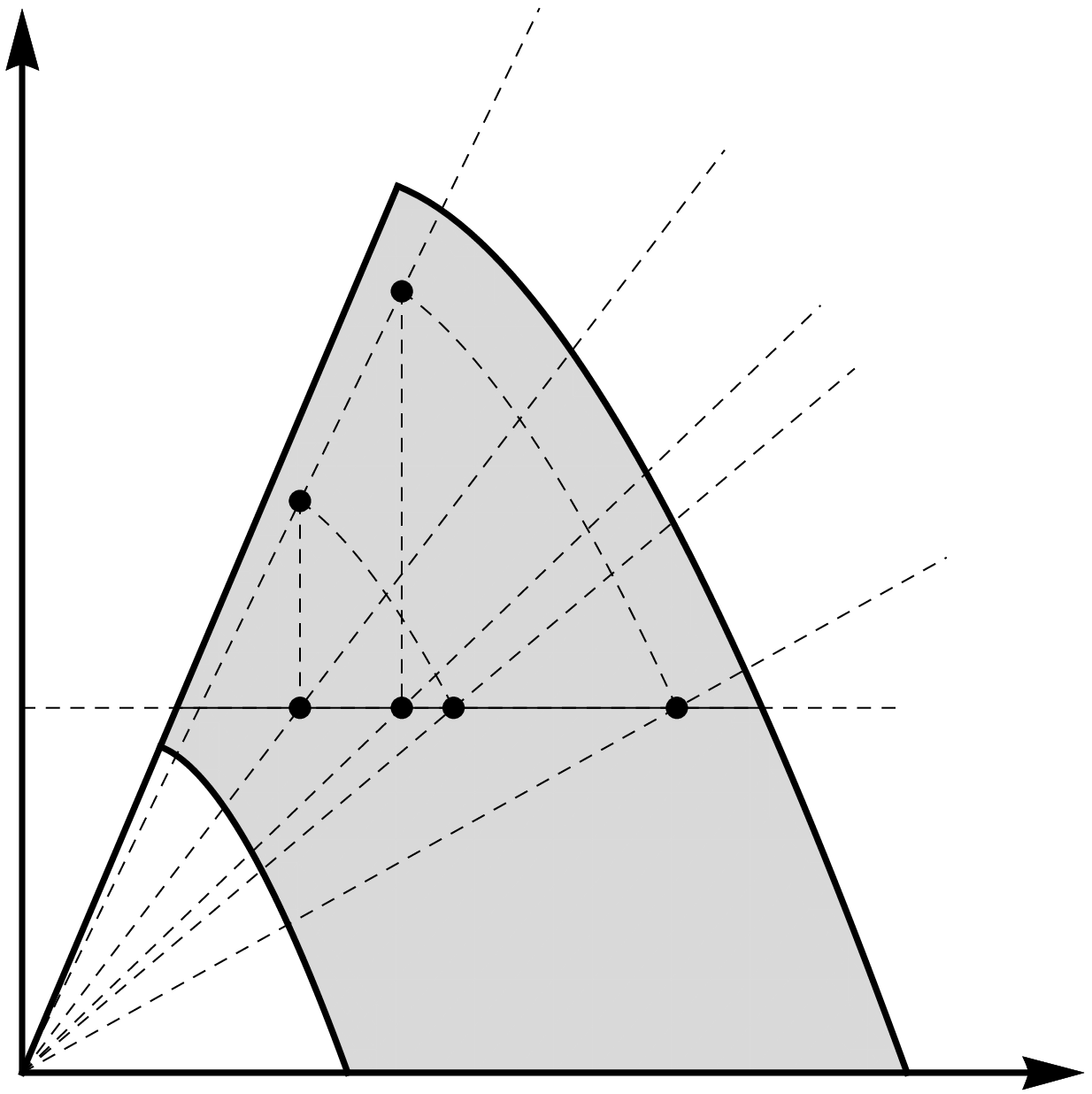}};
\node at (52,0) {$\rho$};
\node at (0,50) {$f$};
\node at (0,20.5) {$F$};
\node at (41,40) {$v_{0,F}^-$};
\node at (44,35) {$v_0^+$};
\node at (48,27) {$v_0^-$};
\node at (37,47) {$v_{0,F}^+$};
\node at (30,51) {$k$};
\end{tikzpicture}
}
\end{center}
\caption{Above $F \in(f_{\rm c}^-,f_{\rm c}^+)$, $v_0^\pm \doteq v_n(t,0_\pm)$ and $v_{0,F}^\pm \doteq F/p^{-1}(\mathtt{W}(u_n(t,0_\pm))-k)$.
With the first two pictures we show that if $v_0^- < k < v_0^+$, then $v_{0,F}^- < k$.
In the last picture we consider the case $v_0^- < v_0^+ < k$ and show that $v_{0,F}^- < v_{0,F}^+ < k$.}
\label{f:DevinTownsendp}
\end{figure}
\begin{figure}[!ht]
\begin{center}
\resizebox{\textwidth}{!}{
\def\ratio{1}
\def\pic{50mm}
\begin{tikzpicture}[every node/.style={anchor=south west,inner sep=0pt},x=1mm/\ratio, y=1mm/\ratio]
\node at (4,4) {\includegraphics[height=\pic]{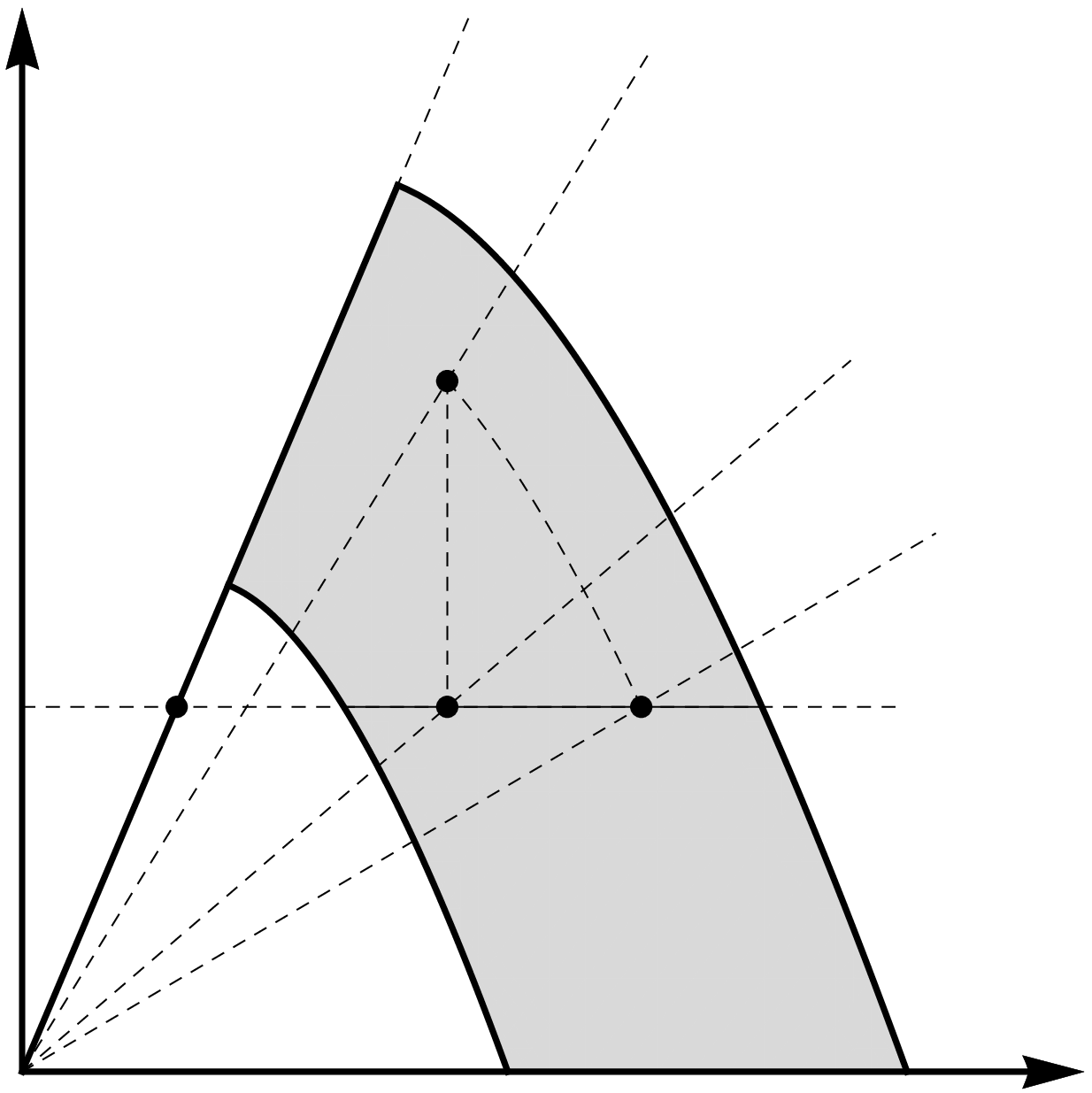}};
\node at (52,0) {$\rho$};
\node at (0,50) {$f$};
\node at (0,20.5) {$F$};
\node at (13,51) {$v_0^+=V$};
\node at (44,37) {$v_{0,F}^-$};
\node at (48,29) {$v_0^-$};
\node at (35,51) {$k$};
\end{tikzpicture}
\quad
\begin{tikzpicture}[every node/.style={anchor=south west,inner sep=0pt},x=1mm/\ratio, y=1mm/\ratio]
\node at (4,4) {\includegraphics[height=\pic]{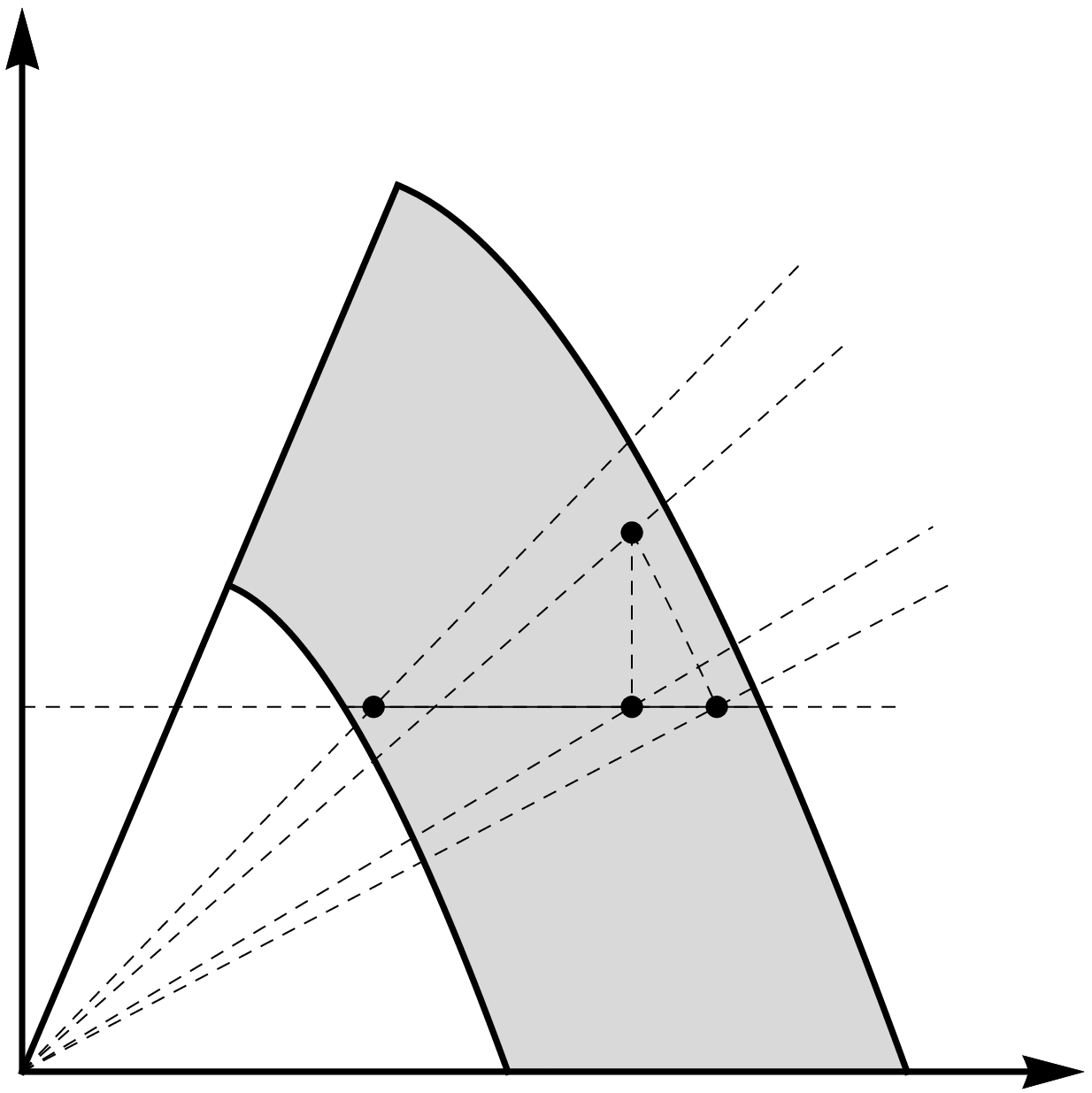}};
\node at (52,0) {$\rho$};
\node at (0,50) {$f$};
\node at (0,20.5) {$F$};
\node at (41,42) {$v_0^+$};
\node at (44,37) {$k$};
\node at (48,25) {$v_0^-$};
\node at (47,29) {$v_{0,F}^-$};
\end{tikzpicture}
\quad
\begin{tikzpicture}[every node/.style={anchor=south west,inner sep=0pt},x=1mm/\ratio, y=1mm/\ratio]
\node at (4,4) {\includegraphics[height=\pic]{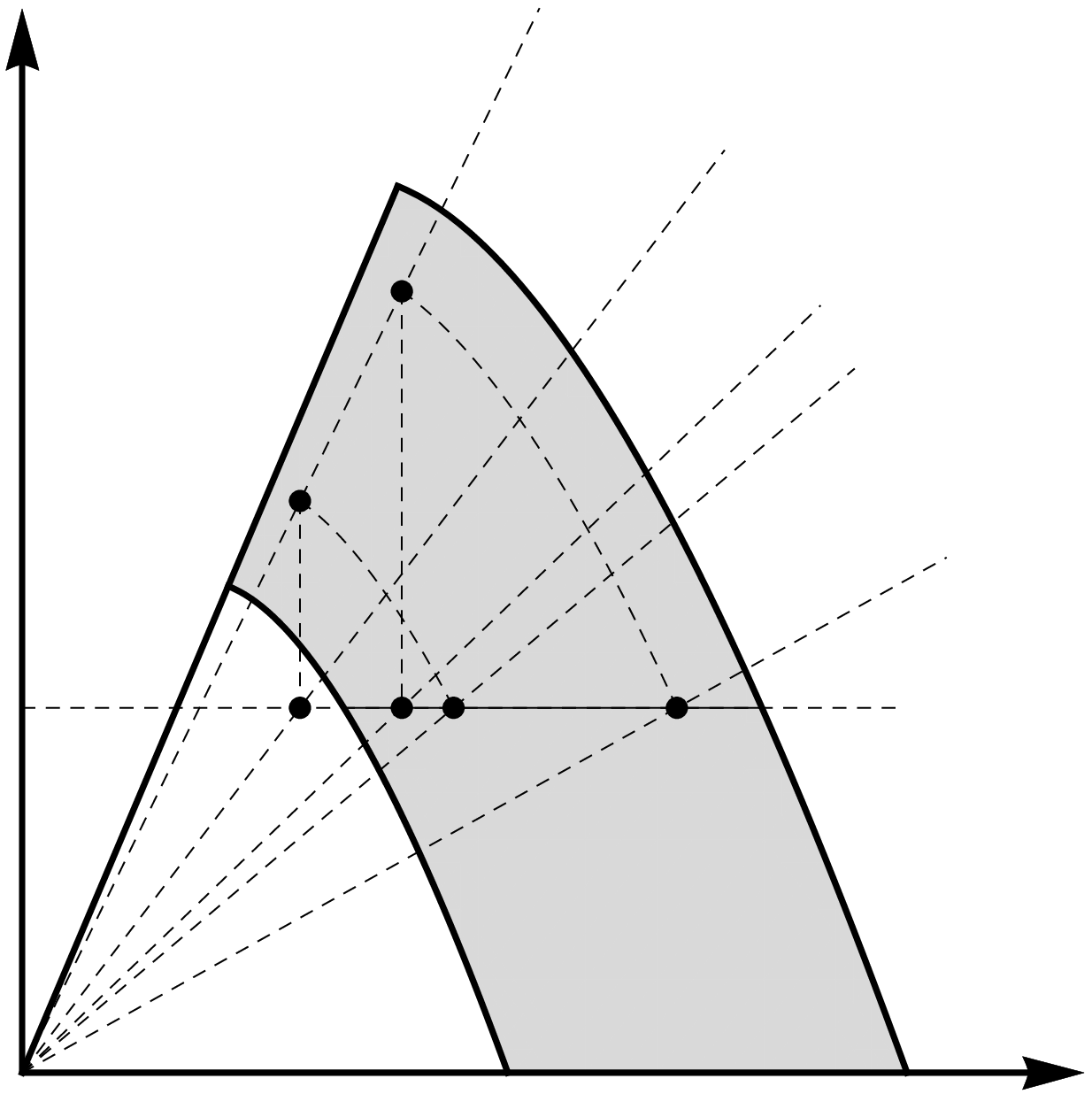}};
\node at (52,0) {$\rho$};
\node at (0,50) {$f$};
\node at (0,20.5) {$F$};
\node at (41,40) {$v_{0,F}^-$};
\node at (44,35) {$v_0^+$};
\node at (48,27) {$v_0^-$};
\node at (37,47) {$v_{0,F}^+$};
\node at (30,51) {$k$};
\end{tikzpicture}
}
\end{center}
\caption{Above $F \in(0,f_{\rm c}^-)$, $v_0^\pm \doteq v_n(t,0_\pm)$ and $v_{0,F}^\pm \doteq F/p^{-1}(\mathtt{W}(u_n(t,0_\pm))-k)$.
With the first two pictures we show that if $v_0^- < k < v_0^+$, then $v_{0,F}^- < k$.
In the last picture we consider the case $v_0^- < v_0^+ < k$ and show that $v_{0,F}^- < v_{0,F}^+ < k$.}
\label{f:DevinTownsend}
\end{figure}
In this case, that is, if the $i$th discontinuity is a NS, then
\begin{align*}
\delta_n^i(t)&=0,&
f\bigl(u_n(t,0_\pm)\bigr) &= F,&
v_F^- &\leq v_n(t,0_-) < v_n(t,0_+),
\\
\dot{\delta}_n^i(t) &= 0,&&&
\mathtt{w}\bigl(u_n(t,0_-)\bigr) &= \mathtt{W}\bigl(u_n(t,0_-)\bigr) \geq \mathtt{W}\bigl(u_n(t,0_+)\bigr),
\end{align*}
hence
\begin{align*}
-\Delta\mathtt{Q}_n^{k,i}(t) &=
\begin{cases}
\dfrac{F}{p^{-1}\Bigl(\mathtt{w}\bigl(u_n(t,0_-)\bigr)-k\Bigr)}
-\dfrac{F}{p^{-1}\Bigl(\mathtt{W}\bigl(u_n(t,0_+)\bigr)-k\Bigr)}
&\text{if }v_n(t,0_-) < v_n(t,0_+) < k,
\\
\dfrac{F}{p^{-1}\Bigl(\mathtt{w}\bigl(u_n(t,0_-)\bigr)-k\Bigr)} - k
&\text{if }v_n(t,0_-) < k \leq v_n(t,0_+),
\\
0
&\text{if }k \leq v_n(t,0_-) < v_n(t,0_+),
\end{cases}
\\
\mathtt{N}^k_F\bigl(u_n(t,0_-)\bigr) &= \begin{cases}
\left[k-\dfrac{F}{p^{-1}\Bigl(\mathtt{W}\bigl(u_n(t,0_-)\bigr)-k\Bigr)}\right]_+ & \hbox{if } F \neq 0,\\
k & \hbox{if } F = 0.
\end{cases}
\end{align*}
Notice that if $F=0$, then $u_n(t,0_+) = (0,V)$ and $u_n(t,0_-) \in [p^{-1}(w^-),R]\times\{0\}$.
We observe, see~\figurename s~\ref{f:DevinTownsendp} and~\ref{f:DevinTownsend}, that $-\Delta\mathtt{Q}_n^{k,i}(t) < 0$ and that $-\Delta\mathtt{Q}_n^{k,i}(t) + \mathtt{N}^k_F(u_n(t,0_-)) \geq 0$ and therefore
\[
\Bigl[ \dot{\delta}_n^i(t) \, \Delta\mathtt{E}_n^{k,i}(t) - \Delta\mathtt{Q}_n^{k,i}(t) \Bigr] \, \phi\bigl(t,\delta_n^i(t)\bigr) + \mathtt{N}^k_F\bigl(u(t,0_-)\bigr) \, \phi(t,0) = 
\Bigl[ - \Delta\mathtt{Q}_n^{k,i}(t) + \mathtt{N}^k_F\bigl(u_n(t,0_-)\bigr) \Bigr] \phi(t,0) \geq 0.
\]
Thus, by proceeding as in the proof of Proposition~\ref{p:JacobCollier} it is easy to see that \eqref{e:PorcupineTreeb} holds true.
Let us just underline that beside the NSs, the only possible stationary discontinuities at $x=0$ are PTs and CDs, however in both of these cases we have $f(u_n(t,0_-)) = 0$ and therefore $\mathtt{N}^k_F(u_n(t,0_-)) = 0$.

We can now prove that if $u$ has a non-classical discontinuity then $f(u(t,0_\pm)) = F$.
This is of course obvious if $F =0$, due to~\ref{CS4} and the fact that $f(u)\geq 0$.
We can therefore assume that $F > 0$ and that $x\mapsto u(t_0,x)$ has a (stationary) non-classical shock $(u_\ell,u_r)$, with $v_\ell < v_r$ and $f(u_\ell) = f(u_r) \doteq f \le F$.
We want to prove that $f=F$.
Consider the test function
\begin{align*}
\phi(t,x) &\doteq \left[\int_{|x|-\varepsilon}^{\infty} \varphi_\varepsilon(z) \, {\d}z \right] \left[\vphantom{\int_{|x|-\varepsilon}^{\infty}} \int_{t - t_0 + \varepsilon}^{t - t_0 +2 \varepsilon} \varphi_\varepsilon(z) \, {\d}z \right] ,
\end{align*}
where $\varphi_\varepsilon$ is a smooth approximation of the Dirac mass centred at $0_+$, $\delta^D_{0_+}$, namely
\[
\varphi_\varepsilon \in \Cc\infty(\R; \R_+),
~\varepsilon >0,
~\supp(\delta_{\varepsilon}) \subseteq[0,\varepsilon],
~\|\varphi_\varepsilon\|_{\L1(\R;\R)} = 1,
~\varphi_\varepsilon \to \delta^D_{0_+}.
\]
Observe that as $\varepsilon$ goes to zero
\begin{align*}
&\phi(t_0,x) \equiv0 \to 0 ,\\
&\phi(t,0) =\int_{t - t_0 + \varepsilon}^{t - t_0 +2 \varepsilon} \varphi_\varepsilon(z) \, {\d}z \to \delta^D_{t_{0-}}(t) ,\\
& \phi_t(t,x) = \left[\int_{|x|-\varepsilon}^{\infty} \varphi_\varepsilon(z) \, {\d}z \right] \left[\vphantom{\int_{|x|-\varepsilon}^{\infty}} \varphi_\varepsilon(t-t_0+2\varepsilon)-\varphi_\varepsilon(t-t_0+\varepsilon)\right]\to 0 ,\\
& \caratt{\R_\pm}(x) \, \phi_x(t,x) \to \mp \, \delta^D_{0_\pm}(x) \, \delta^D_{t_{0-}}(t).
\end{align*}
Then by \eqref{e:Deftones2} for all $k$ belonging to the interval $(\hat{v}(w_\ell,F), \check{v}(v_r,F))$ we have
\begin{align*}
    &\mathtt{Q}^k(u_\ell) - \mathtt{Q}^k(u_r)
    +f \left[\frac{k}{F} - \frac{1}{p^{-1}\bigl(\mathtt{W}(u_\ell)-k\bigr)} \right]_+
    \\
    =&
    \left[ \frac{f}{p^{-1}\bigl(\mathtt{W}(u_\ell)-k\bigr)} - k \right]
    +f \left[\frac{k}{F} - \frac{1}{p^{-1}\bigl(\mathtt{W}(u_\ell)-k\bigr)} \right]
    =
    \left[\frac{f}{F} - 1\right]k\ge 0.
\end{align*}
Since $f\le F$, the above estimate implies that $f = F$ and this concludes the proof.
\end{proof}

We underline that the entropy condition \eqref{e:Deftones} ``becomes'' \eqref{e:Deftones2} if we do not require that the test function $\phi$ satisfy the condition $\phi(\cdot,0)\equiv0$.
Even if it is not necessary for the proof of Theorem~\ref{t:mainF}, we conclude this section by considering in \eqref{e:AdrianBelew1} a test function $\phi$ which may not satisfy the condition $\phi(\cdot,0)\equiv0$.

\begin{proposition}
Let $u^o \in \L1\cap\BV(\R;\Omega)$, $F\in [0,f_{\rm c}^+]$ satisfy \ref{H.1} or \ref{H.2} and $u$ be a limit of the approximate solutions $(u_n)_n$ constructed in Section~\ref{s:approxsolR}.
If the traces at $x=0$ of $(u_n)_n$ and $u$ satisfy for any $\phi \in \Cc\infty((0,\infty)\times\R; \R)$
\begin{align}
&\
\lim_{n\to\infty} \int_0^T
f\bigl(u_n(t,0_-)\bigr)
\Bigl[\mathtt{W}\bigl(u_n(t,0_-)\bigr) - \mathtt{W}\bigl(u_n(t,0_+)\bigr)\Bigr]_+ \,
\phi(t,0) \, \d t
\nonumber\\=&\
\int_0^T
f\bigl(u(t,0_-)\bigr) \,
\Bigl[\mathtt{W}\bigl(u(t,0_-)\bigr) - \mathtt{W}\bigl(u(t,0_+)\bigr)\Bigr]_+ \,
\phi(t,0) \, \d t
\label{e:Katatonia}
\end{align}
then $u$ satisfies the following integral condition for any $\phi \in \Cc\infty((0,\infty)\times\R; \R)$
\[
\int_0^\infty \left[
\int_{\R} 
\bigl[ \rho \, \phi_t
+ f(u)\, \phi_x \bigr] \,
\mathtt{W}(u) \,
\d x 
-
f\bigl(u(t,0_-)\bigr) \,
\Bigl[\mathtt{W}\bigl(u(t,0_-)\bigr) - \mathtt{W}\bigl(u(t,0_+)\bigr)\Bigr]_+ \,
\phi(t,0)
\right] \d t
=0.
\]
\end{proposition}

\begin{proof}
Choose $T>0$ such that $\phi(t,x) = 0$ whenever $t \geq T$.
By \eqref{e:Katatonia}, since $u_n$ is uniformly bounded and $f$ is uniformly continuous on bounded sets, it is sufficient to prove that
\begin{equation}\label{e:BlindMelonb}
\int_0^T \left[
\int_{\R} 
\left[ \rho_n \, \phi_t
+ f(u_n)\, \phi_x \right] \,
\mathtt{W}(u_n) \,
\d x 
-
f\bigl(u_n(t,0_-)\bigr) \,
\bigl[\mathtt{W}\bigl(u_n(t,0_-)\bigr) - \mathtt{W}\bigl(u_n(t,0_+)\bigr)\bigr]_+ \,
\phi(t,0)
\right] \d t
\to0.
\end{equation}
By the Green-Gauss formula the double integrals above can be written as
\[
\int_0^T \sum_{i\in\mathsf{D}(t)}
\bigl[ \dot{\delta}_n^i(t) \, \Delta Y_n^i(t) - \Delta Q_n^i(t) \bigr] \, \phi\bigl(t,\delta_n^i(t)\bigr) \, \d t ,
\]
where
\begin{align*}
\Delta Y_n^i(t) &\doteq 
\rho_n\bigl(t,\delta_n^i(t)_+\bigr) \,
\mathtt{W}\Bigl(u_n\bigl(t,\delta_n^i(t)_+\bigr)\Bigr)
- 
\rho_n\bigl(t,\delta_n^i(t)_-\bigr) \,
\mathtt{W}\Bigl(u_n\bigl(t,\delta_n^i(t)_-\bigr)\Bigr),
\\
\Delta Q_n^i(t) &\doteq 
f\Bigl(u_n\bigl(t,\delta_n^i(t)_+\bigr)\Bigr) \,
\mathtt{W}\Bigl(u_n\bigl(t,\delta_n^i(t)_+\bigr)\Bigr)
- 
f\Bigl(u_n\bigl(t,\delta_n^i(t)_-\bigr)\Bigr) \,
\mathtt{W}\Bigl(u_n\bigl(t,\delta_n^i(t)_-\bigr)\Bigr).
\end{align*}
If $u_n(t,\cdot)$ does not have a non-classical shock at $\delta_n^i(t)$, then by the Rankine-Hugoniot conditions
\[
\dot{\delta}_n^i(t) \, \Delta Y_n^i(t) - \Delta Q_n^i(t)=0;
\]
moreover, if $\delta_n^i(t) = 0$ and $u_n(t,\cdot)$ has a stationary discontinuity at $x=0$, namely a phase transition or a contact discontinuity, then $v_n(t,0_+) = v_n(t,0) = 0$ and therefore $\sign(v_n(t,0_+)) = 0$.
\\
On the other hand, if $\delta_n^i(t) = 0$ and $u_n(t,\cdot)$ has a stationary non-classical shock at $x = 0$, then
\begin{align*}
\dot{\delta}_n^i(t) &=0,&
f\bigl(u_n(t,0_\pm)\bigr) &= F,&
\mathtt{W}\bigl(u_n(t,0_-)\bigr) &\geq \mathtt{W}\bigl(u_n(t,0_+)\bigr),
\end{align*}
and therefore
\begin{align*}
\dot{\delta}_n^i(t) \, \Delta Y_n^i(t) - \Delta Q_n^i(t) 
&= 
-F \,
\Bigl[ \mathtt{W}\bigl(u_n(t,0_+)\bigr) - \mathtt{W}\bigl(u_n(t,0_-)\bigr) \Bigr]
\\&=
f\bigl(u_n(t,0_-)\bigr) \,
\Bigl[\mathtt{W}\bigl(u_n(t,0_-)\bigr) - \mathtt{W}\bigl(u_n(t,0_+)\bigr)\Bigr]_+.
\end{align*}
As a consequence \eqref{e:BlindMelonb} is trivial.
\end{proof}

\section*{Acknowledgements}

M.~D.~Rosini acknowledges the support of Universit\`a degli Studi di Ferrara Project 2017 \lq\lq FIR: Modelli macroscopici per il traffico veicolare o pedonale\rq\rq. 
N.~Dymski acknowledges the support of the French Government Scholarship (BGF) program for joint PhD thesis of the French Embassy in Poland.

{\footnotesize\bibliographystyle{acm}
\bibliography{biblio}}

\hfill

\Addresses

\end{document}